\documentclass[11pt]{amsart}
\usepackage[mathscr]{euscript}
\DeclareFontFamily{OT1}{rsfs}{}
\DeclareFontShape{OT1}{rsfs}{n}{it}{<-> rsfs10}{}
\DeclareMathAlphabet{\curly}{OT1}{rsfs}{n}{it}

\makeatletter
\newcommand{\eqnum}{\refstepcounter{equation}\textup{\tagform@{\theequation}}}
\makeatother

\newcommand\beq[1]{\begin{equation}\label{#1}}
\newcommand\eeq{\end{equation}}
\newcommand\beqa{\begin{eqnarray*}}
\newcommand\eeqa{\end{eqnarray*}}

\title[Birational geometry for d-critical loci]{Birational geometry for
d-critical loci and wall-crossing in Calabi-Yau 3-folds}
\date{}
\author{Yukinobu Toda}

\usepackage{amscd}
\usepackage{amsmath}
\usepackage{amssymb}
\usepackage{amsthm}
\usepackage{float}
\usepackage[dvips]{graphicx}
\usepackage{xypic}

\usepackage{fancybox}

\usepackage{array}
\usepackage{amscd}
\usepackage[all]{xy}

\DeclareFontFamily{U}{rsfs}{%
\skewchar\font127}
\DeclareFontShape{U}{rsfs}{m}{n}{%
<-6>rsfs5<6-8.5>rsfs7<8.5->rsfs10}{}
\DeclareSymbolFont{rsfs}{U}{rsfs}{m}{n}
\DeclareSymbolFontAlphabet
{\mathrsfs}{rsfs}
\DeclareRobustCommand*\rsfs{%
\@fontswitch\relax\mathrsfs}

\theoremstyle{plain}
\newtheorem{thm}{Theorem}[section]
\newtheorem{prop}[thm]{Proposition}
\newtheorem{lem}[thm]{Lemma}
\newtheorem{sublem}[thm]{Sublemma}

\newtheorem{cor}[thm]{Corollary}

\newtheorem{prop-defi}[thm]{Proposition-Definition}
\newtheorem{defi-prop}[thm]{Definition-Proposition}
\newtheorem{thm-defi}[thm]{Theorem-Definition}
\newtheorem{lem-defi}[thm]{Lemma-Definition}

\newtheorem{question}[thm]{Question}
\newtheorem{assum}[thm]{Assumption}

\theoremstyle{definition}
\newtheorem{defi}[thm]{Definition}
\newtheorem{exam}[thm]{Example}

\theoremstyle{remark}
\newtheorem{rmk}[thm]{Remark}

\newcommand{\sslash}{/\!\!/}

\newcommand{\aA}{\mathcal{A}}

\newcommand{\dD}{\mathcal{D}}
\newcommand{\eE}{\mathcal{E}}
\newcommand{\fF}{\mathcal{F}}
\newcommand{\gG}{\mathcal{G}}
\newcommand{\hH}{\mathcal{H}}

\newcommand{\lL}{\mathcal{L}}
\newcommand{\mM}{\mathcal{M}}

\newcommand{\oO}{\mathcal{O}}

\newcommand{\qQ}{\mathcal{Q}}

\newcommand{\sS}{\mathcal{S}}

\newcommand{\vV}{\mathcal{V}}
\newcommand{\wW}{\mathcal{W}}
\newcommand{\xX}{\mathcal{X}}
\newcommand{\yY}{\mathcal{Y}}

\newcommand{\Supp}{\mathop{\rm Supp}\nolimits}
\newcommand{\Hom}{\mathop{\rm Hom}\nolimits}

\newcommand{\dR}{\mathbf{R}}

\newcommand{\Pic}{\mathop{\rm Pic}\nolimits}

\newcommand{\id}{\textrm{id}}

\newcommand{\ch}{\mathop{\rm ch}\nolimits}

\newcommand{\td}{\mathop{\rm td}\nolimits}
\newcommand{\Ext}{\mathop{\rm Ext}\nolimits}
\newcommand{\Spec}{\mathop{\rm Spec}\nolimits}
\newcommand{\rank}{\mathop{\rm rank}\nolimits}
\newcommand{\Coh}{\mathop{\rm Coh}\nolimits}
\newcommand{\QCoh}{\mathop{\rm QCoh}\nolimits}

\newcommand{\cneq}{\mathrel{\raise.095ex\hbox{:}\mkern-4.2mu=}}
\newcommand{\eqcn}{\mathrel{=\mkern-4.5mu\raise.095ex\hbox{:}}}

\newcommand{\ext}{\mathop{\rm ext}\nolimits}

\newcommand{\Aut}{\mathop{\rm Aut}\nolimits}

\newcommand{\codim}{\mathop{\rm codim}\nolimits}
\newcommand{\Ex}{\mathop{\rm Ex}\nolimits}

\newcommand{\Stab}{\mathop{\rm Stab}\nolimits}

\newcommand{\modu}{\mathop{\rm mod}\nolimits}

\newcommand{\Imm}{\mathop{\rm Im}\nolimits}

\newcommand{\GL}{\mathop{\rm GL}\nolimits}

\newcommand{\tr}{\mathop{\rm tr}\nolimits}

\newcommand{\cl}{\mathop{\rm cl}\nolimits}

\newcommand{\lkakko}{[\![}
\newcommand{\rkakko}{]\!]}

\makeatletter
 \renewcommand{\theequation}{%
   \thesection.\arabic{equation}}
  \@addtoreset{equation}{section}
\makeatother

\begin{document}

\begin{abstract}
The notion of d-critical loci was introduced by Joyce
in order to give classical shadows of $(-1)$-shifted 
symplectic derived schemes. 
In this paper, we discuss 
birational geometry for 
d-critical loci, 
by introducing notions such as  
`d-critical flips', `d-critical flops', etc. 
They are not birational maps of the underlying 
spaces, but rather
should be understood as 
virtual birational maps.  

We show that several wall-crossing phenomena of 
moduli spaces of stable objects on Calabi-Yau 3-folds
are described in terms of d-critical 
birational geometry. 
Among them, we show that
wall-crossing diagrams of 
Pandharipande-Thomas (PT)
stable pair moduli spaces,  
which 
are relevant in showing the rationality of PT 
generating series, 
 form a
d-critical minimal model program. 
\end{abstract}

\maketitle

\setcounter{tocdepth}{1}
\tableofcontents

\section{Introduction}
\subsection{Background and motivation}
Let $X$ be a smooth projective variety over $\mathbb{C}$. 
For a given numerical class $v \in H^{2\ast}(X, \mathbb{Q})$
and a stability condition $\sigma$
on the derived category of coherent sheaves on $X$
(e.g. Bridgeland stability condition~\cite{Brs1}, 
weak stability condition~\cite{Tcurve1}), 
we denote by 
$M_{\sigma}(v)$ the coarse moduli space\footnote{The
existence of $M_{\sigma}(v)$ is not obvious in general, and we discuss 
assuming that it exists. It is announced in~\cite{AHLH} that 
$M_{\sigma}(v)$ exists if $\sigma$-semistable objects with 
Chern character $v$ are bounded.} of $S$-equivalence classes of 
$\sigma$-semistable objects on $X$ with Chern character $v$. 
The moduli space $M_{\sigma}(v)$ depends on a choice of a
stability condition
$\sigma$.
In general, 
we have \textit{wall-crossing phenomena}, i.e. 
there is a wall-chamber structure on the space of 
stability conditions such that $M_{\sigma}(v)$
is constant if $\sigma$ lies on a chamber but may change 
if $\sigma$ crosses a wall. 
It is an interesting question how the moduli 
spaces $M_{\sigma}(v)$ vary under wall-crossing of $\sigma$.
More precisely, 
suppose that $\sigma$ lies on a wall
and 
$\sigma^{\pm}$ lie on its adjacent chambers.
Then we 
have the following diagram
(called \textit{wall-crossing diagram})
\begin{align}\label{intro:wall}
\xymatrix{
M_{\sigma^{+}}(v) \ar[rd] & & \ar[ld] M_{\sigma^{-}}(v) \\
&  M_{\sigma}(v).  &
}
\end{align}
If $M_{\sigma^{\pm}}(v)$ are smooth (or singular with mild singularities)
and 
birational, then it makes sense to ask 
whether the diagram (\ref{intro:wall}) is a flip or flop
in birational geometry~\cite{KM}. 
Note that if this happens, 
we have the inequality of canonical 
line bundles of $M_{\sigma^{\pm}}(v)$ (see Definition~\ref{defi:ineqK}):
\begin{align}\label{intro:ineq}
M_{\sigma^{+}}(v) \ge_K M_{\sigma^{-}}(v).
\end{align}
Indeed this is true in some cases, 
and birational geometry of the diagram (\ref{intro:wall}) has 
been especially studied when 
$X$ is an algebraic surface~\cite{MR1355920, MR1362648, MW, MR2729275, 
MR3279532, MR3279532, MR3010070, Todbir}. 
In the above articles, birational geometry 
of the diagram (\ref{intro:wall})
has been 
important in understanding birational geometry
of classical moduli spaces (e.g. Hilbert schemes of points),
or applications to enumerative geometry (e.g. Donaldson invariants). 

However there is a limitation of this research direction. 
In general, the moduli spaces $M_{\sigma^{\pm}}(v)$ can have 
worse singularities than
those which appear in birational geometry~\cite{KM}, 
e.g. terminal singularities, canonical singularities, etc.
In fact by Vakil's Murphy's law~\cite{MR2227692}, 
any singularity can appear on such moduli
spaces, so 
they may not be irreducible, may not be reduced, may not be equidimensional, or so on. 
In such bad cases, it is not even clear what are birational 
maps between them, nor what are their canonical line bundles. 
Moreover even if $M_{\sigma^{\pm}}(v)$ are smooth, they are 
not birational in general. 
So it does not make sense to ask 
whether the 
wall-crossing diagram (\ref{intro:wall}) 
is a flip nor a flop, nor satisfy the inequality (\ref{intro:ineq}). 

If we focus on the case that $X$ is a
Calabi-Yau (CY for short) 3-fold, 
still we have the same issue
as above. 
However in this case, we have additional structures on the 
moduli spaces $M_{\sigma^{\pm}}(v)$ (more precisely on their stable parts) 
called \textit{d-critical structures}. 
This notion was introduced by Joyce~\cite{JoyceD}
in order to give 
classical shadows of $(-1)$-shifted symplectic structures
on derived moduli spaces of stable objects on CY 3-folds~\cite{PTVV}. 
In particular, we have the notion of \textit{virtual canonical line bundles}
on such moduli spaces.

The purpose of this paper is to introduce the notions of 
\textit{d-critical flips, d-critical flops}, etc, 
for morphisms from
d-critical loci 
to schemes or analytic spaces.
They are not birational maps 
of the underlying spaces, but 
rather should be understood as `virtual'
birational maps. 
We then show that, despite of the possible 
bad singularities of $M_{\sigma^{\pm}}(v)$,  
several wall-crossing diagrams (\ref{intro:wall}) 
for a CY 3-fold $X$ fit into these notions of d-critical 
birational transformations. In particular they satisfy 
an analogue of the inequality (\ref{intro:ineq}) for virtual 
canonical line bundles. 

\subsection{D-critical birational geometry}
By definition, a \textit{d-critical locus} introduced by Joyce~\cite{JoyceD}
 consists of data
\begin{align*}
(M, s), \ s \in \Gamma(M, \sS_{M}^0)
\end{align*} where 
$M$ is a $\mathbb{C}$-scheme or an analytic space and 
$\sS_M^0$ is a certain sheaf of $\mathbb{C}$-vector spaces
on $M$
(see Definition~\ref{defi:dcrit}). 
The section $s$ is called a \textit{d-critical structure} of $M$. 
Roughly speaking
if $M$ admits a d-critical structure $s$, 
this means that $M$ is locally written as a critical 
locus of some function on a smooth space, and the section 
$s$ remembers how $M$ is locally written as a critical locus. 

Let $(M^{\pm}, s^{\pm})$
be two d-critical loci and consider a diagram
of morphisms of $\mathbb{C}$-schemes or analytic spaces
\begin{align}\label{intro:dcrit}
\xymatrix{
M^{+} \ar[rd] & & M^{-} \ar[ld] \\
& A. &
}
\end{align}
We introduce the notion 
of a \textit{d-critical flip (resp.~d-critical flop)} to be
 a diagram (\ref{intro:dcrit}) satisfying the following: 
for any
$p \in A$, there is a commutative diagram 
\begin{align*}
\xymatrix{
Y^{+} \ar[rd] 
\ar@/_5pt/[rdd]_-{w^{+}}  \ar@{.>}[rr]^-{\phi}& & Y^{-} \ar[ld] 
\ar@/^5pt/[ldd]^-{w^-}\\
& Z \ar[d]_-{g} &  \\
& \mathbb{C} &
}
\end{align*}
where $ \phi \colon Y^{+} \dashrightarrow Y^{-}$ is a flip (resp.~flop)
of smooth varieties (or complex manifolds), 
such that locally near $p\in A$
there exist isomorphisms 
between 
$M^{\pm}$ and $\{dw^{\pm}=0\}$ as d-critical loci
(see~Definition~\ref{defi:dflip} for details).
Other notions such as a d-critical divisorial contraction, a d-critical 
Mori fiber space (MFS for short), and their generalized version will be also 
defined in a similar way. 
We also introduce the inequality of virtual canonical 
line bundles on d-critical loci, as an analogue of (\ref{intro:ineq})
(see Definition~\ref{defi:dcrit:K}): 
\begin{align}\label{intro:ineq:K}
(M^{+},s^{+}) \ge_K (M^{-}, s^{-}).
\end{align}
For example d-critical flips, d-critical flops 
satisfy (\ref{intro:ineq:K}). 

We remark that a diagram (\ref{intro:dcrit})
being a d-critical flip or a d-critical flop 
does not imply anything on birational geometry of 
$M^{\pm}$ themselves, even when $M^{\pm}$ are smooth. 
Indeed there is an example of a d-critical flip
where $M^{\pm}$ are smooth but their dimensions 
are different (see Example~\ref{exam:toric}), so in particular they are 
not birational. 
Therefore we should interpret 
these notions as `virtual' birational 
maps rather than birational maps 
of the underlying spaces $M^{\pm}$.

\subsection{Wall-crossing in Calabi-Yau 3-folds}
Suppose that $X$ is a smooth projective CY 3-fold, 
and consider a wall-crossing diagram (\ref{intro:wall}). 
If $v$ is primitive, 
then 
$M_{\sigma^{\pm}}(v)$ 
consist of $\sigma^{\pm}$-stable 
objects, and
 $M_{\sigma^{\pm}}(v)$ 
admit canonical d-critical structures 
by~\cite{MR3352237}.
Indeed 
$M_{\sigma^{\pm}}(v)$ are classical truncations of derived schemes 
with $(-1)$-shifted symplectic structures~\cite{PTVV}, and 
the derived Darboux theorem~\cite{MR3352237} for 
$(-1)$-shifted symplectic derived schemes 
yield d-critical structures on them. 

Therefore we can ask whether the diagram (\ref{intro:wall}) is a 
d-critical flip, a d-critical flop, or so on. 
We will answer the above question via `analytic neighborhood theorem'
given in Theorem~\ref{thm:Equiver}. 
This theorem describes the diagram (\ref{intro:wall}) 
analytic locally on $M_{\sigma}(v)$ in 
terms of a wall-crossing diagram of moduli spaces of representations 
of a certain quiver with a (formal but convergent) super-potential. 
A similar result was already proved in~\cite{Todstack}
for moduli spaces of semistable sheaves, and 
we will see that 
the argument can be generalized to our setting, 
assuming the 
existence of good moduli spaces of Bridgeland semistable objects.
The latter existence problem is recently announced to be settled 
by Alper-Halpern-Leistner-Heinloth~\cite{AHLH}. 
The analytic neighborhood theorem
 reduces the above question on the diagram (\ref{intro:wall})
to study birational maps of 
moduli spaces of representations of some quivers without 
relations. 

Using the analytic neighborhood theorem, 
we study d-critical birational geometry of  
wall-crossing diagrams (\ref{intro:wall}) which
appeared in the context of enumerative geometry on
CY 3-folds, e.g. Donaldson-Thomas (DT) invariants~\cite{Thom}, 
Pandharipande-Thomas (PT) invariants~\cite{PT} and also
Gopakumar-Vafa (GV) invariants~\cite{MT}. 
The results are summarized below: 
 
\begin{thm}\emph{(Theorem~\ref{thm:dflop:one}, Theorem~\ref{thm:starflip})}\label{thm:intro:main}
\begin{enumerate}
\item 
In the case of wall-crossing of one dimensional stable sheaves, 
the diagram (\ref{intro:wall}) is a d-critical generalized flop. 

\item In the case of wall-crossing of PT stable 
pair moduli spaces, the 
diagram (\ref{intro:wall}) 
is a d-critical generalized flip 
at any point in $\Imm \pi^-$, 
a d-critical MFS
at any point in $M_{\sigma}(v) \setminus \Imm \pi^-$. 
\end{enumerate}
\end{thm}
The wall-crossing diagrams
in the above cases have applications to enumerative geometry. 
The wall-crossing diagrams of one dimensional stable sheaves (i)
are used in~\cite{TodGV} to show 
that GV invariants defined in~\cite{MT} 
are independent of 
Bridgeland stability conditions,
and also invariant under flops. 
The wall-crossing diagrams of stable pair moduli 
spaces (ii) are used in~\cite{BrH, Tolim, Tolim2, Tsurvey} 
(also in~\cite{MR2888981} for local curve case) to show the 
rationality conjecture of the generating series of 
PT invariants. 
In this case, we have wall-crossing 
diagrams 
which relate PT invariants 
and L invariants 
in \textit{loc.cit.}, 
and Theorem~\ref{thm:intro:main} (ii) shows that they
form a d-critical MMP (see Corollary~\ref{cor:zigzag}). 

As a summary, Theorem~\ref{thm:intro:main}
gives an interpretation of wall-crossing
diagrams in CY 3-folds
relevant in enumerative geometry 
in terms of d-critical birational geometry. 
In Appendix~\ref{sec:other}, we also 
discuss some other examples of 
wall-crossing diagrams in CY 3-folds in terms
of d-critical birational geometry, 
DT/PT correspondence, local K3 surfaces
(see~Theorem~\ref{thm:DT/PT}, Theorem~\ref{thm:locK3}). 
They also have applications to 
enumerative geometry~\cite{BrH, Tcurve1, TodK3}. 
\subsection{Speculation toward d-critical D/K equivalence conjecture}
Let $Y^{+} \dashrightarrow Y^{-}$ be a birational map 
between smooth projective varieties 
satisfying the relation $Y^{+} \ge_K Y^{-}$. 
Then by Bondal-Orlov~\cite{B-O2} and Kawamata~\cite{Ka1}, 
it is conjectured that there exists a fully faithful functor
of derived categories of coherent sheaves
\begin{align*}
D^b(Y^{-}) \hookrightarrow D^b(Y^{+})
\end{align*}
which is an equivalence if $Y^{+}=_K Y^{-}$. 
We call the above conjecture as \textit{D/K equivalence conjecture}. 

We expect a similar conjecture may hold 
for d-critical 
loci, or $(-1)$-shifted symplectic 
derived schemes. 
 Namely 
for a d-critical locus $(M, s)$ 
(probably induced by a $(-1)$-shifted symplectic
derived scheme
equipped 
with additional data), 
there may exist a certain 
triangulated category $\dD(M, s)$ such that, 
if the relation (\ref{intro:ineq:K}) holds, we have 
a fully faithful functor
\begin{align}\label{intro:shifted:emb}
\dD(M^{-}, s^{-}) \hookrightarrow 
\dD(M^{+}, s^{+})
\end{align}
which is an equivalence if (\ref{intro:ineq:K}) is an equality. 
The category $\dD(M^{-}, s^{-})$ may be 
constructed as a gluing of 
$\mathbb{Z}/2\mathbb{Z}$-periodic triangulated categories of 
matrix factorizations
defined locally on each d-critical chart, 
though its construction seems to be a hard 
problem at this moment 
(see~\cite[(J)]{Jslide}, \cite[Section~6.1]{MR3728637}). 
If it exists, the category $\dD(M, s)$
may be interpreted as a kind of `Fukaya category'
of d-critical loci, 
or $(-1)$-shifted symplectic derived schemes
(see~\cite[Conjecture~1.2]{JoySaf}). 
Moreover we expect that the 
numerical realization of 
semi-orthogonal 
complement of the embedding (\ref{intro:shifted:emb})
recovers wall-crossing formula 
of Donaldson-Thomas (DT)
invariants on CY 3-folds established in~\cite{JS, K-S}. 
Thus our d-critical birational geometry 
gives a link of two research subjects 
developed independently, wall-crossing formula of 
DT invariants and D/K equivalence conjecture. 

If $M^{\pm}$ are smooth, 
so in particular $s^{\pm}=0$, we can use 
usual derived categories of coherent sheaves $D^b(M^{\pm})$
to ask 
an analogue of the above question. 
In~\cite{TodDK}, we address this question in the case of 
simple wall-crossing diagrams 
of stable pair moduli spaces.

\begin{figure}\label{fig:intro}
\begin{align*}
\xymatrix{
&  \ovalbox{\mbox{Wall-crossing in DT theory}}  
\ar[dd]^-{\mbox{categorification}} \\
\ovalbox{\mbox{D-critical birational geometry}} 
\ar[ru]^-{\begin{array}{ll}\mbox{geometric} \\
\mbox{interpretation}
\end{array}} \ar[rd]_-{\mbox{analogy}} &  \\
& \ovalbox{\mbox{D/K equivalence conjecture}}
}
\end{align*}
\end{figure}

\subsection{Outline of the paper}
The outline of this paper is as follows. 
In Section~\ref{sec:birational}, we
review basic terminology of 
birational geometry. 
In Section~\ref{sec:dbir}, we recall 
Joyce's d-critical loci
and introduce notions of d-critical 
birational transformations. 
In Section~\ref{sec:moduliCY3}, 
we introduce moduli spaces of 
semistable objects on CY 3-folds 
and formulate the 
question on their wall-crossing diagrams. 
In Section~\ref{sec:quiver}, we set 
notation of moduli spaces of representations 
of quivers with convergent super-potentials. 
In Section~\ref{sec:anatheorem}, we state
analytic neighborhood theorem 
for wall-crossing diagrams in CY 3-folds, and 
give an outline of the proof. 
In Section~\ref{sec:repsym}, we investigate 
wall-crossing phenomena in symmetric 
and extended quivers. 
In Section~\ref{sec:onedim}, we describe 
wall-crossing diagrams of one dimensional stable 
sheaves on CY 3-folds in terms of d-critical flops. 
In Section~\ref{sec:wc:mmp}, we describe 
wall-crossing diagrams of stable pair 
moduli spaces in terms of d-critical flips. 
In Appendix~\ref{sec:Bridgeland}, we review basics on 
Bridgeland stability conditions. 
In Appendix~\ref{sec:other}, we give 
some more examples of wall-crossing in CY 3-folds, 
and describe them in terms of
d-critical birational geometry. 
In Appendix~\ref{sec:append}, we recall how wall-crossing 
diagrams in this paper have been relevant in 
the study of Donaldson-Thomas invariants.

\subsection{Acknowledgements}
The author is grateful to Daniel 
Halpern-Leistner
for explaining the announced work~\cite{AHLH}
on the existence of good moduli spaces for 
Bridgeland semistable objects. 
The author is also grateful to
Chen Jiang and Dominic Joyce 
for valuable discussions. 
The author is supported by World Premier International Research Center
Initiative (WPI initiative), MEXT, Japan, and Grant-in Aid for Scientific
Research grant (No. 26287002) from MEXT, Japan.

\subsection{Notation and convention}
In this paper, all the varieties and schemes are 
defined over $\mathbb{C}$. 
For a smooth variety or a complex manifold $M$, 
we denote by $K_M$ its canonical divisor 
and $\omega_M=\oO_M(K_M)$ its canonical line bundle. 
For 
a smooth projective variety $X$ and 
$\beta, \beta' \in H_2(X, \mathbb{Z})$, we 
write $\beta> \beta'$ if 
$\beta-\beta'$ is a class of an effective one cycle 
on $X$. 
For a projective morphism $f \colon Y \to Z$ of varieties, we denote by 
$\rho(Y/Z)$ its relative Picard number. 
When $f$ is birational, its exceptional locus is 
denoted by $\mathrm{Ex}(f)$. 
For a scheme $M$, we denote by 
$D^b(M) \cneq D^b(\Coh(M))$ the bounded derived category 
of coherent sheaves on $M$. 

\section{Review of birational geometry}\label{sec:birational}
In this section, we review some basics on 
birational geometry and recall several terminologies. 
A standard reference is~\cite{KM}. 
\subsection{Terminology from birational geometry}
Let $Y$ be a projective variety with at worst terminal singularities
(e.g. $Y$ is smooth).
A \textit{minimal model program (MMP for short)}
 of $Y$ is a sequence of birational maps
\begin{align}\label{MMP}
Y=Y_1 \dashrightarrow Y_2 \dashrightarrow \cdots \dashrightarrow
Y_{N-1} \dashrightarrow
Y_N
\end{align}
satisfying the following: 
\begin{enumerate}
\item each $Y_i$ is a projective variety with 
at worst terminal singularities. 
\item each birational map 
$Y_i \dashrightarrow Y_{i+1}$ is either a divisorial 
contraction or a flip. 
\item $Y_N$ is either a minimal model, i.e. 
$K_{Y_N}$ is nef, 
or has a Mori fiber space structure $Y_N \to Z$. 
The former occurs if and only if the Kodaira dimension of $Y$ is non-negative. 
\end{enumerate}

Here a line bundle $L$ on a variety $Y$ is called \textit{nef}
if for any projective curve $C \subset Y$, we have 
$\deg(L|_{C}) \ge 0$. A Cartier divisor 
$D$ on $Y$ is also called \textit{nef} if the associated line bundle 
$\oO_Y(D)$ is nef. 
The minimal model $Y_N$ is not necessary unique, 
but two birational minimal models 
are known to be connected by a sequence of flops~\cite{Kawaflo}. 

The above notions in birational geometry 
are summarized in the 
following definitions. In this paper we only treat the case of smooth 
varieties, but the definitions are the same 
for varieties with terminal singularities.

\begin{defi}\label{defi:contraction}
Let $Y$ be a smooth variety (resp.~complex manifold)
and 
$f \colon Y \to Z$ a projective morphism
of varieties (resp.~analytic spaces)
with $\rho(Y/Z)=1$
and $f_{\ast}\oO_Y=\oO_Z$.  
Then $f$ is called 
\begin{enumerate}
\item \textit{divisorial contraction} 
if $\dim Y=\dim Z$
(i.e. $f$ is birational or bimeromorphic), 
$-K_Y$ is $f$-ample and 
$\Ex(f)$ is a divisor. 
\item \textit{(anti) flipping contraction} if
$\dim Y=\dim Z$, 
$-K_Y$ (resp.~$K_Y$) is $f$-ample and 
$f$ is isomorphic in codimension one. 
\item \textit{flopping contraction} 
if $\dim Y=\dim Z$, crepant (i.e. 
$K_Y \cdot C=0$ for any curve $C \subset Y$ 
such that $f(C)$ is a point)
and $f$ is isomorphic in codimension one. 
\item \textit{Mori fiber space (MFS for short)} if $\dim Z<\dim Y$
and $-K_Y$ is $f$-ample. 
\end{enumerate}
\end{defi}
We can formulate the relevant 
birational transformations using the diagram (\ref{dia:bir})
below in a unified way: 
\begin{defi}\label{defi:dia}
Let $Y^{+}$, $Y^{-}$ be smooth varieties (or complex manifolds), 
and 
consider a diagram
\begin{align}\label{dia:bir}
\xymatrix{
Y^{+} \ar[rd]_-{f^{+}} &  &  Y^{-} \ar[ld]^-{f^{-}} \\
& Z &
}
\end{align}
where $f^+, f^{-}$ are projective morphisms
of varieties (resp.~analytic spaces)
satisfying $f^{\pm}_{\ast}\oO_{Y^{\pm}}=\oO_Z$
if $Y^{\pm} \neq \emptyset$. 
Then the diagram (\ref{dia:bir}) is called 
\begin{enumerate}
\item \textit{divisorial contraction}
 if $f^{+}$ is a divisorial contraction and 
$f^{-}$ is an isomorphism. 
\item \textit{flip} if $f^{+}$ is a flipping contraction, 
and $f^{-}$ is an anti flipping contraction. 
\item \textit{flop} if $f^{+}$, $f^{-}$ are flopping contraction
and the birational map $Y^{+} \dashrightarrow Y^{-}$ is not an isomorphism. 
\item \textit{MFS} if $f^{+}$ is a MFS and 
$Y^{-}=\emptyset$. 
\end{enumerate}
\end{defi}

\subsection{Generalized flips, flops and MFS}
We also use the following generalized terminology, 
without assuming the condition of relative Picard numbers etc: 
\begin{defi}\label{defi:generalized}
In the situation of Definition~\ref{defi:dia}, 
we call a
 diagram (\ref{dia:bir}) 
\begin{enumerate}
\item \textit{generalized flip} if $f^{\pm}$ are  
birational (or bimeromorphic) morphisms, 
$-K_{Y^{+}}$ is $f^{+}$-ample and 
$K_{Y^{-}}$ is $f^{-}$-ample. 
\item \textit{generalized flop} if $f^{\pm}$ are crepant birational (or bimeromorphic)
morphisms, isomorphisms in codimension one, 
and 
there exists a 
$f^+$-ample divisor on $Y^+$
whose strict transform to $Y^-$ is 
$f^-$-anti-ample.  
\item \textit{generalized MFS}
 if $-K_{Y^{+}}$ is $f^{+}$-ample and $Y^{-}=\emptyset$. 
\end{enumerate}
\end{defi}

\begin{rmk}\label{rmk:codim}
In the definition of generalized flip, we 
don't assume that $f^{+}$ is isomorphic in codimension one. 
For example, a divisorial contraction is a generalized flip. 
On the other hand, the morphism $f^{-}$ for a generalized flip
is always isomorphic in codimension 
one by~\cite[Lemma~3.38]{KM}\footnote{The reference of this fact was
pointed out to the author by Chen Jiang.}. 
\end{rmk}

\begin{rmk}\label{rmk:isom}
The conditions (i) and (ii) 
in Definition~\ref{defi:generalized} are not complementary conditions. 
Indeed a diagram (\ref{dia:bir}) is both of 
generalized flip and generalized flop if and only if 
$f^{+}$ and $f^-$ are isomorphisms. 
\end{rmk}

\begin{rmk}\label{rmk:genMFS}
In the definition of generalized MFS, 
we don't assume that 
$\dim Z<\dim Y^+$, 
so $f^+$ can be birational. 
Indeed such a case may happen in wall-crossing 
of stable pair moduli spaces in 
Section~\ref{sec:wc:mmp} (see Remark~\ref{rmk:dMFS}). 
\end{rmk}

\begin{rmk}\label{rmk:MMP}
If the diagram (\ref{dia:bir}) is a generalized flip, 
by applying MMP relative to $Z$ proved in~\cite{BCHM}, 
the birational map $Y^{+} \dashrightarrow Y^-$ decomposes into 
divisorial contractions and flips over $Z$
(though the intermediate varieties may have terminal 
singularities).  
Similarly a generalized flop decomposes into 
flops over $Z$ by~\cite{Kawaflo}. 
\end{rmk}

\subsection{Inequalities of canonical divisors}
Following Kawamata (for example see~\cite{KawBir}), we introduce the 
inequalities of canonical divisors (with a slight modification 
allowing empty sets): 
\begin{defi}\label{defi:ineqK}
In the situation of Definition~\ref{defi:dia}, 
we write 
\begin{enumerate}
\item 
$Y^{+}>_{K} Y^{-}$ if either 
$Y^{-}=\emptyset$ (while $Y^+ \neq \emptyset)$, 
or there is a commutative diagram
\begin{align}\label{com:resol}
\xymatrix{
& W \ar[ld]_-{g^{+}} \ar[rd]^-{g^{-}}  & \\
Y^{+}\ar[rd]_-{f^{+}} & & Y^{-} \ar[ld]^-{f^{-}} \\
& Z & 
}
\end{align}
 such that $g^{\pm}$ are birational and 
$(g^{+})^{\ast}K_{Y^{+}}-(g^{-})^{\ast}K_{Y^{-}}$ is 
linearly equivalent to an 
effective divisor
on $W$. 
\item $Y^{+}=_K Y^{-}$ if either 
$Y^{\pm}=\emptyset$ or there is a commutative 
diagram (\ref{com:resol}) 
for birational maps $g^{\pm}$ such that 
$(g^{+})^{\ast}K_{Y^{+}}$ and $(g^{-})^{\ast}K_{Y^{-}}$
are linearly equivalent. 
\end{enumerate} 
\end{defi}
The canonical divisors decrease
by divisorial contractions and flips,  
while flops keep them 
(see~\cite[Lemma~3.38]{KM}).
Therefore MMP is interpreted as a process 
decreasing 
the canonical divisors, i.e. 
for a MMP (\ref{MMP}), we have the inequalities of 
canonical divisors
\begin{align*}
Y=Y_1 >_K Y_2 >_K \cdots >_K Y_{N-1} >_K Y_N.
\end{align*}
Moreover we have $Y_{N}=_K Y_{N'}'$ if 
$Y_{N'}'$ is another birational minimal model of $Y$.  
By Remark~\ref{rmk:MMP} (or by using~\cite[Lemma~3.38]{KM}), 
if the diagram (\ref{dia:bir}) is a generalized flip 
(resp.~generalized flop), we have 
\begin{align*}
Y^{+}>_K Y^-, \ (\mbox{resp.}~Y^{+}=_K Y^-). 
\end{align*}

\section{Birational transformations for d-critical loci}\label{sec:dbir}
The notion of d-critical loci was introduced by Joyce~\cite{JoyceD}, 
as a classical shadow of $(-1)$-shifted symplectic derived schemes~\cite{PTVV}. 
In this section, we recall its definition and introduce 
analogue of birational transformations in the previous section 
for d-critical loci. 

\subsection{D-critical locus}
Let $M$ be a complex scheme (resp.~complex analytic space). 
In~\cite{JoyceD}, it 
is proved 
that 
there exists a canonical sheaf of 
$\mathbb{C}$-vector spaces 
$\sS_{M}$ on $M$
satisfying the following
property:
for any Zariski (resp.~analytic) 
open subset $U \subset M$
and a closed embedding 
$i \colon U \hookrightarrow Y$
into a smooth scheme (resp.~complex manifold) $Y$, 
there is an exact sequence
\begin{align}\label{S:property}
0 \longrightarrow \sS_{M}|_{U}
 \longrightarrow \oO_Y/I^2 \stackrel{d_{\rm{DR}}}{\longrightarrow}
\Omega_Y/I \cdot \Omega_Y. 
\end{align}
Here $I \subset \oO_Y$ is the ideal sheaf
which defines $U$
and $d_{\rm{DR}}$ is the de-Rham differential. 
Moreover there is a natural decomposition 
\begin{align*}
\sS_M=\sS_M^0 \oplus \mathbb{C}_M
\end{align*}
where $\mathbb{C}_M$ is the constant sheaf
on $M$. 
The sheaf $\sS_M^{0}$ restricted to $U$ is 
the kernel of the 
composition 
\begin{align*}
\sS_M|_{U} \hookrightarrow \oO_Y/I^2 \twoheadrightarrow \oO_{U^{\rm{red}}}. 
\end{align*}
For example, 
suppose that  
$w \colon Y \to \mathbb{C}$ is
an algebraic (resp.~holomorphic) function such that 
\begin{align}\label{R=df}
U=\{dw=0\}, \ 
w|_{U^{\rm{red}}}=0.
\end{align}
Then $I=(dw)$ and 
$w+(dw)^2$ is an element of
$\Gamma(U, \sS_{M}^0|_{U})$. 

\begin{defi}(\cite{JoyceD})\label{defi:dcrit}
A pair $(M, s)$
for a complex scheme (resp.~analytic space) $M$
and $s \in \Gamma(M, \sS_M^0)$
is called \textit{algebraic (resp.~analytic) d-critical 
locus} 
if for any 
$x \in M$, there exist a Zariski (resp.~analytic) open 
neighborhood $x \in U \subset M$,  
a closed embedding $i \colon U \hookrightarrow Y$
into a smooth scheme (resp.~complex manifold) $Y$, 
an element $w \in \Gamma(\oO_Y)$ 
satisfying (\ref{R=df})
such that 
$s|_{U}=w+(dw)^2$
holds. 
In this case, 
the data
\begin{align}\label{crit:chart}
(U, Y, w, i)
\end{align}
is called a \textit{d-critical chart}.
The section $s$ is called a \textit{d-critical 
structure} of $M$.  
\end{defi} 

\begin{rmk}\label{rmk:smooth}
If $M$ is smooth, then $\sS_M^0=0$ so there is a 
unique (trivial) choice of its d-critical structure, $s=0$. 
\end{rmk}

Given a d-critical locus $(M, s)$, there exists a line bundle
$\omega_{M, s}$ on $M^{\rm{red}}$
called \textit{virtual canonical line bundle}\footnote{In~\cite[Section~2.4]{JoyceD}, it was just called canonical line bundle. We put `virtual' in order to 
distinguish with the usual canonical line bundle.}
(see~\cite[Section~2.4]{JoyceD}). 
It satisfies that, 
for any
d-critical chart (\ref{crit:chart})
there is a natural isomorphism 
\begin{align}\label{nat:K0}
\omega_{M, s}|_{U^{\rm{red}}} \stackrel{\cong}{\to} 
\omega_Y^{\otimes 2}|_{U^{\rm{red}}}. 
\end{align}

\begin{rmk}\label{rmk:shifted}
For a derived scheme $M^{\rm{der}}$ with a $(-1)$-shifted 
symplectic structure~\cite{PTVV}, its classical 
truncation $M$ carries a canonical d-critical structure by~\cite{MR3352237}. 
In this case, the virtual canonical line bundle of $M$
is the determinant of  the cotangent complex of $M^{\rm{der}}$ restricted to $M$. 
\end{rmk}

We introduce the following relative version of 
d-critical chart: 
\begin{defi}\label{defi:relchart}
Let $(M, s)$ be an algebraic (resp.~analytic)
d-critical locus and 
$\pi \colon 
M \to A$ a morphism of schemes (resp.~analytic spaces). 
For a Zariski (resp.~analytic) open subset $U \subset A$, suppose that 
there is a following commutative diagram
\begin{align}\label{relchart}
\xymatrix{
\pi^{-1}(U) \ar[d]^-{\pi}
 \ar@<-0.3ex>@{^{(}->}[r]^-{i} & Y \ar[d]^-{f} \ar[rd]^-{w} \\
U \ar@<-0.3ex>@{^{(}->}[r]_-{j} & Z \ar[r]_-{g} & \mathbb{C}
}
\end{align}
where $f \colon Y \to Z$ is a morphism of schemes (resp. analytic spaces), 
$Y$ is smooth, $i$ and $j$ are closed immersions, $g \in \Gamma(\oO_Z)$
 such that the data
 \begin{align*}
 (\pi^{-1}(U), Y, w, i)
 \end{align*}
  is a d-critical chart. 
In this case, we call the diagram (\ref{relchart}) as a
\textit{$\pi$-relative d-critical chart}.  
\end{defi}

\subsection{D-critical birational transformations}
We formulate the terminology of 
birational contractions for d-critical loci, using 
relative d-critical charts in Definition~\ref{defi:relchart}: 
\begin{defi}\label{defi:dbir}
Let $(M, s)$ be an algebraic (resp.~analytic)
d-critical locus and 
\begin{align}\label{mor:MA}\pi \colon 
M \to A
\end{align} 
a morphism of schemes (resp.~analytic spaces). 
We
 call the morphism (\ref{mor:MA}) an \textit{algebraic (resp.~analytic)
d-critical divisorial contraction, d-critical 
(anti) flipping contraction, d-critical 
flopping contraction
at a point $p \in A$} if there exist a Zariski (resp.~analytic) open neighborhood 
$p \in U \subset A$, a $\pi$-relative d-critical chart (\ref{relchart})
such that $f \colon Y \to Z$ is a divisorial contraction, 
(anti) flipping contraction, flopping contraction, MFS,
as in Definition~\ref{defi:contraction}
respectively. 

We call the morphism (\ref{mor:MA}) an 
\textit{algebraic (resp.~analytic)
d-critical 
divisorial contraction, d-critical 
(anti) flipping contraction, d-critical 
flopping contraction, d-critical MFS},
if 
the above corresponding 
condition holds for any $p \in A$.
\end{defi}

A d-critical birational contraction need not to be birational between 
underlying spaces. Indeed, we have the following example: 
\begin{exam}\label{exam:dcont}
Let $U^{\pm}$ be the following affine schemes with 
d-critical structures $s^{\pm}$
\begin{align*}
U^{\pm} \cneq \Spec \mathbb{C}[x, y^{\pm}]/(xy^{\pm}, y^{\pm 2}),  \
s^{\pm}=xy^{\pm 2}+(d(xy^{\pm 2}))^2. 
\end{align*}
By gluing 
$U^+$
and $U^-$ at the smooth open subset
 $\Spec \mathbb{C}[x, x^{-1}]$, 
 we obtain an algebraic 
 d-critical locus 
\begin{align*}
(M, s), \ M=U^+ \cup U^-, \ 
s|_{U^{\pm}}=s^{\pm}.
\end{align*}
 Note that $M^{\rm{red}}=\mathbb{P}^1$, and $M$ is non-reduced at 
 the points $\{0\}$ and $\{\infty\}$. 
The structure morphism 
\begin{align*}
\pi \colon M \to \Spec \mathbb{C}
\end{align*}
 is an algebraic 
d-critical divisorial contraction, though they are not birational in the usual sense. 
Indeed we have the following $\pi$-relative d-critical chart
\begin{align}\notag
\xymatrix{
M \ar[d]^-{\pi}
 \ar@<-0.3ex>@{^{(}->}[r]^-{i} & \widehat{\mathbb{C}}^2 \ar[d]^-{f} \ar[rd]^-{w} \\
\Spec \mathbb{C} \ar@<-0.3ex>@{^{(}->}[r]_-{j} & \mathbb{C}^2 \ar[r]_-{g} & \mathbb{C}
}
\end{align}
where $f$ is the blow-up at $0 \in \mathbb{C}^2$
and $g$ is the function $g(u, v)=uv$. 
\end{exam}

We also formulate a d-critical version of 
birational transformations below: 

\begin{defi}\label{defi:dflip}
Let $(M^{+}, s^{+})$, $(M^{-}, s^{-})$ be 
algebraic (resp.~analytic) d-critical loci and consider 
a diagram 
\begin{align}\label{dia:dbir}
\xymatrix{
M^{+} \ar[rd]_-{\pi^{+}} &  &  M^{-} \ar[ld]^-{\pi^{-}} \\
& A &
}
\end{align}
where $\pi^{\pm}$ are morphisms of schemes (resp.~analytic spaces). 
Then we call the diagram (\ref{dia:dbir}) 
an \textit{algebraic (resp.~analytic) 
d-critical divisorial contraction, 
d-critical (generalized) flip, 
d-critical (generalized) flop, d-critical (generalized) MFS,
at $p \in A$} if there exist a Zariski (resp.~analytic) open neighborhood $p\in U \subset A$
and 
$\pi^{\pm}$-relative d-critical charts 
\begin{align}\label{relchart2}
\xymatrix{
(\pi^{\pm})^{-1}(U) \ar[d]^-{\pi^{\pm}}
 \ar@<-0.3ex>@{^{(}->}[r]^-{i^{\pm}} 
& Y^{\pm} \ar[d]^-{f^{\pm}} \ar[rd]^-{w^{\pm}} \\
U \ar@<-0.3ex>@{^{(}->}[r]_-{j} & Z \ar[r]_-{g} & \mathbb{C}
}
\end{align}
where 
$g \in \Gamma(\oO_Z)$ 
and $j$ are
independent of $\pm$, 
such that the diagram 
\begin{align}\label{dia:YZ}
Y^{+} \stackrel{f^{+}}{\to} Z \stackrel{f^{-}}{\leftarrow} Y^{-}
\end{align}
 is 
a divisorial contraction, (generalized) flip, (generalized) flop,
(generalized) MFS, as in Definition~\ref{defi:dia}, 
Definition~\ref{defi:generalized}
respectively. 

We call the diagram (\ref{dia:dbir}) 
an \textit{algebraic (resp.~analytic)
d-critical divisorial contraction, d-critical (generalized) flip, 
d-critical (generalized) flop, d-critical MFS}, 
respectively, if the above corresponding condition holds for any 
$p \in A$. 
\end{defi}

Here we give some examples of 
d-critical flips, d-critical flops: 
\begin{exam}\label{exam:toric}
Let $V^{+}$, $V^-$ be finite dimensional $\mathbb{C}$-vector spaces 
with dimension $a$, $b$ on which 
$\mathbb{C}^{\ast}$ acts by weight $1$, $-1$ respectively. 
We denote by 
\begin{align*}
\vec{x}=(x_1, \ldots, x_a), \ 
\vec{y}=(y_1, \ldots, y_{b})
\end{align*}
coordinates of $V^+$, $V^-$ respectively. 
For $c \in \mathbb{Z}_{\ge 0}$, let 
$U=\mathbb{C}^c$ (resp.~an analytic open neighborhood 
$0 \in U \subset \mathbb{C}^c)$
with a trivial $\mathbb{C}^{\ast}$-action. 
By taking GIT quotients of 
$V^+ \times V^{-} \times U$ by the 
$\mathbb{C}^{\ast}$-action with respect to the 
character $\pm \id \colon \mathbb{C}^{\ast} \to \mathbb{C}^{\ast}$, 
we obtain
\begin{align*}
&Y^+ \cneq \mathrm{Tot}_{\mathbb{P}(V^+)}(\oO_{\mathbb{P}(V^+)}(-1) \otimes V^-) \times U \\
&Y^- \cneq \mathrm{Tot}_{\mathbb{P}(V^-)}(\oO_{\mathbb{P}(V^-)}(-1) \otimes V^+) \times U.
\end{align*}
Then by setting
\begin{align*}
Z \cneq \Spec \mathbb{C}[x_i y_j : 1\le i \le a, 1\le j\le b] \times U
\end{align*}
we obtain the diagram
\begin{align}\label{dia:toricflip}
Y^{+} \stackrel{f^+}{\to} Z \stackrel{f^-}{\leftarrow} Y^-
\end{align}
which is a standard toric flip if 
$a>b\ge 2$, 
standard toric flop if $a=b \ge 2$
(see~\cite{Rei92}). 
Let us consider $w \in \Gamma(\oO_Z)$ of the form
\begin{align*}
g=\sum_{i=1}^a \sum_{j=1}^b w_{ij}(\vec{u}) x_i y_j, \ 
w_{ij}(\vec{u}) \in \Gamma(\oO_U). 
\end{align*}
We set $w^{\pm}$ by the commutative diagram
\begin{align*}
\xymatrix{
Y^{+} \ar[r]^-{f^{+}} \ar[rd]_-{w^+} 
& Z \ar[d]^-{g} &
\ar[l]_-{f^-} Y^- \ar[ld]^-{w^-}\\
& \mathbb{C}. &
}
\end{align*}
If we have $c \gg 0$ compared to $a, b$, and 
$w_{ij}(\vec{u})$ are sufficiently general, 
then the critical loci
\begin{align}\label{M:pm}
M^{\pm} \cneq \{dw^{\pm}=0\} \subset Y^{\pm}
\end{align}
are smooth of dimension $\pm (a-b)+c-1$. 
Moreover, $f^{\pm}(M^{\pm})$ are contained in 
$\{0\} \times U$. Therefore the diagram
\begin{align}\label{dia:ab2}
M^+ \stackrel{\pi^+}{\to} U \stackrel{\pi^-}{\leftarrow} M^-
\end{align}
is an algebraic (resp.~analytic) d-critical flip if $a>b\ge 2$, 
an algebraic (resp.~analytic) d-critical flop if $a=b \ge 2$. 
Here as $M^{\pm}$ are smooth, the d-critical 
structures $s^{\pm}$ on $M^{\pm}$ must be zero. 
Note that in the former case, the dimensions 
of $M^{\pm}$ are different. 
The fibers of $\pi^{\pm}$ at $u \in U$ are linear subspaces in 
$\mathbb{P}(V^{\pm})$, whose dimensions depend on $u$. 
\end{exam}

Here is an example of analytic d-critical flips, flops 
for a diagram of smooth projective varieties. 
\begin{exam}\label{exam:symC}
Let $C$ be a smooth projective curve with genus $g$, 
and let $S^k(C)$ be the $k$-th symmetric product of $C$:
\begin{align*}
S^k(C) \cneq (\overbrace{C \times \cdots \times C)}^k/\mathfrak{S}_k.
\end{align*}
Note that $S^k(C)$ is a smooth projective variety with dimension $k$. 
Let $\Pic^k(C)$ be the moduli space of line bundles on $C$
with degree $k$, which is a $g$-dimensional complex torus. 
For each $n> 0$, we consider the 
classical diagram of Abel-Jacobi maps
\begin{align}\label{dia:Pflip3}
\xymatrix{
S^{n+g-1}(C)  \ar[rd]_-{\pi^{+}}
 & & 
S^{-n+g-1}(C)
\ar[ld]^-{\pi^{-}} \\
& \Pic^{n+g-1}(C). &
}
\end{align}
Here the morphisms $\pi^{\pm}$ are given by
\begin{align*}
\pi^{+}(Z \subset C)=\oO_C(Z), \ 
\pi^{-}(Z' \subset C)=\omega_C(-Z').
\end{align*}

The diagram (\ref{dia:Pflip3}) appears as a 
special case of wall-crossing of stable pair
moduli spaces discussed in Theorem~\ref{thm:irreducible}
(see Example~\ref{exam:curve}). 
By \textit{loc.~cit.~}, 
at a point $[L] \in \Pic^{n+g-1}(C)$
we see that 
the diagram (\ref{dia:Pflip3}) is 
an analytic d-critical flip if $h^1(L)>1$, an analytic 
d-critical divisorial contraction if $h^1(L)=1$ and 
an analytic d-critical MFS
if $h^1(L)=0$. 
Moreover relative d-critical charts are analytic locally on $\Pic^{n+g-1}(C)$
given as in Example~\ref{exam:toric}. 
Note that $S^{\pm n +g-1}(C)$ are smooth projective 
varieties, which are not birational
for $n>0$ (as the dimensions are different). 
\end{exam}

Here is an example of a d-critical flop
between non-reduced d-critical loci: 
\begin{exam}\label{exam:dflop:atiyah}
Let us consider the case $a=b=2$ and $c=0$ in 
Example~\ref{exam:toric}. 
In this case, the diagram (\ref{dia:toricflip}) is the simplest
example of a flop, called 
\textit{Atiyah flop}. 
Let us take $g \in \Gamma(\oO_Z)$ to be
$g=x_1 x_2 y_1^2$, 
and define $M^{\pm}$ as in (\ref{M:pm}). 
We have the d-critical structures on $M^{\pm}$
by $s^{\pm}=w^{\pm}+(dw^{\pm})^2$, and 
an algebraic d-critical flop
\begin{align*}
M^+ \stackrel{\pi^+}{\to} Z \stackrel{\pi^-}{\leftarrow} M^-
\end{align*}
The schemes $M^{\pm}$
are described as follows. 
Let $\mathbb{P}^1=C^{\pm} \subset Y^{\pm}$ be the 
exceptional loci of $f^{\pm}$. 
Then the reduced part of 
$M^+$ is a smooth divisor on $Y^+$ which contains 
$C^+$. 
However $M^+$ is non-reduced 
along two disjoint curves on $M^+$. 
On the other hand, the reduced part of 
$M^-$ is a union of 
$C^-$ and a smooth divisor on $Y^-$
which intersects $C^-$ at a point $y$. 
The scheme $M^-$ is  
non-reduced along two curves on the above divisor 
which intersect
at $y$. 
The singularities of schemes 
$M^{\pm}$ are 
not treated in birational geometry. 
\end{exam}

\begin{rmk}\label{rmk:dcrit}
As in Remark~\ref{rmk:isom}, 
d-critical generalized flips, flops at $p \in A$
in Definition~\ref{defi:dflip}
include the case that both of $f^{+}$, $f^{-}$
in the diagram (\ref{dia:dbir}) are isomorphisms. 
In this case, the left vertical 
arrows in (\ref{relchart2}) are closed immersions. 
This case also includes the case that 
$(\pi^{\pm})^{-1}(U)=\emptyset$. 
Indeed one can take a closed embedding 
$U \subset Z$ for a smooth $Z$
which admits a smooth morphism $g \colon Z \to \mathbb{C}$. 
Then by taking $Y^{\pm}=Z$, $f^{\pm}=\id$ and $w^{\pm}=g$, 
we have $\{d w^{\pm}\}=\emptyset$. 
\end{rmk}

\begin{rmk}\label{rmk:fiber}
In the notation of Definition~\ref{defi:dflip}, 
suppose that the function $g$ satisfies 
$g \in m_0^2$, where $m_0 \subset \oO_Z$ is the ideal 
sheaf 
of $0 \cneq j(p) \in Z$. 
Then as $g_{\ast} \colon T_{Z, 0} \to T_{\mathbb{C}, g(0)}$ 
is a zero map, we see that (set theoretically)
\begin{align*}
(f^{\pm})^{-1}(0) \subset \{ dw^{\pm}=0\}. 
\end{align*}
In particular if $f^{\pm}$ contracts a curve in $Y^{\pm}$ 
to a point $0 \in Z$, then it lies on $M^{\pm}$ and 
$\pi^{\pm}$ also contracts it to $p$. 
\end{rmk}

\begin{rmk}\label{rmk:contract}
If the condition of Remark~\ref{rmk:fiber} is not satisfied, 
it is possible that $\pi^{\pm} \colon M^{\pm} \to A$ 
do not contract any curve while 
$f^{\pm}$ do. 
For example, let us consider the diagram
\begin{align*}
\xymatrix{
\widehat{\mathbb{C}}^2 \ar[r]^-{f^{+}} \ar[rd]_-{w^+} 
& \mathbb{C}^2 \ar[d]^-{g} &
\ar[l]_-{\id} \mathbb{C}^2 \ar[ld]^-{w^-}\\
& \mathbb{C} &
}
\end{align*}
where $f^+$ is the blow-up at the origin
and $g$ is the projection onto one the factors of $\mathbb{C}^2$. 
By taking the critical locus of $w^{\pm}$, we obtain 
the diagram
\begin{align}\label{ex:empty}
\{dw^+=0\}=\Spec \mathbb{C} \stackrel{\id}{\to} (0 \in \mathbb{C}^2) 
\leftarrow \{dw^-=0\}=\emptyset.
\end{align}
Although $f^{+}$ contracts a $\mathbb{P}^1$ to a point, the above 
diagram does not contract curves. 
\end{rmk}

In order to avoid situations as in Remark~\ref{rmk:dcrit} and Remark~\ref{rmk:contract}, we introduce the 
following strict notion of birational transformations: 
\begin{defi}\label{defi:strict}
In the situation of Definition~\ref{defi:dflip},  
we call
a diagram (\ref{dia:dbir}) 
\textit{strict at $p \in A$}
if $\dim (\pi^{+})^{-1}(p) \ge 1$, i.e. 
$\pi^+$ is not a finite morphism at $p$.
We call a diagram (\ref{dia:dbir}) is
\textit{strict} if 
it is strict at some $p \in A$, i.e. 
$\pi^+$ is not a finite morphism. 
\end{defi}

\begin{rmk}
The diagram (\ref{dia:ab2}) is strict if $a\ge 2$ by Remark~\ref{rmk:fiber}. 
On the other hand, the diagram (\ref{ex:empty}) is not strict. 
\end{rmk}

\subsection{D-critical MMP}
We define the following d-critical version of MMP as follows:
\begin{defi}\label{def:d-mmp}
Let $(M, s)$ be an algebraic (resp.~analytic) d-critical 
locus. A \textit{d-critical MMP} of $(M, s)$ is a sequence
\begin{align}\label{d-MMP}
\xymatrix{
M_1 \ar[rd]_-{\pi_1^+}  &  &  M_2 \ar[ld]^-{\pi_1^-} \ar[rd]_-{\pi_2^+} &  & \cdots \ar[ld] \ar[rd] &
  &  M_{N} \ar[ld]^-{\pi_{N-1}^-}
\\
& A_1  & & A_2 & & A_{N-1}  &
}
\end{align} 
where each $(M_i, s_i)$ is an algebraic (resp.~analytic) d-critical locus, 
$(M_1, s_1)=(M, s)$ as d-critical loci, and for each $i$
the diagram
\begin{align}\label{MAi}
M_i \stackrel{\pi_i^+}{\to} A_i 
\stackrel{\pi_i^-}{\leftarrow} M_{i+1}
\end{align}
 is an algebraic
 (resp.~analytic) d-critical generalized
flip
at any point in $\Imm \pi_i^-$, and
d-critical generalized MFS 
at any point in $A_i \setminus \Imm \pi_i^-$. 
A d-critical MMP is called \textit{strict} if 
each diagram (\ref{MAi})
is strict 
in the sense of Definition~\ref{defi:strict}. 
\end{defi}

We give an example of a d-critical MMP 
from a usual MMP: 
\begin{exam}\label{ex:mmp}
Let $\xX$ be a complex manifold with a projective 
morphism $f \colon \xX \to \Delta$
where $0 \in \Delta \subset \mathbb{C}$ is a small disc. 
Suppose that $f^{-1}(t)$ is a smooth 
minimal model for any $t \in \Delta \setminus \{0\}$. 
Also suppose that we have a $f$-relative 
MMP of $\xX$ over $\Delta$
\begin{align*}
\xX=\xX_1 \dashrightarrow \xX_2 \dashrightarrow \cdots \dashrightarrow 
\xX_{N-1} \dashrightarrow \xX_N
\end{align*}
where $\xX_N \to \Delta$ is a minimal model over $\Delta$
such that each $\xX_i$ is smooth. 
(For example such a MMP always exists when $\dim \xX=2$.)
Then each birational map 
$\xX_i \dashrightarrow \xX_{i+1}$ fits into the diagram
\begin{align*}
\xymatrix{
\xX_i \ar[r]^-{\pi_i^+}  \ar[dr]_-{f_i} & \yY_i \ar[d]_-{g_i} & \ar[l]_{\pi_{i}^-} \ar[ld]^-{f_{i+1}} \xX_{i+1} \\
& \Delta &
}
\end{align*}
where the top diagram is either a divisorial contraction or a flip. 
Let $h \colon \Delta \to \mathbb{C}$ be
defined by $t \mapsto t^2$ and 
set 
\begin{align*}
w_i \cneq h \circ f_i \colon \xX_i \to \mathbb{C}, \ 
M_i \cneq \{dw_i=0\}, \ 
A_i \cneq \yY_i \times_{\Delta} \Spec \mathbb{C}[t]/t^m
\end{align*}
for $m\gg 0$. Note that $M_i$, $A_i$ are projective schemes
with $f_i(M_i) \subset A_i$ for $m\gg 0$, 
and $M_i$ admits a d-critical 
structure 
$s_i=w_i+(dw_i)^2$. 
Then we obtain 
a d-critical MMP (\ref{d-MMP}),
 which is strict by Remark~\ref{rmk:fiber}. 
\end{exam}

\begin{rmk}\label{rmk:schemestr}
In Example~\ref{ex:mmp}, note that $(f_i)^{-1}(0)=M_i$ as a set, but 
their scheme structures may be different. 
For example if $(f_i)^{-1}(0)$ is a curve 
with a nodal singularity at $x \in (f_i)^{-1}(0)$, 
then the scheme structure of $M_i$ at $x$
is given by $\widehat{\oO}_{M_i, x}=\mathbb{C}[[x, y]]/(x^2 y, xy^2)$,
which is a critical locus of 
the function $x^2 y^2$, and not isomorphic to the nodal singularity. 
\end{rmk}

As an analogy of minimal model in birational 
geometry, we introduce the following notion of 
minimal d-critical loci: 
\begin{defi}\label{defi:minimal}
A d-critical locus $(M,s)$
is called \textit{minimal } if 
the virtual canonical line bundle 
$\omega_{M, s}$ is nef. 
\end{defi}
\begin{exam}\label{exam:minimal}
(i)
For a d-critical locus $(M, s)$, if 
$M$ is smooth then 
$\omega_{M, s}=\omega_M^{\otimes 2}$. Therefore 
$(M, s)$ is minimal if and only if $K_M$ is nef, i.e. 
$M$ is minimal in the usual sense. 

(ii) In the situation of Example~\ref{ex:mmp}, 
the d-critical locus $(M_N, s_N)$ is minimal, 
as $\omega_{M_N, s_N}=\omega_{\xX_N}^{\otimes 2}|_{M_N}$ is nef. 
Note that $M_N$ is a projective singular scheme if 
$(f_N)^{-1}(0)$ is singular.

(iii) There is also an example of a singular 
projective minimal d-critical locus $(M, s)$, 
such that $M^{\rm{red}}$ is a smooth non-minimal model. 
Let $Z$ be the $A_1$ surface singularity 
\begin{align*}
Z=\{xy+z^2=0\} \subset \mathbb{C}^3
\end{align*}
and take the blow-up 
$f \colon Y \to Z$ at the origin, which is a 
crepant resolution of the singularity $0 \in Z$. 
We consider the commutative diagram
\begin{align*}
\xymatrix{
Y \ar[r]^-{f} \ar[dr]_-{w} & Z \ar[d]^{(x, y, z) \mapsto x^2+y^2+z^2} \\
&  \mathbb{C}.
}
\end{align*}
We have the following d-critical locus: 
\begin{align*}
(M, s), \ 
M=\{dw=0\}, \ s=w+(dw)^2.
\end{align*}
Then $M^{\rm{red}}=\mathbb{P}^1$ and 
there are 4-points in $\mathbb{P}^1$ at which 
$M$ is non-reduced. At these points, the 
scheme structure of $M$ is given by the critical 
locus of the function $(x, y) \mapsto xy^2$
on $\mathbb{C}^2$. Note that 
$\omega_{M, s}=\omega_Y^{\otimes 2}|_{M^{\rm{red}}} \cong \oO_{\mathbb{P}^1}$
as $\omega_Y$ is trivial. Therefore the d-critical locus
$(M, s)$ is minimal, while $M^{\rm{red}}$ is not minimal. 
\end{exam}
Similarly to the usual minimal model in 
birational geometry, we have the following lemma: 
\begin{lem}\label{lem:minimal}
Let $(M^+, s^+)$ be a minimal 
d-critical locus. 
Then there is no strict diagram
\begin{align}\label{dia:minimal}
M^+ \stackrel{\pi^+}{\to} A  \stackrel{\pi^-}{\leftarrow} M^-
\end{align}
for a d-critical locus $(M^-, s^-)$, 
which is a d-critical generalized flip 
at any point in $\Imm \pi^-$, and
a d-critical generalized MFS at any 
point in $A\setminus \Imm \pi^-$.
\end{lem}
\begin{proof}
Suppose that a diagram (\ref{dia:minimal}) exists. 
As $\pi^+$ is not a finite morphism, there is a projective curve 
$C \subset M^+$
such that $\pi(C)=p$ for some point $p \in A$. 
Let us take $\pi^{\pm}$-relative
d-critical charts (\ref{relchart})
for an open neighborhood $p \in U \subset A$.  
Then as $-K_{Y^+}$ is $f$-ample, 
we have 
\begin{align*}
\deg(\omega_{M^+, s^+}|_{C})=2K_{Y^+} \cdot i^+(C)<0
\end{align*}
which contradicts that $(M^+, s^+)$ is minimal. 
\end{proof}

\begin{rmk}\label{rmk:mmp}
By Lemma~\ref{lem:minimal},
a strict $d$-critical MMP (\ref{d-MMP}) 
terminates at $M_N$ if 
$M_N$ is either minimal or an empty set. 
In the latter case, 
the morphism 
$M_{N-1} \to A_{N-1}$ is a $d$-critical generalized MFS
at any point in $A_{N-1}$. 
\end{rmk}

We
generalize the inequality of 
canonical divisors in 
Definition~\ref{defi:ineqK} to 
virtual canonical line bundles on 
d-critical loci: 
\begin{defi}\label{defi:dcrit:K}
In the situation of Definition~\ref{defi:dflip}, we write
\begin{align}\label{dcrit:DK}
(M^{+}, s^{+}) \ge_{K} (M^{-}, s^{-})
\end{align}
if for any $p \in A$ there is a 
$\pi^{\pm}$-relative d-critical chart 
(\ref{relchart2})
such that 
$Y^{+} \ge_K Y^{-}$ as in Definition~\ref{defi:ineqK}. 
The inequality (\ref{dcrit:DK}) is strict if 
$Y^+ >_K Y^-$ for some $p \in A$. 
\end{defi}
\begin{rmk}\label{rmk:ineq:d}
By (\ref{nat:K0}), the inequality (\ref{dcrit:DK}) is
regarded as an inequality for virtual
canonical bundles of d-critical loci. 
\end{rmk}

If the diagram (\ref{dia:dbir})
is a d-critical divisorial contraction, (generalized) flip, 
(generalized) MFS, we have the inequality 
(\ref{dcrit:DK}). 
In particular for a d-critical MMP (\ref{d-MMP}),
we have the inequalities
\begin{align*}
(M_1 , s_1) \ge_K (M_2, s_2) \ge_K \cdots 
\ge_K (M_N, s_N).
\end{align*} 
Each inequalities are strict if (\ref{d-MMP}) is a strict d-critical MMP. 
Moreover we have the equality 
of (\ref{dcrit:DK}) 
if the diagram (\ref{dia:dbir}) is a d-critical (generalized) flop.

\section{Moduli spaces of semistable objects on CY 3-folds}
\label{sec:moduliCY3}
In this section, we discuss moduli spaces of 
Bridgeland semistable objects on CY 3-folds, and 
introduce their wall-crossing diagrams as in 
the introduction. We address 
a general question 
whether wall-crossing diagrams 
in CY 3-folds
are described in terms of d-critical 
birational transformations introduced in the previous 
section, which is a main topic in this paper. 

\subsection{Moduli spaces of objects on CY 3-folds}
\label{subsec:moduli:general}
Let $X$ be a smooth projective CY 3-fold, i.e.
$\dim X=3$, $K_X=0$
and $H^1(\oO_X)=0$. 
Below we fix a trivialization
\begin{align}\label{trivialization}
\oO_X \stackrel{\cong}{\to} \omega_X.
\end{align}
We denote by $\mM$ the 2-functor
\begin{align*}
\mM \colon \sS ch/\mathbb{C} \to 
\gG roupoid
\end{align*}
sending a 
$\mathbb{C}$-scheme $S$ to the groupoid
of relatively perfect object 
\begin{align}\label{obj:S}
\eE \in D^b(X \times S)
\end{align}
such that 
for each $s \in S$, 
its derived 
restriction $\eE_s$
to $X \times \{s\}$
satisfies 
$\Ext^{<0}(\eE_s, \eE_s)=0$. 
By the result of Lieblich~\cite{LIE}, 
the 2-functor $\mM$ 
is an Artin stack locally of finite type. 
We have the open substack
\begin{align*}
\mM^{\rm{si}} \subset \mM
\end{align*}
consisting of 
simple objects, i.e. 
substacks of objects (\ref{obj:S})
which furthermore satisfies $\Hom(\eE_s, \eE_s)=\mathbb{C}$
for any $s \in S$. 
Then by~\cite[Corollary~4.3.3]{LIE}
(also see~\cite{Inaba})
there is an algebraic space 
$M^{\rm{si}}$ locally of finite type with a morphism
\begin{align}\label{M:simple}
\mM^{\rm{si}} \to M^{\rm{si}}
\end{align}
which is a \'etale locally trivial 
$B\mathbb{C}^{\ast}$-bundle
(i.e. $\mathbb{C}^{\ast}$-gerbe). 
By the result of~\cite{MR3352237},
we have the following: 
\begin{thm}\emph{(\cite{MR3352237})}\label{thm:CYdcrit}
There is a canonical d-critical structure 
on the stack $\mM$ whose virtual
canonical line bundle is given by
\begin{align}\label{vir:K}
\omega_{\mM}^{\rm{vir}} \cneq 
\det \dR \hH om_{pr_{\mM}}(\eE, \eE). 
\end{align} 
Here $\eE$ is a universal sheaf on $X \times \mM$ and 
$pr_{\mM} \colon X \times \mM
 \to \mM$ is the projection. 
The restriction of the above d-critical structure 
to $\mM^{\rm{si}} \subset \mM$ 
descends to a d-critical 
structure on $M^{\rm{si}}$. 
\end{thm}
\begin{rmk}\label{rmk:canonical}
More precisely, the d-critical structure on $\mM$
in Theorem~\ref{thm:CYdcrit}
is canonically 
determined once we choose a 
trivialization (\ref{trivialization}). 
\end{rmk}

\subsection{Moduli spaces of Bridgeland semistable objects}\label{subsec:msstable}
We define
$\Gamma_X \subset H^{2\ast}(X, \mathbb{Q})$ to be the image 
of the Chern character map
\begin{align*}
\ch \colon 
K(X) \to H^{2\ast}(X, \mathbb{Q}).
\end{align*}
Let
$\Stab(X)$ the 
space of Bridgeland stability conditions on 
the derived category
$D^b(X)$ with respect to the 
Chern character map $\ch \colon K(X) \to \Gamma_X$
(see Appendix~\ref{subsec:space}). 
By its definition, a point $\sigma \in \Stab(X)$
is written as
\begin{align}\label{sigma:stab}
\sigma=(Z, \aA) \in \Stab(X), \ 
\aA \subset D^b(X), \ Z \colon \Gamma_X \to \mathbb{C}
\end{align}
where $\aA$ is the heart of a bounded t-structure and
$Z$ is a group homomorphism, satisfying
some conditions. 
\begin{rmk}
So far it is not known whether $\Stab(X) \neq \emptyset$ for 
a projective CY 3-fold in general. 
At the present time, this is only known 
when 
$X$ is an \'etale quotient of an 
abelian 3-fold~\cite{MR3370123, MR3573975}. 
On the other hand, we can generalize the arguments below 
in a modified situation where 
it is easier to construct stability conditions, e.g. 
$\Stab(X)$ for a non-compact CY 3-fold, 
$\Stab(\dD)$ for a triangulated subcategory 
$\dD \subset D^b(X)$, etc. 
\end{rmk}
For a stability condition $\sigma$ as in (\ref{sigma:stab})
and an element $v \in \Gamma_X$, 
we have the substacks
\begin{align}\label{substack:M}
\mM_{\sigma}^{s}(v) \subset \mM_{\sigma}(v) 
\subset \mM
\end{align}
where $\mM_{\sigma}(v)$ consists of 
$\sigma$-semistable objects $E \in \aA$
with $\ch(E)=v$
and $\mM_{\sigma}^s(v)$ 
is the $\sigma$-stable part of $\mM_{\sigma}(v)$. 
Below we discuss under 
the following assumption:
\begin{assum}\label{assum:stack}
The substacks (\ref{substack:M}) are open substacks of 
$\mM$, and they are of finite type. 
\end{assum}
\begin{rmk}
The above assumption holds in the cases 
where $\Stab(X)$ is known to be non-empty
(see~\cite{PiYT}). 
As proven in \textit{loc.~cit.~}, 
the Bogomolov-Gieseker type inequality conjecture
proposed in~\cite{BMT, BMS}
implies both of constructions of stability conditions and 
Assumption~\ref{assum:stack}. 
\end{rmk}
Under Assumption~\ref{assum:stack},  
we can discuss good moduli spaces for 
Artin stacks (\ref{substack:M})
in the sense of~\cite{MR3237451}. 
A general definition is as follows: 
\begin{defi}(\cite{MR3237451})\label{def:goodmoduli}
A morphism $p \colon \mM \to M$, where $\mM$ is an 
Artin stack and $M$ an algebraic space, is called 
a \textit{good moduli 
space} for $\mM$ if the following conditions hold: 
\begin{enumerate}
\item $p$ is quasi-compact and $p_{\ast} \colon \QCoh(\mM) \to \QCoh(M)$
is exact. 
\item The natural map 
$\oO_M \to p_{\ast}\oO_{\mM}$ is an isomorphism.
\end{enumerate}
\end{defi}
A good moduli space $p \colon \mM \to M$ 
is universal for morphisms to algebraic spaces 
(see~\cite[Theorem~6.6]{MR3237451}). 
Namely for a morphism $p' \colon \mM \to M'$ 
for another algebraic space $M'$, there is a unique 
factorization
\begin{align*}
p' \colon \mM \stackrel{p}{\to} M \to M'. 
\end{align*}

Let us return to moduli stacks of (semi)stable 
objects (\ref{substack:M}). 
We will use the following result
which is announced in~\cite{AHLH}
(also see~\cite[Theorem~4.3]{HLK3}, \cite[Theorem~5.11]{Savvas}):
\begin{thm}\emph{(\cite{AHLH})}\label{thm:goodmoduli}
The stack $\mM_{\sigma}(v)$ admits a good moduli space
\begin{align*}
p_M \colon 
\mM_{\sigma}(v) \to M_{\sigma}(v)
\end{align*}
for a separated algebraic space $M_{\sigma}(v)$ of finite type. 
\end{thm}
The good moduli space $M_{\sigma}(v)$ is 
the coarse moduli space of $S$-equivalence classes of 
$\sigma$-semistable objects with Chern character $v$. 
It follows that there is a one to one correspondence 
between closed points of $M_{\sigma}(v)$ and 
$\sigma$-polystable objects,
i.e. $\sigma$-semistable objects 
in $\aA$ with Chern character $v$, 
isomorphic to direct sums of $\sigma$-stable 
objects. 
\begin{rmk}\label{rmk:AHHL}
The existence of the good moduli space 
is well-known for moduli stacks given as 
GIT quotient stacks, say moduli stacks of 
Gieseker semistable sheaves (see~\cite{Hu}). 
In this case, 
we can proceed the arguments without relying on~\cite{AHLH}. 
In general it is not known whether $\mM_{\sigma}(v)$ can be 
constructed as a GIT quotient stack, so 
we rely on~\cite{AHLH} for the existence of 
$M_{\sigma}(v)$. 
\end{rmk}
Let 
$M_{\sigma}^s(v) \subset M_{\sigma}(v)$
be the open subspace consisting of 
$\sigma$-stable objects. 
Then we have the morphism
\begin{align}\label{M:stable}
p_M \colon 
\mM_{\sigma}^s(v) \to M_{\sigma}^s(v)
\end{align}
giving a good moduli space for $\mM_{\sigma}^s(v)$. 
Note that $M_{\sigma}^s(v)$ is also an 
open subspace of the moduli space of 
simple objects $M^{\rm{si}}$
in the previous subsection, and the morphism 
(\ref{M:stable}) is a pull-back of the morphism (\ref{M:simple})
by the open immersion $M_{\sigma}^s(v) \subset M^{\rm{si}}$. 

\subsection{Wall-crossing diagram in CY 3-folds}\label{subsec:wallCY3}
By a general theory of Bridgeland 
stability conditions, there 
is a collection of locally finite 
codimension one submanifolds (called \textit{walls})
\begin{align}\label{wall}
\{\wW_{\lambda}\}_{\lambda \in \Lambda}, \ 
\wW_{\lambda} \subset \Stab(X)
\end{align}
such that 
the moduli stack $\mM_{\sigma}(v)$ is constant if 
$\sigma$ lies
in a connected component of 
the complements of walls (called \textit{chamber}), but may change 
if $\sigma$ crosses a wall. 
Each wall $\wW_{\lambda}$ is defined by the condition
\begin{align}\notag
Z(v_1) \in \mathbb{R}_{>0} Z(v_2), \ v_1+v_2=v
 \in \Gamma_X
\end{align}
where $(v_1, v_2)$ are not proportional
in $(\Gamma_X)_{\mathbb{Q}}$. 

Below, we take $v \in \Gamma_X$ to be primitive, 
i.e. $v$ is not written as a multiple of some element 
in $\Gamma_X$.
Let us take 
\begin{align*}
\sigma=(Z, \aA) \in \Stab(X), \ 
\sigma^{\pm}=(Z^{\pm}, \aA^{\pm}) \in \Stab(X)
\end{align*}
where $\sigma$ lies on a wall and $\sigma^{\pm}$
lie on its adjacent chambers. 
By applying $\mathbb{C}$-action on $\Stab(X)$ if necessary
(see~Remark~\ref{rmk:Caction}),  
we may assume that $\Im Z(v)>0$. 
Then any $\sigma^{\pm}$-semistable object
$E \in \aA^{\pm}$ with $\ch(E)=v$ is
a $\sigma$-semistable object in $\aA$. 
So we have open 
immersions $\mM_{\sigma^{\pm}}(v) \subset \mM_{\sigma}(v)$, 
and the commutative diagram
\begin{align}\label{wall:diagram}
\xymatrix{
\mM_{\sigma^+}(v) \ar@<-0.3ex>@{^{(}->}[r] \ar[d]_{p_M^+} & \mM_{\sigma}(v) 
\ar[d]^-{p_M} & 
\ar@<0.3ex>@{_{(}->}[l] \mM_{\sigma^-}(v) \ar[d]^{p_M^-} \\
M_{\sigma^+}(v) \ar[r]_-{q_M^+} & M_{\sigma}(v) & M_{\sigma^-}(v) 
\ar[l]^-{q_M^-}.
}
\end{align}
Here the top arrows are open immersions, the 
vertical arrows are morphisms to the good moduli 
spaces, and the bottom arrows are induced
by the universality of good moduli spaces. 
Moreover we have
$M_{\sigma^{\pm}}^s(v)=M_{\sigma^{\pm}}(v)$
as $v$ is primitive and $\sigma^{\pm}$ lie on chambers. 
In particular $M_{\sigma^{\pm}}(v)$ 
admit d-critical structures by Theorem~\ref{thm:CYdcrit}, 
and the following question makes sense: 
\begin{question}
For a primitive $v \in \Gamma_X$, is 
the bottom diagram 
in 
(\ref{wall:diagram})
a (generalized) d-critical flip or flop?
\end{question}

In what follows, we study the above question 
via moduli spaces of representations of Ext-quivers.

\section{Representations of quivers with super-potentials}\label{sec:quiver}
In this section, we construct d-critical birational 
transformations introduced in the previous section
via representations of quivers with convergent super-potentials. 
The descriptions in this section will give 
analytic local models of d-critical birational 
transformations for wall-crossing 
diagrams in CY 3-folds 
considered in the previous section.  
\subsection{Representations of quivers}
Recall that a \textit{quiver} $Q$ consists of data
\begin{align*}
Q=(V(Q), E(Q), s, t)
\end{align*}
where $V(Q), E(Q)$ are finite sets 
and $s, t$ are maps
\begin{align*}
s, t \colon E(Q) \to V(Q).
\end{align*}
The set $V(Q)$ is the set of vertices and
$E(Q)$ is the set of edges. 
For $e \in E(Q)$, 
$s(e)$ is the source of $e$
and $t(e)$ is the target of $e$. 
For $i, j \in V(Q)$, we use the following notation
\begin{align}\label{Eab}
E_{i, j} \cneq \{e \in E(Q) : 
s(e)=i, t(e)=j\}, \
\mathbb{E}_{i, j} &\cneq 
\bigoplus_{e \in E_{i, j}} \mathbb{C} \cdot e
\end{align}
i.e. $E_{i, j}$ is the set of edges 
from $i$ to $j$, 
and $\mathbb{E}_{i, j}$ is the 
$\mathbb{C}$-vector space spanned by $E_{i, j}$. 
The \textit{dual quiver} $Q^{\vee}$ of $Q$ is defined by 
\begin{align*}
Q^{\vee} \cneq (V(Q), E(Q), s^{\vee}, t^{\vee}), \ 
s^{\vee} \cneq t, \ t^{\vee} \cneq s. 
\end{align*}

Recall that a \textit{path}
of a quiver $Q$ 
is a composition of edges in $Q$
\begin{align*}
e_1 e_2 \ldots e_n, \ e_i \in E(Q), \ t(e_i)=s(e_{i+1}). 
\end{align*}
The number $n$ above is called the \textit{length} of the path. 
The \textit{path algebra} of 
a quiver $Q$ is
a $\mathbb{C}$-vector space spanned by 
paths in $Q$:
\begin{align*}
\mathbb{C}[Q] \cneq 
\bigoplus_{n\ge 0}
\bigoplus_{e_1, \ldots, e_n \in E(Q), t(e_i)=s(e_{i+1})} \mathbb{C} \cdot e_1 e_2 \ldots e_n.
\end{align*}
Here a path of length zero is a trivial path 
at each vertex of $Q$, and 
the product on $\mathbb{C}[Q]$ is defined by the 
composition of paths. 
 
A \textit{$Q$-representation} consists of
data
\begin{align}\label{rep:Q}
\mathbb{V}=\{
(V_i, u_e) : \ i \in V(Q),  \ e \in E(Q), \ 
u_e \colon V_{s(e)} \to V_{t(e)}\}
\end{align}
where $V_i$ is a finite dimensional 
$\mathbb{C}$-vector space 
and $u_e$ is a linear map. 
It is well-known that a $Q$-representation is nothing but 
a left $\mathbb{C}[Q]$-module
structure on $\bigoplus_{i \in V(Q)}V_i$. 
Also note that the dual of (\ref{rep:Q})
\begin{align}\label{V:dual}
\mathbb{V}^{\vee} \cneq \{
(V_i^{\vee}, u_e^{\vee}) : \ i \in V(Q),  \ e \in E(Q), \ 
u_e^{\vee} \colon V_{t(e)}^{\vee} \to V_{s(e)}^{\vee}\}
\end{align}
is  
a $Q^{\vee}$-representation.

For a $Q$-representation $\mathbb{V}$ as in (\ref{rep:Q}), the vector
\begin{align}\label{m:vect}
\vec{m}=(m_i)_{i \in V(Q)}, \ 
m_i=\dim V_i
\end{align}
is called the \textit{dimension vector} of $\mathbb{V}$. 
For each $i \in V(Q)$, let 
$S_i$ be the one dimensional
 $Q$-representation corresponding to 
the vertex $i$, whose dimension vector is denoted by 
$\vec{i}$. 
We set
\begin{align*}
\Gamma_Q \cneq \bigoplus_{i \in V(Q)} \mathbb{Z} \cdot \vec{i}. 
\end{align*}
Note that the dimension vector (\ref{m:vect}) 
for a non-zero $Q$-representation (\ref{rep:Q}) 
takes its value in 
the positive cone $\Gamma_{Q, >0} \subset \Gamma_Q$:
\begin{align}\label{Q:positive}
\Gamma_{Q, >0} \cneq 
\left\{ \vec{m}=(m_i)_{i \in V(Q)} \in \Gamma_Q : 
m_i \ge 0\right\} \setminus \{0\}.
\end{align}
 
For a given element 
$\vec{m}=(m_i)_{i\in V(Q)} \in \Gamma_{Q, >0}$,  
let $V_i$ be $\mathbb{C}$-vector spaces with 
dimension $m_i$. 
Let us set 
\begin{align}\label{def:G}
G \cneq \prod_{i \in V(Q)} \GL(V_i), \ 
\mathrm{Rep}_Q(\vec{m}) \cneq \prod_{e \in E(V)} \Hom(V_{s(e)}, V_{t(e)}).
\end{align}
The algebraic group $G$ acts on $\mathrm{Rep}_Q(\vec{m})$ by 
\begin{align}\label{G:act}
g \cdot u=\{g_{t(e)}^{-1} \circ u_e \circ g_{s(e)}\}_{e\in E(Q)}
\end{align}
for $g=(g_i)_{i \in V(Q)} \in G$ 
and $u=(u_e)_{e\in E(Q)}$. 
A $Q$-representation with dimension vector $\vec{m}$ is 
determined by a point in $\mathrm{Rep}_Q(\vec{m})$
up to $G$-action. 
The moduli stack of $Q$-representations with 
dimension vector $\vec{m}$ is given by the 
quotient stack 
\begin{align}\label{stack:repQ}
\mM_{Q}(\vec{m}) \cneq \left[ \mathrm{Rep}_Q(\vec{m})/G \right]. 
\end{align}
We have the natural morphism to the GIT quotient
\begin{align}\label{mor:coarse}
p_Q \colon \mM_Q(\vec{m}) \to 
M_{Q}(\vec{m}) \cneq \mathrm{Rep}_Q(\vec{m}) \sslash G. 
\end{align}
Here in general, if a reductive algebraic group $G$ acts on 
an affine scheme $Y=\Spec R$, its GIT quotient is given by
\begin{align*}
Y\sslash G \cneq \Spec (R^G).
\end{align*}
A closed point of $M_{Q}(\vec{m})$ corresponds to a
semi-simple $Q$-representation, i.e. a
direct sum of simple $Q$-representations, 
and $p_Q$ sends a $Q$-representation to its 
semi-simplification. 
We have the commutative diagram
\begin{align}\label{dia:quiver}
\xymatrix{
\mathrm{Rep}_Q(\vec{m}) \ar[r] \ar[rd]_{\pi_Q} & \mM_{Q}(\vec{m}) \ar[d]^{p_Q} \\
& M_{Q}(\vec{m}). 
}
\end{align}

The point $0 \in \mathrm{Rep}_Q(\vec{m})$ and its 
image $0 \in M_Q(\vec{m})$ by the map (\ref{mor:coarse})
correspond to the semi-simple $Q$-representation 
\begin{align*}
\bigoplus_{i\in V(Q)}V_i \otimes S_i.
\end{align*}
A $Q$-representation (\ref{rep:Q})
is called \textit{nilpotent} if any sufficiently large number of 
compositions of the linear maps $u_e$ becomes zero. 
It is easy to see that 
a $Q$-representation is nilpotent if and only if it is 
an iterated extensions of simple objects 
$\{S_i\}_{i \in V(Q)}$. 
In particular, 
the fiber 
\begin{align*}
p_Q^{-1}(0) \subset \mM_Q(\vec{m})
\end{align*}
 for the morphism (\ref{mor:coarse})
consists of nilpotent $Q$-representations
with dimension vector $\vec{m}$. 
The morphism (\ref{mor:coarse})
is a good moduli space
of the stack $\mM_Q(\vec{m})$
(see~Definition~\ref{def:goodmoduli}).

\subsection{Semistable quiver representations}
For a quiver $Q$, let $K(Q)$ be the Grothendieck group of the 
abelian category of 
finite dimensional 
$Q$-representations. 
Let 
$\hH \subset \mathbb{C}$ be the upper half plane, and take 
\begin{align}\label{xi}
\xi=(\xi_i)_{i \in V(Q)} \in \hH^{\sharp V(Q)}, \ \xi_i \in \hH.  
\end{align}
Then we have the group homomorphism 
\begin{align}\label{K:dim}
Z_{\xi} \colon K(Q) \stackrel{\mathbf{dim}}{\to} \Gamma_Q \to
 \mathbb{C}, \ 
[S_i] \mapsto \xi_i. 
\end{align}
Here $\mathbf{dim}$ is the map taking the dimension vectors
of $Q$-representations. 
Then $Z_{\xi}$ defines a Bridgeland stability 
condition on the abelian 
category of finite dimensional $Q$-representations
(see~Appendix~\ref{sec:Bridgeland}). 
For simplicity, we call
$Z_{\xi}$-(semi)stable objects as 
$\xi$-(semi)stable objects. 

For a choice of $\xi$ as in (\ref{xi}), let 
\begin{align}\label{im:rep}
\mathrm{Rep}_Q^{\xi}(\vec{m}) \subset \mathrm{Rep}_Q(\vec{m})
\end{align}
be the Zariski 
open locus consisting of $\xi$-semistable 
$Q$-representations. 
We 
take the associated GIT quotients: 
\begin{align*}
\mM_{Q}^{\xi}(\vec{m}) \cneq [\mathrm{Rep}_Q^{\xi}(\vec{m})/G], \ 
M_{Q}^{\xi}(\vec{m}) \cneq \mathrm{Rep}_Q^{\xi}(\vec{m})\sslash G.
\end{align*}
We have the commutative diagram
\begin{align}\label{com:MQ}
\xymatrix{
\mM_{Q}^{\xi}(\vec{m}) \ar@<-0.3ex>@{^{(}->}[r]^{j_Q^{\xi}} \ar[d]_-{p_Q^{\xi}}
\ar[rd]^-{r_Q^{\xi}} & 
\mM_{Q}(\vec{m}) \ar[d]^-{p_Q} \\
M_Q^{\xi}(\vec{m}) \ar[r]_-{q^{\xi}_Q} & M_Q(\vec{m}).
}
\end{align}
Here $j_Q^{\xi}$ is an open immersion, 
$p_Q$, $p_Q^{\xi}$ are natural morphisms to the 
good moduli spaces, and $q_Q^{\xi}$ is the morphism 
induced by the open immersion (\ref{im:rep}). 
By a general theory of GIT, the morphism $q_Q^{\xi}$ is a projective 
morphism of irreducible varieties
satisfying $q_{Q \ast}^{\xi}\oO_{M_Q^{\xi}(\vec{m})}=\oO_{M_Q(\vec{m})}$. 

Let 
\begin{align*}
M_Q^{s}(\vec{m}) \subset M_Q(\vec{m}), \ 
M_{Q}^{\xi, s}(\vec{m}) \subset M_{Q}^{\xi}(\vec{m})
\end{align*}
be the open subsets of 
simple part, $\xi$-stable part 
respectively. 
It is well-known (for example see~\cite{MR2484736})
that 
both of $M_Q^{s}(\vec{m})$, $M_Q^{\xi, s}(\vec{m})$
are smooth varieties. 
As any preimage $(q_Q^{\xi})^{-1}(x)$
for $x \in M_Q^s(\vec{m})$ is a one point, 
the morphism $q_Q^{\xi}$ is a projective birational 
morphism if $M_Q^s(\vec{m}) \neq \emptyset$.

\subsection{Quivers with convergent super-potentials}\label{subsec:conv}
For a quiver $Q$, by 
taking the completion of 
the path algebra $\mathbb{C}[Q]$ with respect to the 
length of the path, 
we obtain the formal 
path algebra:
\begin{align*}
\mathbb{C}\lkakko Q \rkakko \cneq 
\prod_{n\ge 0}
\bigoplus_{e_1, \ldots, e_n \in E(Q), t(e_i)=s(e_{i+1})} \mathbb{C} \cdot e_1 e_2 \ldots e_n.
\end{align*}
Note that an element $f \in \mathbb{C}\lkakko Q \rkakko$
is written as 
\begin{align}\label{f:element}
f=\sum_{n\ge 0, \{1, \ldots, n+1\} \stackrel{\psi}{\to} V(Q)}
\sum_{e_i \in E_{\psi(i), \psi(i+1)}}
a_{\psi, e_{\bullet}} \cdot e_1 e_2\ldots e_{n}. 
\end{align}
Here 
$a_{\psi, e_{\bullet}} \in \mathbb{C}$, 
$e_{\bullet}=(e_1, \ldots, e_n)$ and 
$E_{\psi(i), \psi(i+1)}$ is defined as in (\ref{Eab}). 
The above element $f$ lies in $\mathbb{C}[Q]$ if and only if
$a_{\psi, e_{\bullet}}=0$ for $n\gg 0$. 
\begin{defi}\label{def:CQ}
The subalgebra
\begin{align*}
\mathbb{C}\{ Q\} \subset \mathbb{C}\lkakko Q \rkakko
\end{align*}
is defined 
to be elements (\ref{f:element}) 
such that $\lvert a_{\psi, e_{\bullet}} \rvert <C^n$ for 
some constant $C>0$ which is independent of $n$. 
\end{defi}
Note that $\mathbb{C}\{Q\}$ contains $\mathbb{C}[Q]$ as 
a subalgebra. 
A \textit{convergent super-potential} of a quiver $Q$ is an element 
\begin{align*}
W \in \mathbb{C}\{ Q \}/[\mathbb{C}\{ Q \}, \mathbb{C}\{ Q \}]. 
\end{align*}
A convergent super-potential $W$ of $Q$ is represented by 
a formal sum
\begin{align}\label{form:W}
W=\sum_{n\ge 1}
\sum_{\begin{subarray}{c}
\{1, \ldots, n+1\} \stackrel{\psi}{\to} V(Q), \\
\psi(n+1)=\psi(1)
\end{subarray}}
\sum_{e_i \in E_{\psi(i), \psi(i+1)}}
a_{\psi, e_{\bullet}} \cdot e_1 e_2\ldots e_{n}
\end{align}
with $\lvert a_{\psi, e_{\bullet}} \rvert <C^n$
for a constant $C>0$. 
The above $W$ is called \textit{minimal}
if $a_{\psi, e_{\bullet}}=0$ for 
$e_{\bullet}=(e_1, \ldots, e_n)$ with $n\le 2$. 

For a dimension vector $\vec{m}$
of $Q$, let $\tr W$ be the formal function 
of $u=(u_e)_{e\in E(Q)} \in \mathrm{Rep}_Q(\vec{m})$ defined by
\begin{align*}
\tr W(u) \cneq \sum_{n\ge 1}
\sum_{\begin{subarray}{c}
\{1, \ldots, n+1\} \stackrel{\psi}{\to} V(Q) \\
\psi(n+1)=\psi(1)
\end{subarray}}
\sum_{e_i \in E_{\psi(i), \psi(i+1)}}
a_{\psi, e_{\bullet}} \cdot \tr(u_n \circ u_{n-1} 
\circ \cdots \circ u_1).
\end{align*}
The above formal function on 
$\mathrm{Rep}_Q(\vec{m})$ is $G$-invariant. 
By the argument of~\cite[Lemma~2.10]{Todstack}
and~\cite[Lemma~4.9]{TodGV}, 
there is an 
analytic open neighborhood $V$ 
and an analytic function $\overline{\tr}(W)$
\begin{align}\label{V:open}
0 \in V \subset M_Q(\vec{m}), \ 
\overline{\tr}(W) \colon V \to \mathbb{C}
\end{align}
such that the formal function 
$\tr W$
absolutely converges on 
$\pi_Q^{-1}(V)$
(here $\pi_Q$ is given in the diagram (\ref{dia:quiver})) to give a 
$G$-invariant analytic function, which factors 
through $\pi_Q^{-1}(V) \to V$: 
\begin{align}\label{tr:W}
\tr W \colon \pi_Q^{-1}(V) \stackrel{\pi_Q}{\to} V
\stackrel{\overline{\tr}(W)}{\to}
\mathbb{C}. 
\end{align}
Then we set
\begin{align}\label{def:MW}
\mathrm{Rep}_{(Q, \partial W)}(\vec{m})|_{V} &\cneq 
\{ d(\tr W)=0\} \subset \pi_Q^{-1}(V), \\ 
\notag
\mM_{(Q, \partial W)}(\vec{m})|_{V}&
 \cneq \left[\{ d(\tr W)=0\}/G \right] \subset [\pi_Q^{-1}(V)/G], \\
\notag
M_{(Q, \partial W)}(\vec{m})|_{V}& \cneq \{ d(\tr W)=0\} \sslash G
\subset V. 
\end{align}
Here $(-)\sslash G$ above is an analytic Hilbert quotient 
(see~\cite{MR1631577, MR3394374, Todstack}).

\begin{rmk}\label{rmk:pW}
For a $Q$-representation corresponding to 
a point in $\pi_Q^{-1}(V)$, it satisfies the 
equation $\{d(\tr W)=0\}$ if and only if it satisfies the relation 
$\partial W$ of the quiver $Q$ given by derivations
of $W$ (see~\cite[Subsection~2.6]{Todstack}). 
\end{rmk}

Let $\xi$ be data as in (\ref{xi}) which defines the 
$\xi$-stability on the category of $Q$-representations, and 
\begin{align}\label{rep:xi}
\mathrm{Rep}_{(Q, \partial W)}^{\xi}(\vec{m})|_{V}
\subset \mathrm{Rep}_{(Q, \partial W)}(\vec{m})|_{V}
\end{align}
be the open locus consisting of $\xi$-semistable 
$Q$-representations. Similarly to (\ref{def:MW}), we define
\begin{align*}
&\mM_{(Q, \partial W)}^{\xi}(\vec{m})|_{V}
\cneq [\mathrm{Rep}_{(Q, \partial W)}^{\xi}(\vec{m})|_{V}/G], \\ 
&M_{(Q, \partial W)}^{\xi}(\vec{m})|_{V}
\cneq \mathrm{Rep}_{(Q, \partial W)}^{\xi}(\vec{m})|_{V}\sslash G. 
\end{align*}
Then we have the commutative diagram
\begin{align}\label{dia:arrow}
\xymatrix{
M_{(Q, \partial W)}^{\xi}(\vec{m})|_{V} \ar[d]_-{q_{(Q, \partial W)}^{\xi}} 
\ar@<-0.3ex>@{^{(}->}[r] & (q_{Q}^{\xi})^{-1}(V) \ar[d]_-{q_{Q}^{\xi}}
 \ar[rd]^-{\tr^{\xi} W}
&  \\
M_{(Q, \partial W)}(\vec{m})|_{V}  
\ar@<-0.3ex>@{^{(}->}[r] & V \ar[r]_{\overline{\tr} W} & \mathbb{C}.
}
\end{align}
Here $\hookrightarrow$ are closed embeddings (see~\cite[Lemma~2.9]{Todstack}), 
$q_{Q}^{\xi}$ is given in the diagram (\ref{com:MQ}), 
$q_{(Q, \partial W)}^{\xi}$ is induced by the 
open immersion 
(\ref{rep:xi}), and
$\tr^{\xi}W$ is defined by the commutativity 
of (\ref{dia:arrow}). 

\begin{lem}\label{MW:dcrit}
Suppose that $M_{Q}^{\xi, s}(\vec{m})=M_Q^{\xi}(\vec{m})$
holds. 
Then there is a d-critical structure on 
$M_{(Q, \partial W)}^{\xi}(\vec{m})|_{V}$
given by 
\begin{align*}
s=\tr^{\xi}W+(d \tr^{\xi}W)^2 
\in \Gamma(\sS_{M_{(Q, \partial W)}^{\xi}(\vec{m})|_{V}}^0)
\end{align*}
such that the diagram (\ref{dia:arrow})
is its $q_{(Q, \partial W)}^{\xi}$-relative 
d-critical chart.
\end{lem}
\begin{proof}
The
assumption implies that 
$\mM_{(Q, \partial W)}^{\xi}(\vec{m})|_{V}$
and
$(r_{Q}^{\xi})^{-1}(V)$
(see (\ref{com:MQ}) for the definition of the 
map $r_Q^{\xi}$)
are $\mathbb{C}^{\ast}$-gerbes
over $M_{(Q, \partial W)}^{\xi}(\vec{m})|_{V}$
and
$(q_{Q}^{\xi})^{-1}(V)$
respectively. 
Since $\mM_{(Q, \partial W)}^{\xi}(\vec{m})|_{V}$
is a critical locus of the function 
$\tr W|_{(r_{Q}^{\xi})^{-1}(V)}$, 
it follows that 
$M_{(Q, \partial W)}^{\xi}(\vec{m})|_{V}$
is a critical locus of the function 
$\tr^{\xi}W$
on the smooth space 
$(q_Q^{\xi})^{-1}(V)$. 
Therefore the lemma holds. 
\end{proof}

\subsection{Wall-crossing of representations of quivers}
Suppose that $\vec{m} \in \Gamma_{Q, >0}$ is primitive, 
i.e. it is not divided by 
a positive integer greater than one. 
For each decomposition $\vec{m}=\vec{m}_1+\vec{m}_2$
for $\vec{m}_i \in \Gamma_{Q, >0}$, 
we set
\begin{align}\label{wall:quiver}
\wW_{(\vec{m}_1, \vec{m}_2)}
\cneq \left\{ \xi \in \hH^{\sharp V(Q)} : 
Z_{\xi}(\vec{m}_1) \in \mathbb{R}_{>0} 
Z_{\xi}(\vec{m}_2)\right\}. 
\end{align}
As $\vec{m}$ is primitive, 
the vectors $\vec{m}_1$, $\vec{m}_2$ are not proportional 
in $(\Gamma_{Q})_{\mathbb{Q}}$. 
Therefore $\wW_{(\vec{m}_1, \vec{m}_2)}$ is a real 
codimension one submanifold of $\hH^{\sharp V(Q)}$. 
Since the possible decomposition 
$\vec{m}=\vec{m}_1+\vec{m}_2$ is finite, 
the above real codimension one submanifolds are 
finite. 

The submanifolds (\ref{wall:quiver}) are the 
set of walls of $\xi$-stability conditions on 
the category of $Q$-representations, 
and chambers are connected components 
of the complement of walls in $\hH^{\sharp V(Q)}$. 
As is well-known, we have the wall-crossing phenomena: 
the moduli spaces $M_Q^{\xi}(\vec{m})$ are constant 
if $\xi$ lies on a chamber, but may change if 
$\xi$ crosses a wall. 
Note that as $\vec{m}$ is primitive, we 
have $M_Q^{\xi, s}(\vec{m})=M_Q^{\xi}(\vec{m})$
if $\xi$ lies on a chamber, hence it is a 
smooth variety. 

Let $W$ be a convergent super-potential of $Q$, and 
take $\xi^{\pm} \in \hH^{\sharp V(Q)}$ which 
lie on chambers. 
Then there is an analytic open 
subset $0 \in V \subset M_Q(\vec{m})$
and the diagram (see the diagram (\ref{dia:arrow}))
\begin{align}\label{dia:qflip}
\xymatrix{
M_{(Q, \partial W)}^{\xi^+}(\vec{m})|_{V} \ar[rd]_-{q_{(Q, \partial W)}^{\xi^+}} & & 
M_{(Q, \partial W)}^{\xi^-}(\vec{m})|_{V} 
\ar[ld]^-{q_{(Q, \partial W)}^{\xi^-}} 
\\
& M_{(Q, \partial W)}(\vec{m})|_{V}. &
}
\end{align}
By Lemma~\ref{MW:dcrit}, the following corollary 
immediately follows: 
\begin{cor}\label{cor:dia:qflip}
The diagram (\ref{dia:qflip}) 
is an analytic (generalized) d-critical flip, 
flop, MFS if the diagram
\begin{align}\label{dia:qflip2}
\xymatrix{
M_Q^{\xi^+}(\vec{m}) \ar[rd]_-{q_{Q}^{\xi^+}} & & 
\ar[ld]^-{q_{Q}^{\xi^-}}
M_Q^{\xi^-}(\vec{m}) \\
& M_Q(\vec{m}) &
}
\end{align}
is a (generalized) flip, flop, MFS, respectively. 
\end{cor}

\begin{rmk}\label{rmk:MFS:quiver}
If 
$M_Q^s(\vec{m}) \neq \emptyset$, 
which follows if the condition in Theorem~\ref{thm:criterion}
is satisfied, then the diagram (\ref{dia:qflip2})
is a birational map of smooth varieties
$M_Q^{\xi^+}(\vec{m}) \dashrightarrow 
M_Q^{\xi^-}(\vec{m})$. 
Therefore the diagram (\ref{dia:qflip2}) is 
a (generalized) MFS only if 
$M_Q^{s}(\vec{m}) = \emptyset$ holds. 
\end{rmk}

\section{Analytic neighborhood theorem of wall-crossing in CY 3-folds}
\label{sec:anatheorem}
In this section, 
we give an analytic neighborhood theorem
for wall-crossing diagrams in CY 3-folds. 
This theorem describes 
above diagrams 
in term of wall-crossing diagrams for 
quivers with convergent super-potentials. 
A similar result was already proved in~\cite{Todstack} 
in the case of moduli spaces of semistable sheaves.
We will see that 
the same argument applies 
to the case of Bridgeland semistable objects, 
if we assume the existence of their
good moduli spaces (see Theorem~\ref{thm:goodmoduli}). 
In this section, we always assume that 
$X$ is a smooth projective CY 3-fold.

\subsection{Ext-quivers}\label{subsec:Ext}
For a smooth projective CY 3-fold $X$, 
let us take a collection of objects in the 
derived category
\begin{align}\notag
E_{\bullet}=(E_1, E_2, \ldots, E_k), \ 
E_i \in D^b(X). 
\end{align}
Here we recall the 
notion of Ext-quiver associated with the 
collection $E_{\bullet}$
and the 
construction of their
convergent super-potentials. 
We will apply the construction here to 
the collection of objects $E_{\bullet}$
associated with polystable objects. 

For each $1\le i, j \le k$, 
we fix a finite subset
\begin{align}\notag
E_{i, j} \subset \Ext^1(E_i, E_j)^{\vee}
\end{align}
giving a basis of $\Ext^1(E_i, E_j)^{\vee}$. 
Let the quiver $Q_{E_{\bullet}}$ 
be defined as follows. 
The set of vertices and edges are given by 
\begin{align*}
V(Q_{E_{\bullet}}) \cneq \{1, 2, \ldots, k\}, \ 
E(Q_{E_{\bullet}}) \cneq \coprod_{1\le i, j \le k}
E_{i, j}. 
\end{align*}
The maps $s, t \colon 
E(Q_{E_{\bullet}}) \to V(Q_{E_{\bullet}})$
are given by 
\begin{align*}
s|_{E_{i, j}} \cneq i, \ t|_{E_{i, j}} \cneq j. 
\end{align*}
The quiver $Q_{E_{\bullet}}$ is called the 
\textit{Ext quiver} associated with the collection $E_{\bullet}$. 
Note that we have $\mathbb{E}_{i, j}=\Ext^1(E_i, E_j)^{\vee}$,  
where $\mathbb{E}_{i, j}$ is defined as in (\ref{Eab}). 

For a map of sets
\begin{align*}
\psi \colon \{1, \ldots, n+1\} \to \{1, \ldots, k\}
\end{align*}
let
$m_n$ be the graded linear maps
\begin{align}\notag
m_n \colon 
\Ext^{\ast}(E_{\psi(1)}, E_{\psi(2)}) \otimes 
&\Ext^{\ast}(E_{\psi(2)}, E_{\psi(3)}) \otimes
\cdots \\
\cdots \otimes 
&\Ext^{\ast}(E_{\psi(n)}, E_{\psi(n+1)})  
\label{m_n}
\to \Ext^{\ast+2-n}(E_{\psi(1)}, E_{\psi(n+1)})
\end{align}
which give a minimal $A_{\infty}$-category structure 
on the dg-category generated by $(E_1, \ldots, E_k)$. 
Since $X$ is a CY 3-fold, we can take 
the above 
$A_{\infty}$-structure (\ref{m_n})
to be cyclic (see~\cite{MR1876072}), i.e. 
for a map 
$\psi$ as above with 
$\psi(1)=\psi(n+1)$, and elements
\begin{align*}
a_i \in \Ext^{\ast}(E_{\psi(i)}, E_{\psi(i+1)}), \ 1\le i\le  n
\end{align*}
we have the relation
\begin{align}\notag
(m_{n-1}(a_1, \ldots, a_{n-1}), a_{n})=(m_{n-1}(a_2, \ldots, a_{n}), a_1). 
\end{align}
Here
$(-, -)$ is the Serre duality pairing
\begin{align}\notag
(-, -) \colon 
\Ext^{\ast}(E_a, E_b) \times \Ext^{3-\ast}(E_b, E_a)
\to \Ext^3(E_a, E_a) \stackrel{\int_X \mathrm{tr}}{\to} \mathbb{C}.
\end{align} 
Let $W_{E_{\bullet}} \in \mathbb{C}\lkakko Q_{E_{\bullet}} \rkakko$
be defined by
\begin{align}\label{def:WE}
W_{E_{\bullet}} \cneq \sum_{n\ge 3}
\sum_{\begin{subarray}{c} 
\{1, \ldots, n+1 \} \stackrel{\psi}{\to} \{1, \ldots, k\} \\
\psi(1)=\psi(n+1)
\end{subarray}}
\sum_{e_i \in E_{\psi(i), \psi(i+1)}}
a_{\psi, e_{\bullet}} \cdot e_1 e_2 \ldots e_n. 
\end{align}
Here the coefficient $a_{\psi, e_{\bullet}}$ is given by 
\begin{align}\notag
a_{\psi, e_{\bullet}} \cneq 
\frac{1}{n}(m_{n-1}(e_1^{\vee}, e_2^{\vee}, \ldots, e_{n-1}^{\vee}), 
e_n^{\vee}). 
\end{align}
In the RHS, 
for $e \in E_{i, j}$
the element $e^{\vee} \in \Ext^1(E_i, E_j)$
is determined by 
$e^{\vee}(e)=1$ and 
$e^{\vee}(e')=0$ for any $e'  \in E_{i, j}$
with $e \neq e'$. 
By taking the Dolbeaut model in defining the $A_{\infty}$-products (\ref{m_n}), the result of~\cite[Lemma~4.1]{Todstack} (based on earlier 
works~\cite{MR1950958, JuTu})
shows that
\begin{align*}
W_{E_{\bullet}} \in \mathbb{C}\{ Q_{E_{\bullet}} \} \subset \mathbb{C}\lkakko Q_{E_{\bullet}} \rkakko .
\end{align*}
Here
$\mathbb{C}\{ Q_{E_{\bullet}} \}$
is given in Definition~\ref{def:CQ}. 
Therefore $W_{E_{\bullet}}$ determines 
a convergent super-potential of $Q_{E_{\bullet}}$.

A collection $E_{\bullet}$ is called
a \textit{simple collection}
if we have 
\begin{align}\label{Ext:simple}
\Ext^{\le 0}(E_i, E_j)=\mathbb{C} \cdot \delta_{ij}.
\end{align}
Here the RHS is 
concentrated on degree zero. 
In this case, the 
algebra $\mathbb{C} \lkakko Q_{E_{\bullet}} \rkakko/ 
(\partial W_{E_{\bullet}})$
gives a pro-representable hull of 
NC deformation functor associated with $E_{\bullet}$, 
developed in~\cite{Lau, Erik, Kawnc, BoBo}. 
By taking the tensor product with 
the universal object, we have an equivalence of 
categories (see~\cite[Corollary~6.7]{Todstack})
\begin{align}\label{equiv:NC}
\Phi_{E_{\bullet}} \colon 
\modu_{\rm{nil}}\mathbb{C} \lkakko Q_{E_{\bullet}} \rkakko/ 
(\partial W_{E_{\bullet}}) \stackrel{\sim}{\to}
\langle E_1, \ldots, E_k\rangle_{\rm{ex}}.
\end{align}
Here 
the LHS is the category of 
nilpotent $\mathbb{C} \lkakko Q_{E_{\bullet}} \rkakko/ 
(\partial W_{E_{\bullet}})$-modules, 
and $\langle -\rangle_{\rm{ex}}$ in the RHS 
is the extension closure. 
The above equivalence sends 
simple objects $S_i$ for $1\le i\le k$
corresponding to the vertex $i \in V(Q_{E_{\bullet}})$
to the object $E_i$. 

\subsection{Analytic neighborhood theorem}
We return to the situation in 
Section~\ref{sec:moduliCY3}. 
Namely we take
stability conditions
\begin{align}\label{take:stab}
\sigma=(Z, \aA) \in \Stab(X), \ 
\sigma^{\pm}=(Z^{\pm}, \aA^{\pm}) \in \Stab(X)
\end{align}
where 
$\sigma^{\pm}$ lie on adjacent chambers
whose closures contain $\sigma$.  
We also take 
a primitive element $v \in \Gamma_X$
with $\Im Z(v)>0$, 
the good moduli spaces of
semistable objects
$M_{\sigma}(v)$, $M_{\sigma^{\pm}}(v)$
and  
the wall-crossing diagram 
(\ref{wall:diagram}). 
As we mentioned in Subsection~\ref{subsec:msstable}, a 
closed point $p \in M_{\sigma}(v)$ corresponds to a 
$\sigma$-polystable object
$E \in \aA$. 
 The object $E$ is of the form
\begin{align}\label{polystable:E}
E=
\bigoplus_{i=1}^k V_i \otimes E_i
\end{align}
where $V_i$ is a finite dimensional vector space, each 
$E_i \in \aA$ is a $\sigma$-stable object
with $\arg Z(E_i)=\arg Z(E_j)$
and
$E_i \not\cong E_j$ for $i\neq j$, 
satisfying 
\begin{align}\label{sum:dimVi}
\sum_{i=1}^k \dim V_i \cdot \ch(E_i)=v.
\end{align}
Let $E_{\bullet}$ be the collection of 
$\sigma$-stable objects in (\ref{polystable:E})
\begin{align}\label{collect:poly}
E_{\bullet}=(E_1, E_2, \ldots, E_k).
\end{align}
Then $E_{\bullet}$ is a simple collection, i.e. 
it satisfies the condition (\ref{Ext:simple}). 
Let $Q_{E_{\bullet}}$ be the Ext-quiver
associated with the collection $E_{\bullet}$. 
We take data $\xi^{\pm} \in \hH^{k}$ as in 
(\ref{xi}) for the 
Ext-quiver $Q_{E_{\bullet}}$ by 
\begin{align}\label{xi+}
\xi^{\pm} \cneq (\xi^{\pm}_i)_{1\le i\le k}, \ 
\xi^{\pm}_i = Z^{\pm}(E_i), \ 1\le i\le k. 
\end{align}
Then we have the associated 
$\xi^{\pm}$-stability condition on
the abelian category of 
$Q_{E_{\bullet}}$-representations. 
Let $\vec{m}=(m_i)_{1\le i\le k}$ be the dimension vector of $Q_{E_{\bullet}}$-representations given by
\begin{align}\label{dimvec:m}
m_i=\dim V_i, \ 1\le i\le k
\end{align}
where $V_i$ is given in (\ref{polystable}). 
Note that 
as $v \in \Gamma_X$ is primitive, 
the vector $\vec{m} \in \Gamma_{Q_{E_{\bullet}}}$ is 
also primitive by the identity (\ref{sum:dimVi}). 

The following \textit{analytic neighborhood theorem}
describes the 
wall-crossing diagram 
in terms of representations of quivers with 
convergent super-potentials, analytic locally 
on the good moduli spaces: 
\begin{thm}\label{thm:Equiver}
For a primitive element 
$v \in \Gamma_X$ and 
stability conditions $\sigma$, $\sigma^{\pm}$
as in (\ref{take:stab}), 
let 
\begin{align}\label{wall:diagram2}
\xymatrix{
M_{\sigma^{+}}(v) \ar[rd]_-{q_M^+} & & \ar[ld]^-{q_M^-} M_{\sigma^{-}}(v) \\
&  M_{\sigma}(v).  &
}
\end{align}
be the wall-crossing diagram as in (\ref{wall:diagram}). 
For a closed point 
$p\in M_{\sigma}(v)$ corresponding to a $\sigma$-polystable 
object (\ref{polystable:E}),
let $Q=Q_{E_{\bullet}}$ be the associated 
Ext-quiver
and $W=W_{E_{\bullet}}$ the 
convergent super-potential as in (\ref{def:WE}). 
Then there exist
 analytic 
open neighborhoods 
\begin{align*}
p \in T \subset M_{\sigma}(v), \ 
0 \in V \subset M_{Q}(\vec{m})
\end{align*}
where $\vec{m}$ is the dimension vector (\ref{dimvec:m}), 
such that we have the commutative diagram of isomorphisms 
\begin{align}\label{dia:local}
\xymatrix{
M^{\xi^{\pm}}_{(Q, \partial W)}(\vec{m})|_{V} \ar[r]^-{\cong} 
\ar[d]_-{q_{(Q, \partial W)}^{\xi^{\pm}}} &
(q_M^{\pm})^{-1}(T) \ar[d]^-{q_M^{\pm}} \\
M_{(Q, \partial W)}(\vec{m})|_{V} \ar[r]^-{\cong} & T. 
}
\end{align}
Here the left vertical arrow is given in (\ref{dia:arrow}), 
the right vertical arrow is given 
in (\ref{wall:diagram2}) pulled back to
$T$. Moreover the top isomorphism 
preserves d-critical structures, where 
the d-critical structure on the LHS is given in Lemma~\ref{MW:dcrit}
and that on the RHS is given in Theorem~\ref{thm:CYdcrit}.
\end{thm}
The above result is proved in~\cite[Theorem~7.7]{Todstack}
in the case of one dimensional (twisted) semistable sheaves, 
and the same argument also applies
if we assume the existence of good moduli spaces for
Bridgeland semistable objects 
(see Theorem~\ref{thm:goodmoduli}). 
In Subsection~\ref{subsec:outline}, we will 
 just give an outline of the proof. 
By Corollary~\ref{cor:dia:qflip}
 and Theorem~\ref{thm:Equiver}, we have the following: 
\begin{cor}\label{cor:analytic}
In the situation of Theorem~\ref{thm:Equiver}, the diagram 
(\ref{wall:diagram2})
is a d-critical (generalized) flip, flop, MFS 
at $p\in M_{\sigma}(v)$ 
(which corresponds to a polystable object (\ref{polystable:E}))
if the diagram
\begin{align}\label{dia:qflip3}
\xymatrix{
M_Q^{\xi^+}(\vec{m}) \ar[rd]_-{q_{Q}^{\xi^+}} & & 
\ar[ld]^-{q_{Q}^{\xi^-}}
M_Q^{\xi^-}(\vec{m}) \\
& M_Q(\vec{m}) &
}
\end{align}
for the Ext-quiver $Q=Q_{E_{\bullet}}$, 
dimension vector $\vec{m}$ in (\ref{dimvec:m})
and $\xi^{\pm}$ as in (\ref{xi+}) is 
a (generalized) flip, flop, MFS, respectively. 
\end{cor}

Using Corollary~\ref{cor:analytic}, 
we give the simplest case of d-critical flips
and
flops in the following example: 
\begin{exam}\label{subsec:dflop}
In the situation of Corollary~\ref{cor:analytic}, 
suppose that $p \in M_{\sigma}(v)$ 
corresponds to a
$\sigma$-polystable object $E$
of the form
\begin{align*}
E=E_1 \oplus E_2, \ E_1 \not\cong E_2
\end{align*}
 where 
$E_1$, $E_2$ are $\sigma$-stable 
with $\arg Z(E_1)=\arg Z(E_2)$. 
Note that we have
\begin{align*}
\chi(E_1, E_2) &\cneq \sum_{ i \in \mathbb{Z}}
(-1)^i \dim \Ext^i(E_1, E_2) \\
&=\dim \Ext^1(E_2, E_1)-\dim \Ext^1(E_1, E_2)
\end{align*}
and the LHS is computed by the Riemann-Roch theorem. 
The Ext-quiver $Q$ 
associated with $E_{\bullet}=(E_1, E_2)$
has 
two vertices, $V(Q)=\{1, 2\}$. 
Let us set
\begin{align*}
V^+ &=\Ext^1(E_1, E_2), \ 
V^-=\Ext^1(E_2, E_1), \\ 
U &=\Ext^1(E_1, E_1) \oplus \Ext^1(E_2, E_2).
\end{align*}
The stack of $Q$-representations of dimension vector 
$\vec{m}=(1, 1)$ is given by
\begin{align*}
\mM_Q(\vec{m})=
[(V^+ \times V^-)/(\mathbb{C}^{\ast})^2] \times U. 
\end{align*}
Here 
the action of 
$(t_1, t_2) \in (\mathbb{C}^{\ast})^2$ on 
$(\vec{x}, \vec{y}) \in V^+ \times V^-$ is given by 
\begin{align*}
(t_1, t_2)
\cdot (\vec{x}, \vec{y})=
(t_1 t_2^{-1} \cdot \vec{x}, t_1^{-1} t_2 \cdot \vec{y}).
\end{align*}
We take $\sigma^{\pm}$ in (\ref{take:stab})
so that 
\begin{align*}
\arg Z^{+}(E_1)>\arg Z^{+}(E_2), \ 
\arg Z^{-}(E_1)<\arg Z^{-}(E_2)
\end{align*}
hold. 
It is easy to see that 
for a point 
$(\vec{x}, \vec{y}, \vec{u}) \in  
\mM_Q(\vec{m})$, 
it is 
$\xi^+$-(semi)stable if and only if $\vec{x} \neq 0$, 
$\xi^-$-(semi)stable if and only if $\vec{y} \neq 0$. 
Therefore the moduli spaces of $\xi^{\pm}$-stable 
$Q$-representations are given by
\begin{align*}
&M_Q^{\xi^{+}}(\vec{m})=
\mathrm{Tot}_{\mathbb{P}(V^+)}(\oO_{\mathbb{P}(V^+)}(-1) \otimes
V^-) \times U, \\ 
&M_Q^{\xi^{-}}(\vec{m})=
\mathrm{Tot}_{\mathbb{P}(V^-)}
(\oO_{\mathbb{P}(V^-)}(-1) \otimes
V^+) \times U.
\end{align*}
It follows that the diagram (\ref{dia:qflip3}) is a 
standard toric flip (resp.~flop) 
given in~Example~\ref{exam:toric}
if 
$\chi(E_1, E_2)<0$ (resp.~$\chi(E_1, E_2)=0$).
Therefore by Corollary~\ref{cor:analytic}, the  
the diagram (\ref{wall:diagram2}) is an analytic d-critical 
flip if $\chi(E_1, E_2)<0$ (resp.~flop if $\chi(E_1, E_2)=0$) 
at the point $p=[E_1 \oplus E_2] \in M_{\sigma}(v)$. 
\end{exam}

\subsection{Outline of the proof of Theorem~\ref{thm:Equiver}}
\label{subsec:outline}
\begin{proof}
Let us take a point $p \in M_{\sigma}(v)$ 
which corresponds to 
a $\sigma$-polystable object $E$ as in 
(\ref{polystable:E}). 
Note that for the Ext-quiver
$Q=Q_{E_{\bullet}}$
associated with the collection (\ref{collect:poly}), we have 
\begin{align*}
\mathrm{Rep}_Q(\vec{m})=\Ext^1(E, E), \ 
G \cneq \prod_{i=1}^k \GL(V_i)=\Aut(E)
\end{align*}
where $\vec{m}$ is the dimension vector (\ref{dimvec:m}). 
Under the above identifications, the 
$G$-action on $\mathrm{Rep}_Q(\vec{m})$ is compatible 
with the conjugate $\Aut(E)$-action on $\Ext^1(E, E)$. 

Let $\eE^{\bullet}_i \to E_i$ be a resolution of $E_i$ by vector bundles, 
and set 
$\eE^{\bullet}=\oplus_{i=1}^k V_i \otimes \eE_i^{\bullet}$. 
We have linear maps of degree $1-n$
\begin{align}\label{linear:I}
I_n \colon \Ext^{\ast}(E, E)^{\otimes n} \to 
\mathfrak{g}_{\eE_{\bullet}}^{\ast} \cneq
A^{0, \ast}(\hH om^{\ast}(\eE^{\bullet}, \eE^{\bullet}))
\end{align}
giving a $A_{\infty}$ quasi-isomorphism 
between the minimal 
$A_{\infty}$-algebra 
$\Ext^{\ast}(E, E)$ and
the Dolbeaut dg Lie algebra $\mathfrak{g}_{\eE_{\bullet}}^{\ast}$. 
 For $x \in \Ext^1(E, E)$, the formal fum
 \begin{align}\label{I:sum}
 I_{\ast}(x) \cneq \sum_{n\ge 1}I_n(x, \ldots, x)
 \end{align}
 has a convergent radius
 (see~\cite[Lemma~4.1]{Todstack}).
Let $\pi_Q \colon \mathrm{Rep}_Q(\vec{m}) \to M_Q(\vec{m})$ be the 
natural map to the GIT quotient,
and 
take a sufficiently small analytic open subset 
$0 \in V \subset M_Q(\vec{m})$. 
The infinite sum (\ref{I:sum})
determines an $\Aut(E)$-equivariant 
analytic map
\begin{align}\label{Iast0}
I_{\ast} \colon \pi_Q^{-1}(V) \to \widehat{\mathfrak{g}}_{\eE_{\bullet}}^1.
\end{align}
Here $\widehat{\mathfrak{g}}_{\eE_{\bullet}}^1$
is a certain Sobolev completion of 
$\mathfrak{g}_{\eE_{\bullet}}^1$
(see~\cite[Lemma~5.1]{Todstack}). 
On the open subset 
$\pi_Q^{-1}(V) \subset \Ext^1(E, E)$, the Mauer-Cartan 
locus of the minimal $A_{\infty}$-algebra
$\Ext^{\ast}(E, E)$
is given by the critical 
locus of 
the analytic function
\begin{align*}
\tr W \colon \pi_Q^{-1}(V) \to \mathbb{C}
\end{align*}
given
in (\ref{tr:W})
for 
the convergent super-potential 
$W=W_{E_{\bullet}}$ defined as in (\ref{def:WE}). 
Then the restriction of the map (\ref{Iast0}) 
to the Mauer-Cartan locus 
determines a smooth 
 morphism 
 of analytic stacks 
 of relative dimension zero (see~\cite[Proposition~4.3]{Todstack})
 \begin{align}\label{Iast}
 I_{\ast} \colon \mM_{(Q, \partial W)}(\vec{m})|_{V}. 
 \to \mM
 \end{align}
Here $\mM$ is the moduli stack of complexes given in 
Subsection~\ref{subsec:moduli:general}. 

 By shrinking $V$ if necessary, the 
image of the above morphism lies in 
the open substack $\mM_{\sigma}(v) \subset \mM$.
Let $p_M \colon \mM_{\sigma}(v) \to M_{\sigma}(v)$ be the 
morphism to the good moduli space. 
Then the argument of~\cite[Proposition~5.4]{Todstack}
shows that the map (\ref{Iast})
 induces the commutative diagram of isomorphisms 
\begin{align}\label{Iast:isom}
\xymatrix{
\mM_{(Q, \partial W)}(\vec{m})|_{V} \ar[r]_-{\cong}^-{I_{\ast}}
\ar[d]_-{p_{Q}} &
p_M^{-1}(T) \ar[d]^-{p_M} \\
M_{(Q, \partial W)}(\vec{m})|_{V} \ar[r]_-{\cong} & T
}
\end{align}
 for some analytic open neighborhood 
 $p \in T \subset M_{\sigma}(\beta, n)$. 
Here in \textit{loc.~cit.~}, 
the 
claim is stated for moduli stacks of semistable sheaves
and their good moduli spaces. 
But the argument can be generalized to the 
case of Bridgeland semistable objects, as 
the explicit construction of good moduli spaces
are not needed in the proof of \textit{loc.~cit.~}
The properties we used for the good moduli spaces are 
their existence and the \'etale slice theorem. 
For the Bridgeland semistable objects, 
the former is given in~\cite{AHLH} and the latter 
for the map $p_M \colon \mM_{\sigma}(v) \to M_{\sigma}(v)$
is given 
in~\cite{AHR}. 

 The diagram (\ref{Iast:isom})
in particular implies the isomorphism
\begin{align}\label{I:fiber}
I_{\ast} \colon p_{Q}^{-1}(0) \stackrel{\cong}{\to}
p_M^{-1}(p).
\end{align}
We show that 
the above isomorphism restricts to the isomorphisms
\begin{align}\label{I:restrict}
I_{\ast} \colon p_{Q}^{-1}(0) \cap 
\mM^{\xi^{\pm}}_{(Q, \partial W)}(\vec{m})|_{V} 
\stackrel{\cong}{\to} p_M^{-1}(p) \cap 
\mM_{\sigma^{\pm}}(v). 
\end{align}
Indeed if the above isomorphism holds, then the 
argument of~\cite[Theorem~6.8]{Todstack} shows 
that, after shrinking $V$, $T$ if necessary, we have the isomorphism
\begin{align*}
I_{\ast} \colon 
\mM^{\xi^{\pm}}_{(Q, \partial W)}(\vec{m})|_{V} 
\stackrel{\cong}{\to}
p_M^{-1}(T) \cap 
\mM_{\sigma^{\pm}}(v). 
\end{align*}
By taking the associated isomorphism on good moduli spaces, we obtain 
the desired diagram (\ref{isom:weak}). 
The comparison of d-critical structures follow from 
the argument of~\cite[Appendix~A]{TodGV}. 

We show the isomorphism (\ref{I:restrict}).
Note that $\mathbb{C}$-valued points of 
$p_{Q}^{-1}(0)$ consist
of nilpotent $Q$-representations with 
dimension vector $\vec{m}$, 
and those of $p_M^{-1}(p)$ consist of 
objects in the extension closure $\langle E_1, \ldots, E_k \rangle_{\rm{ex}}$
in $\aA$
with Chern character $v$. 
Then the isomorphism (\ref{I:restrict}) follows from 
Lemma~\ref{lem:category}
below. 
\end{proof}

\begin{lem}\label{lem:category}
The isomorphism 
(\ref{I:fiber}) is induced by 
the equivalence of categories 
given in (\ref{equiv:NC})
\begin{align}\label{Istar:eq}
\Phi_{E_{\bullet}} \colon 
\modu_{\rm{nil}}\mathbb{C}[[Q]]/(\partial W) \stackrel{\sim}{\to}
\langle E_1, \ldots, E_k \rangle_{\rm{ex}}.
\end{align}
Under the above equivalence, 
a nilpotent 
$Q$-representation
$\mathbb{V}$ is $\xi^{\pm}$-semistable 
if and only if $\Phi_{E_{\bullet}}(\mathbb{V})$ is 
$\sigma^{\pm}$-semistable in $\aA$. 
\end{lem}
\begin{proof}
The compatibility of $I_{\ast}$ with 
$\Phi_{E_{\bullet}}$ is due 
to~\cite[Theorem~6.8]{Todstack}, and the preservation 
of the stability follows from 
the argument of~\cite[Lemma~7.8]{Todstack}
without any modification.
\end{proof}

\section{Representations of symmetric and extended quivers}
\label{sec:repsym}
By Corollary~\ref{cor:analytic}, 
in order to see whether a given 
wall-crossing diagram in a CY 3-fold is a d-critical flip or flop, 
it is enough to 
see whether a wall-crossing diagram 
in the Ext-quiver
is a flip or flop. 
In this section, we 
study the latter problem in detail 
in the case of 
symmetric quivers and 
their extended version.
The results in this section will be applied to 
geometric situations in later sections. 

\subsection{Some general facts on representations of quivers}
Here 
we give some general facts on 
moduli spaces of representations of quivers. 
Below, 
we use the notation in Section~\ref{sec:quiver}. 
Let $Q$ be a quiver, 
$\vec{m} \in \Gamma_{Q, >0}$ a dimension vector
of $Q$, and $V_i$ for each $i \in V(Q)$
is a vector space with 
dimension $m_i$.  
As in (\ref{mor:coarse}), we have the 
moduli stack 
of 
$Q$-representations
with dimension vector $\vec{m}$
\begin{align}\label{stack:MQ}
\mM_Q(\vec{m})=\left[ \prod_{e \in E(Q)} \Hom(V_{s(e)}, V_{t(e)})
/G  \right], \ 
G=\prod_{i \in V(Q)} \GL(V_i)
\end{align}
and its 
good moduli space 
$\mM_Q(\vec{m}) \to M_Q(\vec{m})$. 
Let $\xi \in \hH^{\sharp V(Q)}$ be a choice 
of a stability condition 
as in (\ref{xi}), 
and $M_Q^{\xi}(\vec{m})$ the 
good moduli space of $\xi$-semistable representations.
As in (\ref{com:MQ}),
we have the natural morphism
\begin{align}\label{nat:q_Q}
q_Q^{\xi} \colon M_Q^{\xi}(\vec{m}) \to M_Q(\vec{m}).
\end{align} 
Let $M_Q^s(\vec{m}) \subset M_Q(\vec{m})$ be
the simple part. 
If $M_Q^s(\vec{m}) \neq \emptyset$, then 
the morphism (\ref{nat:q_Q}) is always birational. 

A criterion of the condition 
$M_Q^s(\vec{m}) \neq \emptyset$
is given in~\cite{MR958897}. 
In order to state this, we prepare some terminology. 
 A full subquiver $Q' \subset Q$ is 
said to be \textit{strongly connected}
if each
 couple from its vertex set belongs to an oriented cycle. 
For $\vec{m} \in \Gamma_Q$, let 
$\mathrm{supp}(\vec{m})$ be
the set of $i \in V(Q)$ with $m_i \neq 0$. 
Note that $\mathrm{supp}(\vec{m})$ is regarded as a full subquiver 
of $Q$. 
We also use the pairing 
on $\Gamma_Q$
defined by
\begin{align}\label{pairing}
\langle \vec{m}, \vec{m}'\rangle 
\cneq \sum_{i\in V(Q)} m_i \cdot m_i' -\sum_{e \in E(Q)}
m_{s(e)} \cdot m_{t(e)}'. 
\end{align}
We denote by $\widetilde{A}_n$ the 
extended Dynkin $A_n$-quiver, i.e. 
$Q(\widetilde{A}_n)=\{1, 2, \ldots, n\}$ 
with one arrow from 
$i$ to $i+1$ for each $1\le i\le n-1$ and 
$n$ to $1$. 
The result of~\cite{MR958897}
is stated as follows: 
\begin{thm}\emph{(\cite[Theorem~4]{MR958897})}\label{thm:criterion}
For $\vec{m} \in \Gamma_Q$, we have 
$M_Q^{s}(\vec{m}) \neq \emptyset$
if and only if 
either one of the following conditions hold: 
\begin{enumerate}
\item 
We have $m_i=1$ for all $i \in \mathrm{supp}(\vec{m})$ 
and 
$\mathrm{supp}(\vec{m})$ is a quiver of type 
$A_1$ or 
$\widetilde{A}_n$ for $n\ge 1$. 
\item
The quiver 
$\mathrm{supp}(\vec{m})$ is not of the above type, 
is strongly connected  
and 
\begin{align}\label{cond:simple}
\langle \vec{m}, \vec{i} \rangle 
\le 0, \ 
\langle \vec{i}, \vec{m} \rangle 
\le 0, \ 
\mbox{ for all }
i \in V(Q).
\end{align}
\end{enumerate}
\end{thm}

\begin{rmk}\label{rmk:simple}
Suppose that one of the conditions (\ref{cond:simple}) fails, e.g. 
\begin{align*}
\langle \vec{m}, \vec{i} \rangle=
m_i-\sum_{j \in V(Q), e \in E_{j, i}} m_j>0
\end{align*}
for some $i \in V(Q)$. 
Then for a $Q$-representation $\mathbb{V}$ as in (\ref{rep:Q}), 
the natural map
\begin{align*}
\sum_{j \in V(Q), e \in E_{j, i}}u_e \colon 
\bigoplus_{e \in E_{j, i}} V_j \to V_i
\end{align*}
has a non-trivial cokernel. 
Therefore we have a surjection 
$\mathbb{V} \twoheadrightarrow S_i$
and $M_Q^s(\vec{m})=\emptyset$ holds. 
Similarly if $\langle \vec{i}, \vec{m} \rangle >0$, 
then there is an injection $S_i \hookrightarrow \mathbb{V}$. 
These facts will be used later. 
\end{rmk}

Next, we consider canonical line 
bundles on moduli spaces of quiver representations. 
Suppose that $\vec{m}$ is primitive in 
$\Gamma_{Q}$, and let 
$M_Q^{\xi, s}(\vec{m}) \subset M_Q^{\xi}(\vec{m})$ be the 
$\xi$-stable part. 
Then by~\cite[Proposition~5.3]{Kin}, 
there exist universal $Q$-representations
\begin{align*}
\vV_i \to M_Q^{\xi, s}(\vec{m}), \ 
\mathbf{u}_e \colon \vV_{s(e)} \to \vV_{t(e)}, \ 
i \in V(Q), \ e \in E(Q).
\end{align*}
Here $\vV_i$ is a vector bundle on $M_Q^{\xi, s}(\vec{m})$ whose fiber 
is $V_i$, 
$\mathbf{u}_e$ is a map of vector bundles, 
and 
for a point $p \in M_Q^{\xi, s}(\vec{m})$ corresponding to 
a $Q$-representation (\ref{rep:Q}) 
we have $(\mathbf{u}_e)|_{p}=u_e$. 
In this case, 
we have the following lemma on the canonical 
line bundle of the smooth variety 
$M_Q^{\xi, s}(\vec{m})$:
\begin{lem}\label{lem:formula:K}
In the above situation,
we have 
\begin{align}\label{formula:K}
\omega_{M_Q^{\xi, s}(\vec{m})}=\bigotimes_{e \in E(Q)}\det \vV_{s(e)}
\otimes \det \vV_{t(e)}^{\vee}.  
\end{align}
\end{lem}
\begin{proof}
When $\vec{m}$ is primitive, 
the $\xi$-stable part of 
the stack $\mM_Q(\vec{m})$
is a trivial $\mathbb{C}^{\ast}$-gerbe over $M_Q^{\xi, s}(\vec{m})$. 
By the description of the stack $\mM_Q(\vec{m})$ in (\ref{stack:MQ}), 
the 
canonical line bundle of
the stack $\mM_Q(\vec{m})$ 
is induced by the one dimensional $G$-representation $G \to \mathbb{C}^{\ast}$ 
given by 
\begin{align*}
g=(g_i)_{i \in V(Q)} \mapsto 
\prod_{e \in E(Q)} \det g_{s(e)} \cdot (\det g_{t(e)})^{-1}.
\end{align*}
Therefore the identity (\ref{formula:K}) holds. 
\end{proof}

\subsection{Flops via representations of symmetric quivers}
Here we investigate 
the morphism $q_{Q}^{\xi}$ in (\ref{nat:q_Q})
for a symmetric quiver $Q$, defined 
below: 
\begin{defi}\label{def:symmetric}
A quiver $Q$ is called \textit{symmetric} if 
$\sharp E_{i, j}=\sharp E_{j, i}$ for any 
$i, j \in V(Q)$. 
Here $E_{i, j}$ is defined as in (\ref{Eab}). 
\end{defi}
\begin{rmk}\label{rmk:symmetric}
A symmetric condition of a quiver $Q$ 
is equivalent to that 
the pairing (\ref{pairing}) is symmetric. 
\end{rmk}

Below for a symmetric quiver $Q$, we 
fix identifications $E_{i, j}=E_{j, i}$, so 
that $Q$ and $Q^{\vee}$ are identified. 
In particular for a $Q$-representation 
$\mathbb{V}$, its dual representation 
$\mathbb{V}^{\vee}$ given in (\ref{V:dual}) is also a
$Q$-representation. 
We have the following lemma on the 
non-emptiness of the moduli spaces of stable
$Q$-representations: 
\begin{lem}\label{lem:slop}
For a symmetric quiver $Q$
and $\vec{m} \in \Gamma_Q$, we have 
$M_Q^{\xi, s}(\vec{m}) \neq \emptyset$
for some $\xi \in \hH^{\sharp V(Q)}$ 
if and only if $M_Q^{\xi, s}(\vec{m}) \neq \emptyset$
for any $\xi \in \hH^{\sharp V(Q)}$.  
\end{lem}
\begin{proof}
It is enough to show that 
if $M_Q^{\xi, s}(\vec{m}) \neq \emptyset$ 
for some $\xi$, 
then $M_Q^{s}(\vec{m}) \neq \emptyset$. 
Suppose that $M_Q^{\xi, s}(\vec{m}) \neq \emptyset$. 
We apply the criterion in Theorem~\ref{thm:criterion}
to show that $M_Q^s(\vec{m}) \neq \emptyset$. 
Since $\mathrm{supp}(\vec{m})$ is symmetric, 
it is of type $A_1$ or $\widetilde{A}_n$
only if $n=1$ or $n=2$. 
In $A_1$ and $\widetilde{A}_1$ case, 
we have $M_Q^{\xi, s}(\vec{m})=M_Q^{s}(\vec{m})$
for any dimension vector $\vec{m}$, 
and the statement is obvious. 

Suppose that $\mathrm{supp}(\vec{m})$ is 
$\widetilde{A}_2$, and write its 
vertices as $\{1, 2\}$. 
We may assume
that $\arg Z_{\xi}(S_1)>\arg Z_{\xi}(S_2)$, 
where $Z_{\xi}$ is given by (\ref{K:dim}). 
Let us write a $Q$-representation 
corresponding to a point
in $M_Q^{\xi, s}(\vec{m})$
as
\begin{align}\notag
\mathbb{V}=
\left(\xymatrix{
V_1 \ar@<0.5ex>[r]^-{e_{12}}& V_2 \ar@<0.5ex>[l]^-{e_{21}}
}\right)
\end{align}
where $V_i$ are vector spaces with dimension $m_i$
and $e_{12}$, $e_{21}$ are linear maps. 
The $\xi$-stability implies that
$\Hom(S_1, \mathbb{V})=0$ and $\Hom(\mathbb{V}, S_2)=0$.
These conditions imply that $e_{12}$ is an isomorphism, 
so we can assume $V_1=V_2=V$ and $e_{12}=\id$. 
Let $v \in V$ be an eigen vector of $e_{21}$
with eigen value $\lambda$. 
Then we have an injection of $Q$-representations
\begin{align*}
\left(\xymatrix{
\mathbb{C}v \ar@<0.5ex>[r]^-{\id}& \mathbb{C}v
 \ar@<0.5ex>[l]^-{\lambda \cdot}
}\right)
\hookrightarrow \mathbb{V}.
\end{align*}
As $\mathbb{V}$ is $\xi$-stable, 
the above injection must be an isomorphism. 
Therefore we have $m_1=m_2=1$
and $M_Q^{\xi, s}(\vec{m}) \neq \emptyset$
follows from Theorem~\ref{thm:criterion}. 

Suppose that $\mathrm{supp}(\vec{m})$ is not of the above types. 
If $\mathrm{supp}(\vec{m})$ is not
strongly connected, then as 
it is symmetric it must be disconnected. 
Then $M_Q^{\xi, s}(\vec{m}) =\emptyset$ for any $\xi$, 
which is a contradiction. 
Therefore $\mathrm{supp}(\vec{m})$ is strongly connected. 
For $i \in V(Q)$, 
note that we have $\langle \vec{m}, \vec{i} \rangle= \langle \vec{i}, \vec{m} \rangle$
as $Q$ is symmetric (see Remark~\ref{rmk:symmetric}). 
If $\langle \vec{m}, \vec{i} \rangle=\langle \vec{i}, \vec{m} \rangle>0$, then by 
Remark~\ref{rmk:simple}
for any $Q$-representation 
$\mathbb{V}$ with dimension vector $\vec{m}$,
there exist an injection $S_i \hookrightarrow \mathbb{V}$
and a surjection $\mathbb{V} \twoheadrightarrow S_i$. 
Such a representation $\mathbb{V}$ can never be 
$\xi$-stable for any choice of $\xi$, 
which is a contradiction. 
Therefore the criterion of Theorem~\ref{thm:criterion} is applied
and $M_Q^s(\vec{m}) \neq \emptyset$ holds. 
\end{proof}

For a symmetric quiver $Q$ and $\vec{m} \in \Gamma_Q$, let us take
$\xi^{\pm}=(\xi_i^{\pm})_{i \in V(Q)} \in \hH^{\sharp V(Q)}$
satisfying the following: 
\begin{align}\label{satisfy:xi}
\Re (\xi^+_i)=-\Re (\xi^-_i)  \in \mathbb{Z}, \ 
\sum_{i \in V(Q)} m_i \cdot \Re (\xi^{\pm}_i)=0.
\end{align}
We have the following diagram
\begin{align}\label{dia:sflop}
\xymatrix{
M_{Q}^{\xi^{+}}(\vec{m}) \ar[rd]_-{q_{Q}^{\xi^{+}}}
 & & M_Q^{\xi^{-}}(\vec{m}) \ar[ld]^-{q_{Q}^{\xi^{-}}} \\
& M_Q(\vec{m}). &
} 
\end{align}
By lemma~\ref{lem:slop}, 
we have $M_Q^{\xi^{+}, s}(\vec{m}) \neq \emptyset$ if
and only if 
$M_Q^{\xi^{-}, s}(\vec{m}) \neq \emptyset$. 

\begin{prop}\label{prop:sflop}
Suppose that $\vec{m}$ is primitive and
 $M_{Q}^{\xi^{\pm}, s}(\vec{m})=M_Q^{\xi^{\pm}}(\vec{m}) \neq \emptyset$
hold. Then the diagram (\ref{dia:sflop}) is a generalized flop 
of smooth varieties $M_{Q}^{\xi^{\pm}}(\vec{m})$. 
\end{prop} 
\begin{proof}
Under the assumption, the 
morphisms $q_Q^{\xi^{\pm}}$ are projective birational 
morphisms from smooth varieties $M_Q^{\xi^{\pm}}(\vec{m})$. 
Moreover the canonical divisors of $M_Q^{\xi^{\pm}}(\vec{m})$
are trivial by 
Lemma~\ref{lem:formula:K}
and the symmetric condition of $Q$. 
We show that 
$q_{Q}^{\xi^{\pm}}$ are isomorphisms in codimension one. 
By~\cite[Lemma~4.4]{TodGV}, the maps $q_Q^{\xi^{\pm}}$ are semismall, 
i.e. 
there is a stratification 
$\{S_{\lambda}\}_{\lambda}$ of $M_Q(\vec{m})$ such that 
for any $x \in S_{\lambda}$ 
we have 
the inequality
\begin{align}\label{ineq:strata}
\dim (q^{\xi^{\pm}}_Q)^{-1}(x) \le \frac{1}{2}\codim S_{\lambda}.
\end{align}
 Moreover
from the proof of \textit{loc.cit.} and~\cite[Theorem~1.4]{MeRe}
(which is used in \textit{loc.cit.}), 
 under the 
assumption $M_Q^{s}(\vec{m}) \neq \emptyset$, 
the equality holds in (\ref{ineq:strata})
only for the dense strata 
$S_{\lambda}=M_Q^s(\vec{m})$, i.e. 
$q^{\xi^{\pm}}_Q$ are small maps. 
In particular, $q_Q^{\xi^{\pm}}$ are isomorphisms in codimension one. 

It remains to show that there exists a $q_Q^{\xi^{+}}$-ample divisor 
on $M_Q^{\xi^{+}}(\vec{m})$ whose strict transform to 
$M_Q^{\xi-}(\vec{m})$ is $q_Q^{\xi^{-}}$-anti ample. 
Let us consider the following 
characters of $G$: 
\begin{align}\label{character}
g=(g_i)_{i \in V(Q)} \mapsto 
(\det g_i)^{\Re (\xi^{\pm}_i)}. 
\end{align}
They define 
$G$-equivariant line bundles on 
$\mathrm{Rep}_Q(\vec{m})$, 
hence on the stack $\mM_Q(\vec{m})$, 
which we write as $\lL_{\pm}$. 
Note that by the condition (\ref{satisfy:xi}), 
the characters (\ref{character}) are trivial 
on the diagonal torus $\mathbb{C}^{\ast} \subset G$. 
Therefore the restrictions of $\lL_{\pm}$ to
$\mM_Q^{\xi^{+}, s}(\vec{m})$, 
$\mM_Q^{\xi^{-}, s}(\vec{m})$
descend to line bundles 
$(L_{\pm})^{+}$, $(L_{\pm})^{-}$ on 
$M_Q^{\xi^{+}, s}(\vec{m})$, 
$M_Q^{\xi^{-}, s}(\vec{m})$ respectively. 
By the GIT construction of $M_{Q}^{\xi^{\pm}}(\vec{m})$
(see~\cite{Kin}), 
the line bundle
$(L_{+})^{+}$ is $q_{Q}^{\xi^{+}}$-ample, 
and $(L_-)^-$ is $q_Q^{\xi^{-}}$-ample.  
By the construction
of the above line bundles, 
the strict transform of $(L_+)^{+}$ on $M_Q^{\xi^{+}}(\vec{m})$ to 
$M_Q^{\xi^{-}}(\vec{m})$
is $(L_+)^{-}=((L_-)^{-})^{\vee}$, which is 
$q_{Q}^{\xi^{-}}$-anti-ample. 
Therefore
the diagram (\ref{dia:sflop}) is a generalized flop 
\end{proof}

For a convergent super-potential $W$ of $Q$, 
let us take an analytic open neighborhood
$0 \in V \subset M_Q(\vec{m})$ as in (\ref{V:open}). 
For two data $\xi^{\pm}$ as above, 
as in (\ref{dia:qflip}) we have the 
diagram 
\begin{align}\label{dia:sym:flop}
\xymatrix{
M_{(Q, \partial W)}^{\xi^{+}}(\vec{m})|_{V} 
\ar[rd]_-{q_{(Q, \partial W)}^{\xi^{+}}} & & 
\ar[dl]^-{q_{(Q, \partial W)}^{\xi^{-}}}
M_{(Q, \partial W)}^{\xi^{-}}(\vec{m})|_{V} \\
&  M_{(Q, \partial W)}(\vec{m})|_{V}. &
}
\end{align}
By Lemma~\ref{MW:dcrit}
and Proposition~\ref{prop:sflop}, we obtain the following corollary: 
\begin{cor}\label{cor:dflop}
Let $(Q, W)$ be a symmetric quiver $Q$ with a convergent 
super-potential $W$. 
Suppose that $\vec{m} \in \Gamma_Q$ is primitive and 
take $\xi^{\pm}$ satisfying (\ref{satisfy:xi}). 
If
$M_{Q}^{\xi^{\pm}, s}(\vec{m})=M_Q^{\xi^{\pm}}(\vec{m}) \neq \emptyset$
holds, 
then the diagram (\ref{dia:sym:flop}) is an analytic d-critical 
generalized flop. 
\end{cor}
\subsection{Flips via representations of extended quivers}\label{subsec:extend}
Let $Q$ be a symmetric quiver, and 
choose non-negative integers $a_i$, $b_i$
for each $i \in V(Q)$, and another non-negative integer $c$. 
We construct an extended quiver $Q^{\star}$ 
in the following way: 
the set of vertices is 
\begin{align*}
V(Q^{\star})=\{0\} \sqcup V(Q). 
\end{align*}
For $i, j \in V(Q) \subset V(Q^{\star})$, the 
set of edges from $i$ to $j$ in $Q^{\star}$ is the same 
as that in $Q$. The numbers of other edges are given by
\begin{align*}
\sharp E_{0, i}=a_i, \ \sharp E_{i, 0}=b_i, \ i \in V(Q), 
\ \sharp E_{0, 0}=c.  
\end{align*}
For example, see Figure~1. 
Note that $Q^{\star}$ contains $Q$ as a subquiver. 
In particular, any $Q$-representation is 
regarded as a $Q^{\star}$-representation in a natural way. 

\begin{figure}\label{figure:Q}
\begin{align*}
\xymatrix{
1    \ar@/^5pt/[rr] \ar@/^5pt/[ddrr] 
&& \ar@/^5pt/[ll]
2 \ar@/^5pt/[rr] \ar@/^15pt/[rr] 
&& \ar@/^5pt/[ll] \ar@/^15pt/[ll]
\ar[lldd]
3 \\ \\
&& 0 \ar@/^15pt/[uull] \ar@/^5pt/[uu] \ar@/_5pt/[uu]
\ar@/_20pt/[uurr]
\ar@/_10pt/[uurr]
\ar@/^25pt/[uull]
\ar@/^5pt/[uull]  
  && 
 }
\end{align*}
\caption{Picture of $Q^{\star}$
for $V(Q)=\{1, 2, 3\}$}
\end{figure}
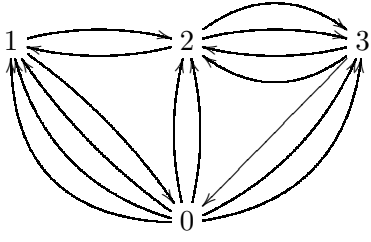

For a dimension vector $\vec{m} \in \Gamma_Q$
of $Q$, we define the extended dimension vector 
$\vec{m}^{\star} \in \Gamma_{Q^{\star}}$ by
\begin{align*}
\vec{m}^{\star} \cneq \vec{0}+\vec{m}
\end{align*}
i.e.
$(\vec{m}^{\star})_0=1$ and 
$(\vec{m}^{\star})_i=m_i$ for $i \in V(Q)$. 
The following lemma is obvious from the 
construction of $Q^{\star}$: 
\begin{lem}\label{lem:obvious}
Giving a $Q^{\star}$-representation $\mathbb{V}^{\star}$
with dimension vector $\vec{m}^{\star}$ is equivalent to 
giving a $Q$-representation 
\begin{align}\label{Qrep:V}
\mathbb{V}=\{(V_i, u_e)\}_{i \in V(Q), e \in E(Q)}
\end{align}
of dimension vector $\vec{m}$, 
together with linear maps
\begin{align}\label{lmaps}
\mathbb{E}_{0, i} \to V_i, \ 
\mathbb{E}_{i, 0} \otimes V_i \to \mathbb{C}
\end{align}
for each $i \in V(Q)$. 
Here 
$\mathbb{E}_{i, j}$ is the $\mathbb{C}$-vector 
space defined as in (\ref{Eab}). 
\end{lem}

Let us take data
$\xi^{\pm}=(\xi^{\pm}_i) \in \hH^{\sharp V(Q^{\star})}$ 
for $Q^{\star}$ 
satisfying
\begin{align}\label{star:data}
\xi^{\pm}_{i}=\sqrt{-1}, \ i \in V(Q), \ 
\Re (\xi^{+}_0) <0, \
 \Re (\xi^{-}_0)>0.
\end{align}
The following lemma 
characterizes $\xi^{\pm}$-semistable 
$Q^{\star}$-representations:  
\begin{lem}\label{lem:xistab}
Let $\mathbb{V}^{\star}$ be a $Q^{\star}$-representation 
of dimension vector $\vec{m}^{\star}$, 
given by a $Q$-representation 
$\mathbb{V}$ as in (\ref{Qrep:V})
together with linear maps (\ref{lmaps}). 
\begin{enumerate}
\item The object $\mathbb{V}^{\star}$ is $\xi^{+}$-semistable
if and only if it is $\xi^{+}$-stable if and only if 
the images of linear maps 
$\mathbb{E}_{0, i} \to V_i$ in (\ref{lmaps}) generate
$\bigoplus_{i\in V(Q)}V_i$ as a $\mathbb{C}[Q]$-module. 

\item The object $\mathbb{V}^{\star}$ is $\xi^{-}$-semistable
if and only if it is $\xi^{-}$-stable if and only if 
the images of linear maps 
$\mathbb{E}_{i, 0} \to V_i^{\vee}$
induced by right maps in (\ref{lmaps}) generate
$\bigoplus_{i \in V(Q)}V_i^{\vee}$ 
as a $\mathbb{C}[Q]$-module.
Here the $\mathbb{C}[Q]$-module structure 
on $\bigoplus_{i\in V(Q)}V_i^{\vee}$
is given by the dual $Q^{\vee}(=Q)$
representation 
$\mathbb{V}^{\vee}$
of $\mathbb{V}$. 
\end{enumerate}
\end{lem}
\begin{proof}
(i) By a choice of $\xi^{+}$, a $Q^{\star}$-representation 
$\mathbb{V}^{\star}$
of dimension vector $\vec{m}^{\star}$ is $\xi^{+}$-(semi)stable 
if and only if there is no surjection 
$\mathbb{V}^{\star} \twoheadrightarrow \mathbb{V}'$ 
as $Q^{\star}$-representations
where $\mathbb{V}'$ 
is a non-zero $Q$-representation. 
The last condition is equivalent to that 
the images of linear maps $\mathbb{E}_{0, i} \to V_i$ 
generate $\bigoplus_{i\in V(Q)}V_i$
as a $\mathbb{C}[Q]$-module. 

(ii) By a choice of $\xi^{-}$, 
a $Q^{\star}$-representation 
$\mathbb{V}^{\star}$
of dimension vector $\vec{m}^{\star}$ is $\xi^{-}$-(semi)stable 
if and only if there is no injection
$\mathbb{V}' \hookrightarrow \mathbb{V}^{\star}$ 
as $Q^{\star}$-representations
where $\mathbb{V}'$
is a non-zero $Q$-representation.  
This is equivalent to that the dual representation 
$(\mathbb{V}^{\star})^{\vee}$
does not admit a surjection to $\mathbb{V}''$ 
in the category of $(Q^{\star})^{\vee}$-representations,
where 
$\mathbb{V}''$
is a non-zero $Q^{\vee}(=Q)$-representation. 
Therefore similarly to (i), we conclude (ii). 
\end{proof}
By Lemma~\ref{lem:xistab}, we have 
$M_{Q^{\star}}^{\xi^{\pm}, s}(\vec{m}^{\star})
=M_{Q^{\star}}^{\xi^{\pm}}(\vec{m}^{\star})$, 
so they are smooth quasi-projective varieties. 
Let us consider the following diagram
\begin{align}\label{dia:sflip}
\xymatrix{
M_{Q^{\star}}^{\xi^{+}}(\vec{m}^{\star})
 \ar[rd]_-{q_{Q^{\star}}^{\xi^{+}}}
 & & M_{Q^{\star}}^{\xi^{-}}(\vec{m}^{\star}) 
\ar[ld]^-{q_{Q^{\star}}^{\xi^{-}}} \\
& M_{Q^{\star}}(\vec{m}^{\star}). &
} 
\end{align}

\begin{lem}\label{lem:ample}
Suppose that $a_i>b_i$ for all $i \in V(Q)$. Then 
in the diagram (\ref{dia:sflip}), 
we have the following: 
\begin{enumerate}
\item The anti-canonical divisor 
of 
$M_{Q^{\star}}^{\xi^{+}}(\vec{m}^{\star})$ 
is ample. 
\item The 
canonical divisor of 
$M_{Q^{\star}}^{\xi^{-}}(\vec{m}^{\star})$
is ample. 
\end{enumerate}
\end{lem}
\begin{proof}
For $i \in V(Q^{\star})$ and $e \in E(Q^{\star})$, let
\begin{align*}
\vV_i^{\pm} \to M_{Q^{\star}}^{\xi^{\pm}}(\vec{m}^{\star}), \ 
\mathbf{u}_{e} \colon \vV_{s(e)}^{\pm} \to \vV_{t(e)}^{\pm}
\end{align*}
be a universal $Q^{\star}$-representations. 
Note that such a universal representations exists
as $\vec{m}^{\star}$ is primitive and 
$M_{Q^{\star}}^{\xi^{\pm}, s}(\vec{m}^{\star})
=M_{Q^{\star}}^{\xi^{\pm}}(\vec{m}^{\star})$ hold. 
For $i=0$, the vector bundles 
$\vV_0^{\pm}$ are line bundles by 
our choice of $\vec{m}^{\star}$, so by 
replacing $\vV_i^{\pm}$
by $\vV_i^{\pm} \otimes (\vV_0^{\pm})^{-1}$ we may 
assume that $\vV_0^{\pm}$ are trivial line bundles.   
Then by (\ref{formula:K}), we have 
\begin{align}\label{formula:K2}
\omega_{M_{Q^{\star}}^{\xi^{\pm}}(\vec{m}^{\star})}
=\bigotimes_{i \in V(Q)} \det (\vV_i^{\pm})^{b_i-a_i}.
\end{align}

Let $\mathbf{E} \subset \mathbb{C}[Q]$ be the 
vector subspace generated by paths of the form
$e_1 e_2 \cdots e_n$ for $n\ge 1$ such
that $s(e_1) =0$, $t(e_1) \in V(Q)$
and $e_2, \ldots, e_n \in E(Q)$. 
Then the compositions
$\mathbf{u}_{e_n} \circ \cdots \circ \mathbf{u}_{e_1}$
determine the morphism of sheaves
\begin{align*}
\mathbf{E} \otimes \oO_{M_{Q^{\star}}^{\xi^{+}}(\vec{m}^{\star})}
\to \bigoplus_{i\in V(Q)} \vV_i^{+}. 
\end{align*}
Then Lemma~\ref{lem:xistab} (i) implies that 
the above morphism is surjective. 
Therefore each $\vV_i^{+}$ is generated by 
its global sections. 
By (\ref{formula:K2})
and the assumption $a_i>b_i$ for all $i\in V(Q)$, 
the line bundle 
$(\omega_{M_{Q^{\star}}^{\xi^{+}}(\vec{m}^{\star})})^{-1}$
is generated by its global sections. 
In order to show that it is ample, it is enough to show that 
it has positive degree on any projective curve 
on $M_{Q^{\star}}^{\xi^{+}}(\vec{m}^{\star})$. 
Let $C$ be a smooth projective curve and 
take a non-constant map
\begin{align*}h \colon C \to M_{Q^{\star}}^{\xi^{+}}(\vec{m}^{\star}).
\end{align*}
Note that each degree of $h^{\ast}\vV_i^{+}$
is non-negative, as it is globally generated. 
Therefore
if the degree of 
$h^{\ast}(\omega_{M_{Q^{\star}}^{\xi^{+}}(\vec{m}^{\star})})^{-1}$
is non-positive, then 
each degree of $h^{\ast}\vV_i^{+}$
must be zero. 
By Sublemma~\ref{sublem1} below, 
in this case
the vector bundle $h^{\ast}\vV_i^{+}$ 
must be a direct
sum of $\oO_C$. 
Then the pull-back of 
the universal map $\mathbf{u}_e$ to $C$
by the map $h$ has to be constant. 
But this implies that the map 
$h$ is constant, which is a contradiction. 
Therefore (i) of the lemma holds. 
The result of (ii) follows from 
Lemma~\ref{lem:xistab} (ii) and the dual argument of (i). 
\end{proof}
We used the following sublemma: 
\begin{sublem}\label{sublem1}
Let $C$ be a smooth projective curve and 
$\vV$ a vector bundle on it. 
Suppose that $\vV$ is generated by its global sections, 
and $\deg \vV=0$. 
Then $\vV$ is isomorphic to a direct sum of 
$\oO_C$.
\end{sublem}
\begin{proof}
We set $W=H^0(C, \vV)$, and 
$r=\rank \vV$. 
The natural surjection 
$W \otimes \oO_C \twoheadrightarrow \vV$
induces a morphism
\begin{align*}
\gamma \colon C \to \mathrm{Gr}(W, r) \hookrightarrow 
\mathbb{P}^N. 
\end{align*}
Here $N=\dim \bigwedge^r W$, 
and the right arrow is the Pl{\"u}cker embedding. 
Since $\deg \vV=\deg \gamma^{\ast}\oO(1)$, 
the assumption $\deg \vV=0$ implies that 
$\gamma$ is a constant map. 
Also the vector bundle $\vV$ is a pull-back of the universal 
quotient bundle
 on $\mathrm{Gr}(W, r)$.
Therefore
$\vV$ is isomorphic to a direct sum of $\oO_C$. 
\end{proof}

Using the above lemma, we show the following proposition: 
\begin{prop}\label{prop:Mflip}
Suppose that $a_i>b_i$ for all $i \in V(Q)$.
 Then 
in the diagram (\ref{dia:sflip}), 
either one of the followings hold: 
\begin{enumerate}
\item We have $M_{Q^{\star}}^{\xi^{-}}(\vec{m}^{\star})=\emptyset$
and the diagram (\ref{dia:sflip}) is a generalized MFS. 
\item We have 
$M_{Q^{\star}}^{\xi^{\pm}}(\vec{m}^{\star})\neq \emptyset$
and the diagram (\ref{dia:sflip}) is a generalized flip. 
\end{enumerate}
\end{prop}
\begin{proof}
If $M_{Q^{\star}}^{s}(\vec{m}^{\star}) \neq \emptyset$, 
then 
the maps in (\ref{dia:sflip}) are birational 
so 
(ii) holds by Lemma~\ref{lem:ample}. 
Below we show that
$M_{Q^{\star}}^{s}(\vec{m}^{\star}) = \emptyset$
implies that 
$M_{Q^{\star}}^{\xi^{-}}(\vec{m}^{\star}) = \emptyset$. 
Note that for $i\in V(Q)$ we have
\begin{align*}
\langle \vec{m}^{\star}, \vec{i} \rangle 
=\langle \vec{m}, \vec{i} \rangle -a_i, \ 
\langle \vec{i}, \vec{m}^{\star}\rangle 
=\langle \vec{i}, \vec{m} \rangle -b_i.
\end{align*}
We also have the identities
\begin{align*}
\langle \vec{m}^{\star}, \vec{0} \rangle 
=1-\sum_{i\in V(Q)} b_i m_i, \ 
\langle \vec{0}, \vec{m}^{\star} \rangle 
=1-\sum_{i\in V(Q)} a_i m_i.
\end{align*}
By our assumption $a_i>b_i$, we have the inequalities
\begin{align}\label{ineq:langle}
\langle \vec{m}^{\star}, \vec{i} \rangle 
<\langle \vec{i}, \vec{m}^{\star} \rangle, \ 
\langle \vec{0}, \vec{m}^{\star}\rangle
<
\langle \vec{m}^{\star}, \vec{0} \rangle.
\end{align}
If $\langle \vec{i}, \vec{m}^{\star} \rangle>0$
for $i\in V(Q)$,  
then by Remark~\ref{rmk:simple}
any $Q^{\star}$-representation $\mathbb{V}^{\star}$ of dimension vector 
$\vec{m}^{\star}$ admits an injection $S_i \hookrightarrow 
\mathbb{V}^{\star}$. 
Therefore we have $M_{Q^{\star}}^{\xi^{-}}(\vec{m}^{\star}) = \emptyset$. 
Similarly if 
$\langle \vec{m}^{\star}, \vec{0} \rangle>0$, then 
any such $\mathbb{V}^{\star}$ admits a surjection 
$\mathbb{V}^{\star}\twoheadrightarrow S_0$, which implies 
$M_{Q^{\star}}^{\xi^{-}}(\vec{m}^{\star}) = \emptyset$. 
Therefore by the inequalities (\ref{ineq:langle}), 
we may assume that 
$\langle \vec{m}^{\star}, \vec{j}\rangle \le 0$
and $\langle \vec{j}, \vec{m}^{\star} \rangle \le 0$
for any $j \in V(Q^{\star})$. 

By Theorem~\ref{thm:criterion}, 
the condition 
$M_{Q^{\star}}^{s}(\vec{m}^{\star}) = \emptyset$
implies that $\mathrm{supp}(\vec{m}^{\star})$
 is not strongly connected. 
Let $Q_1, \cdots, Q_l$ be the connected components of 
$\mathrm{supp}(\vec{m})$, which are subquivers of 
$Q$. 
As $\mathrm{supp}(\vec{m}^{\star})$ is not 
strongly connected and $Q$ is symmetric, there is 
$1\le k \le l$ such that 
we have $b_i=0$ for any $i \in V(Q_k)$. 
This implies that any 
$Q^{\star}$-representation $\mathbb{V}^{\star}$ of dimension vector 
$\vec{m}^{\star}$ admits an injection
$\mathbb{V}' \hookrightarrow \mathbb{V}^{\star}$ for a
non-zero 
$Q_k$-representation $\mathbb{V}'$. 
Therefore we have $M_{Q^{\star}}^{\xi^{-}}(\vec{m}^{\star}) = \emptyset$
and (i) holds. 
\end{proof}

\begin{exam}\label{exam:Grass}
Let $Q$ be the symmetric quiver 
with one vertex and no loops, and 
write $V(Q)=\{1\}$. 
Then $Q^{\star}$ has two vertices 
$V(Q^{\star})=\{0, 1\}$, 
with $a_1$-arrows from 
$0$ to $1$, 
$b_1$-arrows from 
$1$ to $0$, and $c$-loops at 
$0$. The dimension vector 
$\vec{m}^{\star}$ is written as
$\vec{0}+m \cdot \vec{1}$
for $m \in \mathbb{Z}_{>0}$. 
Let $V$ be a vector space with dimension $m$. 
Then by Lemma~\ref{lem:obvious}, 
the stack $\mM_{Q^{\star}}(\vec{m}^{\star})$ is written as
\begin{align*}
\mM_{Q^{\star}}(\vec{m}^{\star})
=\left[(\Hom(\mathbb{E}_{0, 1}, V) \times 
\Hom(\mathbb{E}_{1, 0}, V^{\vee}))/(\mathbb{C}^{\ast} \times \GL(V))
\right] \times \mathbb{E}_{0, 0}^{\vee}. 
\end{align*}
By Lemma~\ref{lem:xistab}, we see that
\begin{align*}
&M_{Q^{\star}}^{\xi^+}(\vec{m}^{\star})
=\mathrm{Tot}_{\mathrm{Gr}(\mathbb{E}_{0, 1}, m)}
(\qQ_{0, 1}^{\vee} \otimes \mathbb{E}_{1, 0}^{\vee}) \times 
\mathbb{E}_{0, 0}^{\vee}, \\
&M_{Q^{\star}}^{\xi^-}(\vec{m}^{\star})
=\mathrm{Tot}_{\mathrm{Gr}(\mathbb{E}_{1, 0}, m)}
(\qQ_{1, 0}^{\vee} \otimes \mathbb{E}_{0, 1}^{\vee}) \times 
\mathbb{E}_{0, 0}^{\vee}.
\end{align*}
Here $\qQ_{i, j}$ is the universal 
quotient bundle on 
$\mathrm{Gr}(\mathbb{E}_{i, j}, m)$.
In this case, the birational map
$M_{Q^{\star}}^{\xi^+}(\vec{m}^{\star})
\dashrightarrow M_{Q^{\star}}^{\xi^-}(\vec{m}^{\star})$
is a 
Grassmannian flip. 
\end{exam}

Let $W^{\star}$ be a convergent super-potential of $Q^{\star}$, 
and take an analytic 
neighborhood $0 \in V \subset M_{Q^{\star}}(\vec{m}^{\star})$
as in (\ref{V:open}).
We have the diagram
 \begin{align}\label{dia:sym:flip}
\xymatrix{
M_{(Q^{\star}, \partial W^{\star})}^{\xi^{+}}(\vec{m}^{\star})|_{V} 
\ar[rd]_-{q_{(Q^{\star}, \partial W^{\star})}^{\xi^{+}}} & & 
\ar[dl]^-{q_{(Q^{\star}, \partial W^{\star})}^{\xi^{-}}}
M_{(Q^{\star}, \partial W^{\star})}^{\xi^{-}}(\vec{m}^{\star})|_{V} \\
&  M_{(Q^{\star}, \partial W^{\star})}(\vec{m}^{\star})|_{V}. &
}
\end{align}
By Lemma~\ref{MW:dcrit} and Proposition~\ref{prop:Mflip}, 
we have the following corollary: 
\begin{cor}\label{cor:dflip}
For the diagram (\ref{dia:sym:flip}),
either one of the followings holds:
\begin{enumerate}
\item We have $M_{Q^{\star}}^{\xi^{-}}(\vec{m}^{\star})=M_{(Q^{\star}, \partial W^{\star})}^{\xi^{-}}(\vec{m}^{\star})|_{V}=\emptyset$
and the diagram (\ref{dia:sym:flip}) is a d-critical generalized MFS. 
\item The diagram (\ref{dia:sym:flip}) is a d-critical generalized flip. 
\end{enumerate} 
\end{cor}

We also have the strictness (see Definition~\ref{defi:strict})
of the diagram (\ref{dia:sym:flip})
under some conditions: 
\begin{lem}\label{lem:strict}
Suppose that $W^{\star}$ is minimal (see (\ref{form:W})) and 
$a_i>m_i$ for any $i \in V(Q)$. Then 
the diagram (\ref{dia:sym:flip}) is strict
at $0 \in M_{(Q^{\star}, \partial W^{\star})}(\vec{m}^{\star})|_{V}$. 
\end{lem}
\begin{proof}
We need to show that the map 
$q_{(Q^{\star}, \partial W^{\star})}^{\xi^{+}}$ in the diagram (\ref{dia:sym:flip}) is not a finite morphism
at $0 \in M_{(Q^{\star}, \partial W^{\star})}(\vec{m}^{\star})|_{V}$. 
We consider nilpotent $Q^{\star}$-representations $\mathbb{V}^{\star}$
given by $Q$-representations (\ref{Qrep:V}) where $u_e=0$ for all $e \in E(Q)$, 
together with surjective linear maps 
$\mathbb{E}_{0, i} \to V_i$, and zero maps 
$\mathbb{E}_{i, 0} \otimes V_i \to 0$
in (\ref{lmaps}). 
The isomorphism 
classes of such $Q^{\star}$-representations 
form the product of 
Grassmannians $\mathrm{Gr}(\mathbb{E}_{0, i}, m_i)$
for all $i \in V(Q)$. 
By Lemma~\ref{lem:xistab}, such $Q^{\star}$-representations are 
$\xi^+$-stable. 
They also satisfy the relation $\partial W^{\star}$ (see Remark~\ref{rmk:pW}) 
by the minimality of $W^{\star}$, 
so we have 
\begin{align*}
\prod_{i \in V(Q)} \mathrm{Gr}(\mathbb{E}_{0, i}, m_i) \subset 
\left(q_{(Q^{\star}, \partial W^{\star})}^{\xi^{+}} \right)^{-1}(0). 
\end{align*}
Since the LHS is not zero dimensional 
by the assumption $a_i >m_i$, 
the lemma holds. 
\end{proof}

\section{D-critical flops of moduli spaces of one dimensional sheaves}
\label{sec:onedim}
In this section, we show that 
wall-crossing phenomena of 
one dimensional stable sheaves on CY 3-folds are described in terms of 
d-critical (generalized) flops. 
The proof of this result is related to the proof of the 
wall-crossing formula of Gopakumar-Vafa invariants given in~\cite{TodGV}.

\subsection{Twisted semistable sheaves}\label{subsec:twist2}
For a smooth projective CY 3-fold $X$, let 
$\Coh_{\le 1}(X) \subset \Coh(X)$
be
the abelian subcategory of 
coherent sheaves $E$ on $X$ whose supports have dimensions less than or equal to one, 
and 
\begin{align*}
D^b_{\le 1}(X) \cneq D^b(\Coh_{\le 1}(X))\subset D^b(X)
\end{align*}
its bounded derived category. 
Let $\Gamma_{\le 1}$ be defined by
\begin{align}\label{def:Gamma}
\Gamma_{\le 1} \cneq H_2(X, \mathbb{Z}) \oplus \mathbb{Z}.
\end{align}
The Chen character 
of an object in $D^b_{\le 1}(X)$ takes
its value in $\Gamma_{\le 1}$, and given by
\begin{align}\label{ch:Gamma}
\ch(E)=(\ch_2(E), \ch_3(E)) =([E], \chi(E)). 
\end{align}
Here $[E]$ is the fundamental one cycle associated with $E$. 

We denote by $\Stab_{\le 1}(X)$ the 
space of Bridgeland stability conditions on 
$D^b_{\le 1}(X)$ with respect to the 
Chern character map (\ref{ch:Gamma}). 
Namely a point $\sigma \in \Stab_{\le 1}(X)$
is a pair 
\begin{align*}
\sigma=(Z, \aA), \ 
\aA \subset D^b_{\le 1}(X), \ 
Z \colon \Gamma_{\le 1} \to \mathbb{C}
\end{align*}
where $\aA$ is the heart of a bounded t-structure on 
$D^b_{\le 1}(X)$ and $Z$
is a group homomorphism, satisfying 
some conditions (see Appendix~\ref{sec:Bridgeland}). 
By Theorem~\ref{thm:Stab}, 
the forgetting map $(Z, \aA) \mapsto Z$ gives a local 
homeomorphism 
 \begin{align*}
\Stab_{\le 1}(X) \to (\Gamma_{\le 1})_{\mathbb{C}}^{\vee}. 
\end{align*}
Let $A(X)_{\mathbb{C}}$ be the complexified ample cone of $X$
defined by
\begin{align*}
A(X)_{\mathbb{C}} \cneq \{B+i\omega \in H^2(X, \mathbb{C}) : \omega \mbox{ is ample }\}. 
\end{align*}
For a given element $B+i\omega \in A(X)_{\mathbb{C}}$, let 
$Z_{B, \omega}$ be the group homomorphism 
defined by 
\begin{align}\label{ZBw}
Z_{B, \omega} \colon \Gamma_{\le 1} \to \mathbb{C}, \ 
(\beta, n) \mapsto -n+(B+i\omega)\beta.
\end{align}
Then the pair
\begin{align}\label{sigma:Bw}
\sigma_{B, \omega} \cneq (Z_{B, \omega}, \Coh_{\le 1}(X))
\end{align}
determines a point in $\Stab_{\le 1}(X)$. 
The map 
\begin{align*}
A(X)_{\mathbb{C}} \to \Stab_{\le 1}(X),  \ 
(B, \omega) \mapsto \sigma_{B, \omega}
\end{align*}
is a continuous injective 
map, whose image is denoted by 
\begin{align*}
U(X) \subset \Stab_{\le 1}(X).
\end{align*}  

\begin{rmk}
For an object $E \in \Coh_{\le 1}(X)$, it is 
$\sigma_{B, \omega}$-(semi)stable if and only if for any subsheaf 
$0 \neq F \subsetneq E$, we have the inequality
\begin{align*}
\mu_{B, \omega}(F)<(\le) \mu_{B, \omega}(E).
\end{align*}
Here $\mu_{B, \omega}(E) \in \mathbb{R} \cup \{\infty\}$ is defined by 
\begin{align}\label{mu:slope}
\mu_{B, \omega}(E) \cneq 
\frac{\chi(E)-B \cdot [E]}{\omega \cdot [E]}
=-\frac{\Re Z_{B, \omega}(E)}{\Im Z_{B, \omega}(E)} 
\end{align}
when $\omega \cdot [E] \neq 0$, 
and $\mu_{B, \omega}(E)=\infty$
when $\omega \cdot [E]=0$. 
\end{rmk}

\subsection{Moduli spaces of one dimensional stable sheaves}
Let 
us take
\begin{align*}
\beta \in H_2(X, \mathbb{Z}), \ 
\sigma=(Z, \Coh_{\le 1}(X)) \in U(X)
\end{align*}
where $\beta$ is an effective curve class. 
We denote by $\mM_{\sigma}(\beta)$
the moduli stack of $\sigma$-semistable 
$E \in \Coh_{\le 1}(X)$ 
satisfying $\ch(E)=(\beta, 1)$. 
The stack $\mM_{\sigma}(\beta)$ is an Artin 
stack locally of finite type, with a 
good moduli space
\begin{align*}
p_M \colon 
\mM_{\sigma}(\beta) \to M_{\sigma}(\beta)
\end{align*}
for a projective scheme $M_{\sigma}(\beta)$ (see~\cite[Lemma~7.4]{Todstack}). 
A closed point of $M_{\sigma}(\beta)$ corresponds to a 
$\sigma$-polystable sheaf, i.e. a direct sum
\begin{align}\label{polystable}
E=\bigoplus_{i=1}^k V_i \otimes E_i
\end{align}
where each $V_i$ is a finite dimensional 
vector space, $E_i \in \Coh_{\le 1}(X)$ is a $\sigma$-stable sheaf
with $\arg Z(E_i)=\arg Z(E)$ and 
$E_i \not\cong E_j$ for $i\neq j$.
\begin{rmk}
The stack $\mM_{\sigma}(\beta)$ is a GIT quotient stack
(see~\cite[Lemma~7.4]{Todstack}), 
so the good moduli space $M_{\sigma}(\beta)$
exists without relying on~\cite{AHLH}
(see Remark~\ref{rmk:AHHL}). 
\end{rmk}

As we mentioned in Subsection~\ref{subsec:wallCY3}, there 
is a 
wall-chamber structure on 
$\Stab_{\le 1}(X)$. 
On the subset $U(X) \subset \Stab_{\le 1}(X)$, 
each wall is given by
\begin{align}\notag
\{(Z, \Coh_{\le 1}(X)) \in U(X) : 
Z(v_1) \in \mathbb{R}_{>0} Z(v_2) \}
\end{align}
for each decomposition
\begin{align*}
(\beta, 1)=v_1+v_2, \ 
v_i=(\beta_i, n_i) \in \Gamma_{\le 1}.
\end{align*}
Here $\beta_i$ is an effective curve class. 
Suppose that 
$\sigma \in U(X)$ 
 lies in one of the above walls, 
and write $\sigma=\sigma_{B, \omega}$ as in (\ref{sigma:Bw})
 for $B+i\omega \in A(X)_{\mathbb{C}}$.
Let us take another stability conditions
$\sigma^{\pm} \in U(X)$ 
written as 
\begin{align}\label{spm}
\sigma^{\pm}=\sigma_{B \pm \varepsilon_B, \omega \pm \varepsilon_{\omega}} \in U(X)
\end{align}
for $\varepsilon_B +i \varepsilon_{\omega} \in H^2(X, \mathbb{C})$.
We assume that $\varepsilon_B +i \varepsilon_{\omega}$
is sufficiently small and 
general so that 
both of $\sigma^{\pm}$ lie on chambers. 
Similarly to the diagram (\ref{wall:diagram2}), we 
have the diagram
\begin{align}\label{dia:Mone}
\xymatrix{
M_{\sigma^{+}}(\beta) \ar[rd]_-{q_M^{+}} & & M_{\sigma^{-}}(\beta) 
\ar[ld]^-{q_M^{-}} \\
& M_{\sigma}(\beta). &
}
\end{align}
Note that as $(\beta, 1)$ is primitive in 
$\Gamma_{\le 1}$
and $\sigma^{\pm}$ lie on chambers, 
both of $M_{\sigma^{\pm}}(\beta)$
consist of $\sigma^{\pm}$-stable 
sheaves. Therefore by Theorem~\ref{thm:CYdcrit}, 
they 
admit d-critical structures.  
Applying Corollary~\ref{cor:analytic}, 
we have the following: 
\begin{thm}\label{thm:dflop:one}\label{thm:Dflop:one}
The diagram (\ref{dia:Mone})
is an analytic d-critical generalized flop. 
\end{thm}
\begin{proof}
For a point $p \in M_{\sigma}(\beta)$, 
suppose that it corresponds to a 
polystable sheaf $E$ of the form (\ref{polystable}). 
Since each $E_i$ has at most 
one dimensional support,  
the Ext-quiver $Q=Q_{E_{\bullet}}$ 
associated with the collection 
\begin{align*}
E_{\bullet}=(E_1, E_2, \ldots, E_k)
\end{align*}
is 
symmetric (see~\cite[Lemma~5.1]{TodGV}). 
We take
data $\xi^{\pm} \in \hH^{k}$ as in (\ref{xi})
for the quiver $Q$, 
given by 
\begin{align*}
\xi^{\pm}_i=
Z^{\pm}_{B \pm \varepsilon_B, \omega \pm \varepsilon_{\omega}}(E_i)=
Z_{B, \omega}(E_i) \pm (\varepsilon_B+i\varepsilon_{\omega}) \cdot [E_i].
\end{align*}
As $\arg Z_{B, \omega}(E_i)=\arg Z_{B, \omega}(E)$, by 
taking rotations and scaling
of $\xi^{\pm}$, 
and also perturbing $\varepsilon_B+i\varepsilon_{\omega}$ if necessary, 
we may assume that $\xi^{\pm}$ satisfy the condition (\ref{satisfy:xi}). 
Then the result follows from 
Corollary~\ref{cor:analytic}
and Corollary~\ref{cor:dflop}. 
\end{proof}

\subsection{Example: elliptic CY 3-fold}\label{exam:elliptic}
Here we discuss
an example of 
 wall-crossing of one dimensional 
stable sheaves
on an elliptic CY 3-fold. 
Let $S=\mathbb{P}^2$ and take 
general elements 
\begin{align*}
u \in H^0(S, \oO_S(-4K_S)), \ 
v \in H^0(S, \oO_S(-6K_S)).
\end{align*}
Then as in~\cite[Section~6.4]{Tsurvey}, 
we have a simply connected CY 3-fold 
$X$ with a flat elliptic fibration
\begin{align}\label{pi:XS}
\pi_X \colon X \to S
\end{align}
defined by the equation 
$zy^2=uxz^2+vz^3$
in the projective bundle 
\begin{align*}
\mathbb{P}_S(\oO_S(-2K_S) \oplus \oO_S(-3K_S) \oplus \oO_S) \to S. 
\end{align*}
Here $[x:y:z]$ is the homogeneous coordinate
of the above projective bundle. 
Note that $\pi_X$ admits a section
$\iota \colon S \to X$
whose image $D \cneq \iota(S)$
correspond to the fiber point 
$[0:1:0]$. 
Let $H \subset X$ be 
the pull-back of a hyperplane in $\mathbb{P}^2$
to $X$ by $\pi_X$. We have
\begin{align*}
H^2(X, \mathbb{R})=\mathbb{R}[D]+\mathbb{R}[H].
\end{align*}
Let $F$ be a fiber of $\pi_X$ and 
$l \subset D$ a line. 
Then $[F]$ and $[l]$ span the Mori cone
$\overline{NE}(X)$ of $X$. 
The intersection matrix is given by
\begin{align*}
\left(\begin{array}{cc}
H \cdot l & D \cdot l \\
H \cdot F & D \cdot F
\end{array}  \right)
=\left(\begin{array}{cc}
1 & -3 \\
0 & 1
\end{array}  \right).
\end{align*}

We fix an ample divisor $\omega_0$ on $X$
and write
$d_1=\omega_0 \cdot F>0$, 
$d_2=\omega_0 \cdot l>0$. 
Let us take an effective 
curve class $\beta$ and write it 
as 
\begin{align*}
\beta=r[F]+k[l], \ 
r, k \in \mathbb{Z}_{\ge 0}.
\end{align*}
We consider wall-chamber structure 
on the subset of $U(X)$, 
given by the image of the map
\begin{align}\label{identify:U}
H^2(X, \mathbb{R}) \to
U(X), \ 
B \mapsto \sigma_{B, \omega_0}.
\end{align}
We identify the image of (\ref{identify:U})
with $H^2(X, \mathbb{R})$
by the above map. 
For each decomposition in $\Gamma_{\le 1}(X)$
\begin{align}\label{decom:beta1}
(\beta, 1)=(\beta_1, n_1)+(\beta_2, n_2), \ 
\beta_i=r_i[F]+k_i[l]
\end{align}
the wall is given by the equation 
$\mu_{B, \omega_0}(\beta_1, n_1)=\mu_{B, \omega_0}(\beta, 1)$.
By writing $B=x[D]+y[H]$, 
a direct computation shows that 
the above condition is equivalent to
\begin{align}\label{wall:equation}
(3d_1+d_2)x-d_1y=
\frac{r_1d_1+k_1 d_2 -n_1(rd_1+kd_2)}{rk_1-kr_1}.
\end{align}
It follows that every wall is proportional 
to the line $y=(3+d_2/d_1)x$, so any two 
wall are
disjoint if they do not coincide. 

In the case of $k=1$, i.e. 
$\beta=r[F]+[l]$, 
we have the following 
decomposition 
\begin{align}\label{beta:decom}
(r[F]+[l], 1)=(r[F], 1)+([l], 0). 
\end{align} 
We set 
\begin{align}\notag
B_0=\frac{d_2}{r(3d_1+d_2)}[D], \ 
B_{\pm}=B_0 \pm \varepsilon[D], \ 
0<\varepsilon \ll 1.
\end{align}
The above $B_0$ 
satisfies the equation 
(\ref{wall:equation})
determined by the decomposition (\ref{beta:decom}).
Therefore $\sigma_0 \cneq \sigma_{B_0, \omega_0}$ lies 
on a wall  
and $\sigma_{\pm} \cneq \sigma_{B_{\pm}, \omega_0}$ lie on 
its adjacent chambers. 

In the above $k=1$ case, we can describe 
the wall-crossing diagram (\ref{dia:Mone}) 
for $\sigma=\sigma_0$
in terms of d-critical simple flop. 
It is easy to see that 
$\sigma_0$ does not lie 
on a wall determined by a 
decomposition of the form (\ref{decom:beta1}), 
other than (\ref{beta:decom}). 
Therefore any point $p\in M_{\sigma_0}(\beta)$,
which do not correspond to $\sigma_0$-stable sheaf,
corresponds to a 
a $\sigma_0$-polystable 
sheaf $E$ of the form
\begin{align*}
E=E_1 \oplus E_2, \ 
\ch(E_1)=(r[F], 1), \ 
\ch(E_2)=([l], 0).
\end{align*}
Here $E_1$, $E_2$ are $\sigma$-stable 
sheaves. 
Then $E_1$ is a 
stable vector bundle on 
a fiber $F$ with rank $r$ and degree $1$, 
and $E_2=\oO_l(-1)$
for some line $l \subset D$. 
If $F \cap l=\emptyset$, then 
\begin{align*}
\Ext^1(E_1, E_2)=\Ext^1(E_2, E_1)=0
\end{align*}
and 
we have $(q_M^{\pm})^{-1}(U)=\emptyset$
for a small open subset $p \in U \subset M_{\sigma}(\beta)$. 
This is an obvious case of d-critical generalized flop 
(see Remark~\ref{rmk:dcrit}). 
Otherwise
$F \cap l$ is a one point, and  
it is easy to check that
\begin{align*}
&\Ext^1(E_1, E_1)=\mathbb{C}^3, \ 
\Ext^1(E_2, E_2)=\mathbb{C}^2, \\ 
&\Ext^1(E_1, E_2)=\Ext^1(E_2, E_1)=\mathbb{C}^r. 
\end{align*}
The associated Ext-quiver $Q$ 
has two vertices $\{1, 2\}$, 
$r$-arrows from $1$ to $2$ and $2$ to $1$, 
and the number of 
loops at $1$, $2$ are $3$, $2$ respectively. 
In this case, 
the diagram (\ref{dia:Mone}) 
is a d-critical simple flop at $p \in M_{\sigma}(\beta)$ 
as we mentioned 
in Example~\ref{subsec:dflop}.

\subsection{D-critical flops under flops}\label{subsec:FF}
For a CY 3-fold $X$ and $\beta \in H_2(X, \mathbb{Z})$, we define
\begin{align}\label{moduli:GV}
M_X(\beta) \cneq M_{\sigma=\sigma_{(0, \omega)}}(\beta).
\end{align}
Here $\omega$ is an ample divisor on $X$. 
Note that $M_X(\beta)$ is independent of a choice of 
$\omega$ (see~\cite[Remark~3.2]{MT}).
Moreover $M_X(\beta)$ consists of stable sheaves, 
so it admits a d-critical structure
by Theorem~\ref{thm:CYdcrit}.
The moduli space (\ref{moduli:GV}) was used in~\cite{MT}
in the definition of Gopakumar-Vafa invariants on 
CY 3-folds.  

Suppose that we have a flop diagram of 
CY 3-folds
\begin{align}\label{flop:3-fold}
\xymatrix{
X \ar[rd]\ar@{.>}[rr]^-{\phi} & & X^{\dag} \ar[ld] \\
& Y. &
}
\end{align}
Let $\phi_{\ast}\beta \in H_2(X^{\dag}, \mathbb{Z})$
be defined by 
$\phi_{\ast} \beta \cdot D=\beta \cdot \phi_{\ast}^{-1}D$
for any divisor $D$ on $X^{\dag}$. Here 
$\phi_{\ast}^{-1}D$ is a strict transform of $D$ to $X$. 
In the above situation, we have the following:
\begin{thm}\label{thm:FF}
The d-critical loci 
$M_{X}(\beta)$ and $M_{X^{\dag}}(\phi_{\ast}\beta)$
are connected by a sequence of analytic 
d-critical 
generalized flops. 
\end{thm}
\begin{proof}
Under a 3-fold flop (\ref{flop:3-fold}),
we have the commutative diagram
(see~\cite{Br1, ToBPS}):
\begin{align*}
\xymatrix{
D^b_{\le 1}(X) \ar[r]^-{\Phi} \ar[d]_-{\ch} & D^b_{\le 1}(X^{\dag}) \ar[d]^-{\ch} \\
\Gamma_{\le 1} \ar[r]_{\Phi_{\Gamma}} & \Gamma_{\le 1}^{\dag}.
}
\end{align*}
Here 
$\Phi$ is an equivalence of derived categories, 
$\Gamma_{\le 1}^{\dag}=H_2(X^{\dag}, \mathbb{Z})
\oplus \mathbb{Z}$
and 
$\Phi_{\Gamma}$ is an isomorphism 
which takes 
$(\beta, n)$ to $(\phi_{\ast}\beta, n)$. 
The above equivalence induces 
the isomorphism
\begin{align*}
\Phi_{\ast} \colon 
\Stab_{\le 1}(X) \stackrel{\cong}{\to} 
\Stab_{\le 1}(X^{\dag}).
\end{align*}
Under the above isomorphism, 
the closures of $\Phi_{\ast}U(X)$
and $U(X^{\dag})$ intersect
(see~\cite[Lemma~6.4]{TodGV}).
Let $\omega^{\dag}$ be an ample 
divisor on $X^{\dag}$. 
We take a
path 
\begin{align*}
\gamma \colon 
[0, 1]
\to 
\Phi_{\ast}\overline{U}(X) \cup 
\overline{U}(X^{\dag})
\end{align*}
such that 
$\gamma(0)=
\Phi_{\ast}\sigma_{0, \omega}$
and $\gamma(1)=\sigma_{(0, \omega^{\dag})}$. 
By perturbing $\gamma$ if necessary, 
we can assume that 
each wall-crossing of the path is given 
as in (\ref{spm}). 
Moreover the intersection of 
$\Phi_{\ast}\overline{U}(X)$
and $\overline{U}(X^{\dag})$ is 
not a wall by~\cite[Lemma~6.7]{TodGV}. 
Therefore applying Theorem~\ref{thm:Dflop:one}
at each wall, we obtain the 
result. 
\end{proof}

\section{D-critical flips of moduli spaces of stable pairs}\label{sec:wc:mmp}
In this section, we show that wall-crossing phenomena of 
Pandharipande-Thomas stable pair moduli spaces~\cite{PT}, 
studied in~\cite{BrH, Tolim, Tolim2, Tsurvey, MR2888981}, 
are described in terms of d-critical birational geometry. 
The wall-crossing phenomena here were used in \textit{loc.cit.}
to show the rationality of the generating series of 
stable pair invariants, which we will review in Appendix~\ref{sec:append}. 
\subsection{Stable pairs}
Let $X$ be a smooth projective CY 3-fold over $\mathbb{C}$. 
The notion of stable pairs by Pandharipande-Thomas is given below: 
\begin{defi}(\cite{PT})
A \textit{stable pair} 
consists of data
\begin{align*}
(F, s), \ F \in \Coh_{\le 1}(X), \ s \colon \oO_X \to F
\end{align*}
where $F$ is a pure one dimensional sheaf and $s$ is a morphism of 
coherent sheaves which is surjective in dimension one. 
\end{defi}
As in (\ref{def:Gamma}), 
we set $\Gamma_{\le 1}=H_2(X, \mathbb{Z}) \oplus \mathbb{Z}$. 
For $(\beta, n) \in \Gamma_{\le 1}$, 
let 
\begin{align}\label{moduli:PT}
P_n(X, \beta)
\end{align}
 be the moduli space of 
stable pairs $(F, s)$ 
such that $\ch(F)=(\beta, n)$. 
It is proved in~\cite{PT, HT2} 
that the moduli space (\ref{moduli:PT}) is a projective scheme 
with a symmetric perfect obstruction theory. 
Indeed the moduli space (\ref{moduli:PT}) is 
identified with the moduli space of two term complexes 
in the derived category (here $\oO_X$ is located in degree zero)
\begin{align}\label{twoterm}
I^{\bullet}=
(\cdots \to 0 \to\oO_X \stackrel{s}{\to} F \to 0 \to \cdots) \in D^b(X) 
\end{align}
satisfying a certain stability condition on it (see~\cite{PT, Tcurve1}). 
It follows that by Theorem~\ref{thm:CYdcrit}
there is a canonical d-critical structure 
on the moduli space of stable pairs (\ref{moduli:PT}).

\subsection{Weak semistable objects}
Below we study wall-crossing in
some abelian subcategory in $D^b(X)$
defined below: 
\begin{defi}
We define the subcategory $\aA_X$ in $D^b(X)$ by 
\begin{align}\label{cat:A}
\aA_X \cneq \langle \oO_X, \Coh_{\le 1}(X)[-1] \rangle_{\rm{ex}} \subset 
D^b(X).
\end{align}
Here $\langle \ast \rangle_{\rm{ex}}$ is the smallest extension 
closed subcategory containing $\ast$. 
\end{defi}
The category (\ref{cat:A}) is an abelian subcategory
of $D^b(X)$
(see~\cite[Lemma~6.2]{Tcurve1}). 
We have the Chern character map
\begin{align}\label{cl:AX}
\cl \colon K(\aA_X) \to \Gamma_{\le 1}^{\star}
\cneq \mathbb{Z} \oplus \Gamma_{\le 1}
\end{align}
sending $\oO_X$ to $(1, 0)$ and 
$F \in \Coh_{\le 1}(X)$ to 
$(0, \ch(F))$. 
We will be interested in certain rank one objects in 
$\aA_X$. 
We have the following lemma describing 
rank one objects in $\aA_X$, whose 
proof is obvious: 
\begin{lem}
An object $E \in D^b(X)$ with
$\rank(E)=1$
is an object in $\aA_X$ if and only if there 
exist distinguished triangles in $D^b(X)$
\begin{align}\label{dist:E}
\xymatrix{
0=E_0 \ar[rr] & & E_1 \ar[ld]    \ar[rr]    &  & E_2 \ar[dl] \ar[rr] &  & E_3=E \ar[dl] \\
    &  F_1 \ar@{.>}[lu]^-{[1]} &  &  F_2 \ar@{.>}[lu]^-{[1]}  & 
&  F_3 \ar@{.>}[lu]^-{[1]} &
}
\end{align}
such that 
\begin{align*}
F_1 \in \Coh_{\le 1}(X)[-1], \ 
F_2=\oO_X, \ F_3 \in \Coh_{\le 1}(X)[-1].
\end{align*}
In this case, 
the top sequence of (\ref{dist:E}) is a filtration in $\aA_X$. 
\end{lem}

\begin{rmk}\label{rmk:IA}
For a two term complex $I^{\bullet}$ in (\ref{twoterm})
associated with a stable pair, 
it
is an object in $\aA_X$ as it fits into the exact sequence in $\aA_X$
\begin{align*}
0 \to F[-1] \to I^{\bullet} \to \oO_X \to 0.
\end{align*}
Similarly the derived dual 
$\mathbb{D}(I^{\bullet})$
for $\mathbb{D}(\ast)=\dR \hH om(\ast, \oO_X)$
is an object in $\aA_X$ as it fits into the exact sequence in $\aA_X$
\begin{align*}
0 \to \oO_X \to \mathbb{D}(I^{\bullet}) \to F^{\vee}[-1] \to 0
\end{align*}
where $F^{\vee} \cneq \eE xt^2_{\oO_X}(F, \oO_X)$. 
\end{rmk}

In what follows, we fix an ample divisor 
$\omega$ on $X$.
For $t \in \mathbb{R}$, let $\mu_t^{\star}$ be the 
slope function on the abelian category $\aA_X$
defined by, for $E \in \aA_X$
\begin{align*}
\mu_t^{\star}(E) \cneq \left\{ \begin{array}{cc}
t & \mbox{ if } E \notin \Coh_{\le 1}(X)[-1] \\
\mu_{\omega}(E)=\chi(E)/\omega \cdot [E]
 & \mbox{ if } E \in \Coh_{\le 1}(X)[-1]. 
\end{array}  \right. 
\end{align*}
Here $\mu_{\omega} \cneq 
\mu_{0, \omega}$ is 
the slope function on 
one dimensional sheaves
defined as in (\ref{mu:slope})
for $B=0$. 
The above slope function $\mu_t^{\star}$
on $\aA_X$ satisfies the weak see-saw
property, and defines the 
weak stability condition on $\aA_X$
(see~\cite{Tolim2, Tsurvey}):
\begin{defi}
An object $E \in \aA_X$ is \textit{$\mu_t^{\star}$-(semi)stable}
if for any exact sequence $0 \to F \to E \to G \to 0$
in $\aA_X$, we have the inequality
\begin{align*}
\mu_t^{\star}(F)<(\le) \mu_t^{\star}(G).
\end{align*}
\end{defi}

\begin{rmk}
The above $\mu_t^{\star}$-stability condition on $\aA_X$
can 
also be 
formulated in terms of  a
Bridgeland-type weak stability condition
introduced in~\cite{Tcurve1}, 
by taking the filtration 
$\{0\} \oplus \Gamma_{\le 1} \subset \Gamma_{\le 1}^{\star}$.
See~\cite{Tsurvey} for details.
\end{rmk}

\subsection{Moduli spaces of weak semistable objects}\label{subsec:weak}
Let $\mM$ be the 
moduli stack of objects
in $D^b(X)$
considered in Subsection~\ref{subsec:moduli:general}. 
For $(\beta, n) \in \Gamma_{\le 1}$ and $t \in \mathbb{R}$, we have 
the substack
\begin{align}\label{sub:star}
\mM_t^{\star}(\beta, n) \subset \mM
\end{align}
consisting of $\mu_t^{\star}$-semistable 
objects $E \in \aA_X$ satisfying 
\begin{align*}
\cl(E)=(1, -\beta, -n) 
\in \Gamma_{\le 1}^{\star}. 
\end{align*}
The substack (\ref{sub:star}) is 
an open substack of $\mM$, which is an Artin stack 
of finite type over $\mathbb{C}$
(see~\cite[Proposition~3.17]{Tolim2}, \cite[Proposition~5.4]{Tsurvey}). 
Similarly to Theorem~\ref{thm:goodmoduli}, 
the result of~\cite{AHLH} is applied 
to the stack $\mM_t^{\star}(\beta, n)$.
So it admits a good moduli space
\begin{align}\label{good:star}
p_t \colon
\mM_t^{\star}(\beta, n) \to M_t^{\star}(\beta, n)
\end{align} 
where $M_t^{\star}(\beta, n)$
is a separated algebraic space of finite type. 
A closed point of $M_t^{\star}(\beta, n)$ corresponds to 
a $\mu_t^{\star}$-polystable object
$E \in \aA_X$ written as 
\begin{align}\label{mu:pstable}
E=\bigoplus_{i=0}^k V_i \otimes E_i, \ E_i \in \aA_X.
\end{align}
Here 
each $V_i$ is a finite dimensional vector space with 
$V_0=\mathbb{C}$, the object
$E_0 \in \aA_X$ is a rank one 
$\mu_t^{\star}$-stable object, 
and each $E_i$ for $1\le i\le k$ is 
isomorphic to $F_i[-1]$ for a $\mu_{\omega}$-semistable
sheaf $F_i \in \Coh_{\le 1}(X)$ with $\mu_{\omega}(F_i)=t$,
such that $F_i \not\cong F_j$ for $i\neq j$.
The moduli space $M_t^{\star}(\beta, n)$ is related to stable pair 
moduli spaces as follows: 
\begin{prop}\emph{(\cite[Theorem~3.21]{Tolim2}, \cite[Proposition~5.4]{Tsurvey})}\label{prop:isom}
For $\lvert t \rvert \gg 0$, 
the moduli spaces $M_t^{\star}(\beta, n)$
consists of $\mu_t^{\star}$-stable objects. 
Moreover we have the isomorphisms 
\begin{align}\label{isom:PT}
&P_n(X, \beta) \stackrel{\cong}{\to}
M_{t}^{\star}(\beta, n), \ t \gg 0, \\
\label{isom:PT2}
&P_{-n}(X, \beta) \stackrel{\cong}{\to}
M_{t}^{\star}(\beta, n), \ t \ll 0.
\end{align}
The isomorphism (\ref{isom:PT}) sends a stable pair $(F, s)$
to the two term complex (\ref{twoterm}), and the isomorphism (\ref{isom:PT2})
sends $(F, s)$ to the derived dual of (\ref{twoterm})
(see Remark~\ref{rmk:IA}). 
\end{prop}
\begin{proof}
For $\lvert t \rvert \gg 0$, it is proved in~\cite[Theorem~3.21]{Tolim2}, 
\cite[Proposition~5.4]{Tsurvey}
that the stack $\mM_t^{\star}(\beta, n)$
consists of $\mu_t^{\star}$-stable 
objects, and 
we have the isomorphism
\begin{align}\notag
&[P_{\pm n}(X, \beta)/\mathbb{C}^{\ast}] \stackrel{\cong}{\to}
\mM_{t}^{\star}(\beta, n), \ \pm t \gg 0.
\end{align}
Here $\mathbb{C}^{\ast}$ acts on $P_{\pm n}(X, \beta)$ trivially, 
and the above isomorphisms
are defined as in the statement of the proposition. 
By taking the good moduli spaces of both sides, 
we obtain the proposition. 
\end{proof}
\begin{rmk}\label{rmk:dualpair}
By the isomorphism (\ref{isom:PT2}),
an object $E \in \aA_X$ is 
isomorphic to $\mathbb{D}(I^{\bullet})$
for a stable pair (\ref{twoterm}) 
as in Remark~\ref{rmk:IA} if and only if 
$E$ fits into an exact sequence in $\aA_X$
\begin{align*}
0 \to \oO_X \to E \to F'[-1] \to 0
\end{align*} 
for some $F' \in \Coh_{\le 1}(X)$
such that $\Hom(F''[-1], E)=0$ for any 
$F'' \in \Coh_{\le 1}(X)$. 
This fact will be not used later in this paper. 
\end{rmk}

\subsection{Wall-crossing of weak semistable objects}
For $t \in \mathbb{R}$, we set 
\begin{align*}
t^{\pm} \cneq t\pm \varepsilon, \ 0<\varepsilon \ll 1.
\end{align*}
Then we have open immersions of stacks
\begin{align}\label{stack:open}
\mM_{t^{+}}^{\star}(\beta, n) \subset \mM_t^{\star}(\beta, n)
\supset \mM_{t^{-}}^{\star}(\beta, n)
\end{align}
We define $\wW \subset \mathbb{R}$
to be the subset of $t \in \mathbb{R}$ where at least 
one of the 
open immersions in (\ref{stack:open}) is not an isomorphism, i.e. 
$\wW$ is the set of walls for the 
$\mu_t^{\star}$-stability for $t \in \mathbb{R}$. 
\begin{lem}\label{lem:Wfinite}
The subset $\wW \subset \mathbb{R}$ is a finite set.
\end{lem}
\begin{proof}
An element $t \in \wW$ is of the form
\begin{align*}
t=\frac{n'}{\omega \cdot \beta'}
\end{align*}
for $0<\beta'<\beta$ and $n' \in \mathbb{Z}$. 
In particular, $\wW$ is locally finite. 
It is also bounded by 
Proposition~\ref{prop:isom}, 
hence 
$\wW$ is a finite set. 
\end{proof}
We write $\wW$ as 
\begin{align}\label{W:wall}
\wW=\{\infty=t_0> t_1>t_2> \cdots >t_l >t_{l+1}=-\infty:
 t_1, \cdots, t_l \in \mathbb{R} \}. 
\end{align}
Note that each $\mM_t^{\star}(\beta, n)$ is constant 
if $t$ lies in a 
connected component of $\mathbb{R} \setminus \wW$, 
but may change if $t$ crosses one of $t_i$. 
For $t \in \wW$, 
the open immersions (\ref{stack:open})
induce the 
diagram of good moduli spaces
\begin{align}\label{dia:Mstar}
\xymatrix{
M_{t^{+}}^{\star}(\beta, n) \ar[rd]_-{q_{M^{\star}}^{+}} & & M_{t^{-}}^{\star}(\beta, n) 
\ar[ld]^-{q_{M^{\star}}^{-}} \\
& M_{t}^{\star}(\beta, n). &
}
\end{align}
Note that $M_{t^{\pm}}(\beta, n)$ consists of 
$\mu_{t^{\pm}}$-stable objects, hence they 
admit d-critical structures by Theorem~\ref{thm:CYdcrit}. 
Below, we investigate d-critical 
birational geometry of the diagram (\ref{dia:Mstar}). 

Let us take a point 
$p \in M_{t}^{\star}(\beta, n)$ in the diagram (\ref{dia:Mstar}), 
and suppose that it corresponds 
to a $\mu_t^{\star}$-polystable object (\ref{mu:pstable}).
Let 
$Q^{\star}=Q_{E_{\bullet}^{\star}}$
be the Ext-quiver 
associated with the collection
\begin{align}\label{collect:star}
E_{\bullet}^{\star}=
(E_0, E_1, \ldots, E_k).
\end{align}
On the other hand, 
let $Q=Q_{E_{\bullet}}$ be the Ext-quiver 
associated with the collection of shift of one dimensional sheaves
\begin{align}\label{collect:shift}
E_{\bullet}=(E_1, E_2, \ldots, E_k).
\end{align}
As we observed in the proof of Theorem~\ref{thm:Dflop:one}, 
the quiver $Q$ is 
symmetric. 
Then we have
\begin{align*}
V(Q^{\star})=\{0\} \sqcup V(Q)
\end{align*}
and the numbers of arrows from $0$ to 
$i \in V(Q)$, 
from $i$ to $0$, and the loops at $0$ are
\begin{align}\label{numbers:ai}
a_i \cneq \ext^1(E_0, E_i), \ 
b_i \cneq \ext^1(E_i, E_0), \ 
c \cneq \ext^1(E_0, E_0) 
\end{align}
respectively. 
Therefore $Q^{\star}$ is obtained from $Q$ as in
the construction of Subsection~\ref{subsec:extend}. 
We also have the convergent super-potential 
\begin{align*}
W^{\star} \cneq W_{E_{\bullet}^{\star}}
\in \mathbb{C}\{ Q^{\star} \}/[\mathbb{C}\{ Q^{\star} \}, \mathbb{C}\{ Q^{\star} \}]
\end{align*}
of $Q^{\star}$ associated with the collection 
(\ref{collect:star}) as in the construction of (\ref{def:WE}). 
The following lemma (which will be used in Lemma~\ref{prop:isomM})
is obvious from the constructions, 
and we leave readers for details: 
\begin{lem}\label{rmk:W:restrict}
Let $\mathbb{C}\{Q^{\star}\} \twoheadrightarrow \mathbb{C}\{Q\}$
be the linear map sending a path
$e_1 e_2 \ldots e_n$ to itself if each $e_i \in E(Q)$
and to zero otherwise. 
The above map induces the linear map
\begin{align*}
\mathbb{C}\{ Q^{\star} \}/[\mathbb{C}\{ Q^{\star} \}, \mathbb{C}\{ Q^{\star} \}]
 \twoheadrightarrow 
 \mathbb{C}\{ Q \}/[\mathbb{C}\{ Q  \}, \mathbb{C}\{ Q\}], \ 
 f \mapsto f|_Q.
 \end{align*}
Under the above map, 
we have $W^{\star}|_{Q}=W$, 
where $W=W_{E_{\bullet}}$ is the convergent 
super-potential of $Q$
associated with the collection 
(\ref{collect:shift}). 
\end{lem}

Let $\vec{m}^{\star}$ be the dimension vector of $Q^{\star}$ given by
\begin{align}\label{dimvec:m2}
(\vec{m}^{\star})_i=\dim V_i, \ 
0 \le i\le k
\end{align}
where $V_i$ is given in (\ref{mu:pstable}). 
Note that we have $\vec{m}^{\star}=\vec{0}+\vec{m}$, where 
$\vec{m}$ is the dimension vector of $Q$
given by $m_i=\dim V_i$, $1\le i\le k$. 
Let us also 
take 
\begin{align*}
\xi^{\pm} =(\xi_i^{\pm})_{0 \le i\le k}
\in \hH^{\sharp V(Q^{\star})}
\end{align*}
 as in (\ref{star:data}). 
Similarly to Theorem~\ref{thm:Equiver}, we have the local 
description of the morphisms in (\ref{dia:Mstar})
in terms of $Q^{\star}$:

\begin{thm}\label{thm:Equiver2}
For a closed point 
$p\in M_{t}^{\star}(\beta, n)$ corresponding to a 
$\mu_t^{\star}$-polystable 
object (\ref{mu:pstable}),
let $Q^{\star}=Q_{E_{\bullet}^{\star}}$ be the 
Ext-quiver
associated with $E_{\bullet}^{\star}$ 
and $W^{\star}=W_{E_{\bullet}^{\star}}$ the 
convergent super-potential of $Q^{\star}$ constructed in (\ref{def:WE}). 
Then there exist
 analytic 
open neighborhoods 
\begin{align*}
p \in T \subset M_{t}^{\star}(\beta, n), \ 
0 \in V \subset M_{Q^{\star}}(\vec{m}^{\star})
\end{align*}
where $\vec{m}^{\star}$ is the dimension vector (\ref{dimvec:m2}), 
such that we have the commutative diagram of isomorphisms 
\begin{align}\label{isom:weak}
\xymatrix{
M^{\xi^{\pm}}_{(Q^{\star}, \partial W^{\star})}(\vec{m}^{\star})|_{V} \ar[r]^-{\cong} 
\ar[d]_-{q_{(Q^{\star}, \partial W^{\star})}^{\xi^{\pm}}} &
(q_{M^{\star}}^{\pm})^{-1}(T) \ar[d]^-{q_{M^{\star}}^{\pm}} \\
M_{(Q^{\star}, \partial W^{\star})}(\vec{m}^{\star})|_{V} \ar[r]^-{\cong} & T. 
}
\end{align}
Here the left vertical arrow is given in (\ref{dia:sym:flip}), 
the right vertical arrow is given 
in (\ref{dia:Mstar}) pulled back to
$T$. 
Moreover the top isomorphism 
preserves $d$-critical structures, where 
the $d$-critical structure on the LHS is given in Lemma~\ref{MW:dcrit}
and that on the RHS is given in Theorem~\ref{thm:CYdcrit}. 
\end{thm}
\begin{proof}
The proof is almost identical to Theorem~\ref{thm:Equiver}.
The only required modification is 
the proof of the preservation of 
stability in Lemma~\ref{lem:category}, 
as we need to compare Bridgeland stability 
on $Q^{\star}$-representations given by $\xi^{\pm}$
with weak 
stability $\mu_{t^{\pm}}^{\star}$ on $\aA_X$. 
In Lemma~\ref{lem:category2} below, we
 give an ad-hoc proof for this using 
characterizations of semistable objects in both sides. 
\end{proof}
\begin{lem}\label{lem:category2}
 For the collection (\ref{collect:star}), 
let
\begin{align*}
\Phi_{E^{\star}_{\bullet}} \colon 
\modu_{\rm{nil}}\mathbb{C}[[Q^{\star}]]/(\partial W^{\star}) \stackrel{\sim}{\to}
\langle E_0, E_1, \ldots, E_k \rangle_{\rm{ex}}
\end{align*}
be the equivalence of categories
given in (\ref{equiv:NC}).
Under the above equivalence, 
a nilpotent 
$Q^{\star}$-representation
$\mathbb{V}$ is $\xi^{\pm}$-semistable 
if and only if $\Phi_{E^{\star}_{\bullet}}(\mathbb{V})$ is 
$\mu_{t^{\pm}}^{\star}$-semistable in $\aA_X$. 
\end{lem}
\begin{proof}
Let $S_i$
for $0\le i\le k$ be the 
simple $Q^{\star}$-representation
corresponding to the vertex $i \in V(Q^{\star})$. 
As in the proof of Lemma~\ref{lem:xistab}, 
a $Q^{\star}$-representation
$\mathbb{V}$ is $\xi^{+}$-semistable 
if and only if there is 
no non-trivial surjection 
$\mathbb{V} \twoheadrightarrow \mathbb{V}'$
of $Q^{\star}$-representations 
such that $\mathbb{V}'$ is 
an object in $\langle S_1, \ldots, S_k \rangle_{\rm{ex}}$. 

For an object $F \in \langle E_0, E_1, \ldots, E_k \rangle_{\rm{ex}}$, 
we claim that it is $\mu_{t_{+}}^{\star}$-semistable 
if and only if there is no non-trivial surjection 
$F \twoheadrightarrow F'$ in 
$\langle E_0, E_1, \ldots, E_k \rangle_{\rm{ex}}$ 
such that $F' \in \langle E_1, \ldots, E_k \rangle_{\rm{ex}}$. 
The only if direction is obvious from the definition of 
$\mu_{t_{+}}^{\star}$-stability. 
In order to show the if direction, 
suppose that $F$ is not $\mu_{t_{+}}^{\star}$-semistable in $\aA_X$. 
Then there exists an exact sequence 
\begin{align}\label{exact:AFB}
0 \to A \to F \to B \to 0
\end{align}
in $\aA_X$
such that $\mu_{t_{+}}^{\star}(A)>\mu_{t_{+}}^{\star}(B)$. 
By taking the limit $t_+ \to t$, we 
have $\mu_t^{\star}(A) \ge \mu_t^{\star}(B)$. 
On the other hand, since 
$F$ is $\mu_{t}^{\star}$-semistable, we have 
$\mu_{t}^{\star}(A) \le \mu_t^{\star}(B)$. 
It follows that $\mu_t^{\star}(A)=\mu_t^{\star}(B)$, 
therefore both $A$, $B$ are $\mu_t^{\star}$-semistable. 
By the uniqueness of Jordan-H$\ddot{\rm{o}}$lder factors, 
the exact sequence (\ref{exact:AFB}) is an exact sequence 
in $\langle E_0, E_1, \ldots, E_k \rangle_{\rm{ex}}$. 
Then by the inequality $\mu_{t_{+}}^{\star}(A)>\mu_{t_{+}}^{\star}(B)$, 
we have $B \in \langle E_1, \ldots, E_k \rangle_{\rm{ex}}$. 
Therefore the if direction is also proved. 

Now the lemma for the plus sign holds by the above descriptions of 
$\xi^{+}$-semistable
$Q^{\star}$-representations and $\mu_{t_{+}}^{\star}$-semistable 
objects
in $\aA_X$,  
together with the 
fact that the equivalence (\ref{Istar:eq}) sends $S_i$ to $E_i$. 
The result for the minus sign holds in a similar way. 
\end{proof}

Using Theorem~\ref{thm:Equiver2}, 
we can describe the diagram (\ref{dia:Mstar})
in terms of d-critical birational transformations:
\begin{thm}\label{thm:starflip}
For $t \in \wW$ with $t>0$, 
the diagram (\ref{dia:Mstar}) is an 
analytic d-critical generalized flip at 
any $p \in \Imm q_{M^{\star}}^-$, 
and an analytic d-critical generalized MFS 
at any $p \in M_t^{\star}(\beta, n) \setminus \Imm q_{M^{\star}}^-$. 

Moreover for each effective curve class $\beta$, there is 
$t(\beta)>0$ (which is independent of $n$)
such that if $t>t(\beta)$ 
then the diagram (\ref{dia:Mstar}) is strict. 
There is also $n(\beta)>0$ such that 
if $n>n(\beta)$, then any $t \in \wW$ satisfies $t>t(\beta)$. 
Therefore in this case, the diagram (\ref{dia:Mstar}) is always strict. 
\end{thm}
\begin{proof}
Suppose that $p \in M_{t}^{\star}(\beta, n)$ 
corresponds to  
a $\mu_t^{\star}$-polystable object (\ref{mu:pstable}), 
and take $1\le i\le k$. Let 
$a_i, b_i$ be as in (\ref{numbers:ai}).  
Since $\mu_{\omega}(E_i)=t>0$ 
and $E_i \in \Coh_{\le 1}(X)[-1]$, 
we have $\chi(E_i)<0$. 
By the Riemann-Roch theorem and 
$\hom(E_0, E_i)=\hom(E_i, E_0)=0$, we have 
\begin{align*}
0>\chi(E_i)=\chi(E_0, E_i)=-a_i+b_i.
\end{align*}
Therefore we have $a_i>b_i$, and 
the first statement follows from 
Corollary~\ref{cor:dflip} and Theorem~\ref{thm:Equiver2}. 

We show the second strictness statement. 
Again suppose that $p$ 
corresponds to a $\mu_t^{\star}$-polystable object (\ref{mu:pstable}), 
but now assume that (\ref{mu:pstable}) is not $\mu_t^{\star}$-stable, i.e. 
$k\ge 1$. 
We write 
\begin{align*}
\cl(E_0)=(1, -\beta_0, -n_0), \ \ch(E_i)=-(\beta_i, n_i)
\end{align*}
for $1\le i\le k$. 
Then we have $\beta_0 \ge 0$, $\beta_i>0$ for $1\le i\le k$, and 
\begin{align}\label{eq:beta}
&\beta_0+m_1 \beta_1+\cdots+m_k \beta_k=\beta, \\ 
\label{eq:n}
&n_0+m_1 n_1+\cdots + m_k n_k=n
\end{align}
where $m_i=\dim V_i$ for $1\le i\le k$. 
By the identity (\ref{eq:beta}), for a fixed $\beta \ge 0$ there is only 
a finite number of possibilities for $k$, $\beta_i$ and $m_i$. 
If we have $n_i \le m_i$ for some $1\le i\le k$, then we have
\begin{align*}
t=\frac{n_i}{\omega \cdot \beta_i} \le \frac{m_i}{\omega \cdot \beta_i}
\end{align*}
and the RHS is bounded above. 
Therefore there is $t(\beta)>0$ such that if $t>t(\beta)$
then $n_i>m_i$ for any $1\le i\le k$. 
But then 
\begin{align*}
-n_i=\chi(E_i)=-a_i+b_i
\end{align*}
which implies $a_i \ge n_i>m_i$. 
Therefore the 
diagram (\ref{dia:Mstar}) is strict at $p$ by 
Lemma~\ref{lem:strict}. 

Suppose that $t \in \wW$ satisfies $t \le t(\beta)$. 
Applying Lemma~\ref{lem:boundn} below
for $\beta_0$ and all $0<t \le t(\beta)$, 
we can find $n'(\beta_0) \in \mathbb{Z}$ such that 
$n_0\le n'(\beta_0)$ holds. 
By (\ref{eq:n}) and using $t=n_i/\omega \cdot \beta_i$, we 
obtain
\begin{align*}
t \ge \frac{n-n'(\beta_0)}{\sum_{i=1}^k m_i (\omega \cdot \beta_i)}.
\end{align*}
There is $n(\beta)>0$ such that for $n>n(\beta)$, 
the RHS is bigger than $t(\beta)$. 
Therefore for such $n(\beta)$, the desired statement holds. 
\end{proof}

We have used the following lemma: 
\begin{lem}\label{lem:boundn}
For each effective curve class $\beta$
and $t \in \mathbb{R}$, there is 
$n_t(\beta) >0$ such that 
$\mM_t^{\star}(\beta, n) = \emptyset$ 
for $\lvert n \rvert >  n_t(\beta)$. 
\end{lem} 
\begin{proof}
The lemma is proved in the proof of~\cite[Lemma~4.4]{Tolim2}. 
\end{proof}

In the following example, we see
that Grassmannian flips
appear as relative d-critical charts: 
\begin{exam}\label{exam:Gflip}
In Theorem~\ref{thm:starflip}, suppose that 
$p \in M_t^{\star}(\beta, n)$ corresponds to a 
$\mu_t^{\star}$-polystable object $E \in \aA_X$ of the form
\begin{align*}
E=E_0 \oplus (V \otimes F[-1])
\end{align*}
where 
$E_0 \in \aA_X$ is a rank one $\mu_t^{\star}$-stable object, 
$V$ is a finite dimensional 
vector space, $F \in \Coh_{\le 1}(X)$
satisfying $\Ext^1(F, F)=0$, e.g. 
$F=\oO_C(k)$ for a rational 
curve $\mathbb{P}^1=C \subset X$
with $N_{C/X}=\oO_{C}(-1)^{\oplus 2}$. 
In this case the Ext-quiver 
$Q^{\star}$ 
for $\{E_0, F[-1]\}$
is the same one considered in 
Example~\ref{exam:Grass}, 
and the birational map 
\begin{align*}
M_{Q^{\star}}^{\xi^{+}}(\vec{m}^{\star}) \dashrightarrow 
M_{Q^{\star}}^{\xi^-}(\vec{m}^{\star})
\end{align*}
 is a 
Grassmannian flip
as in Example~\ref{exam:Grass}. 
For example,
the wall-crossing diagrams 
in local $\mathbb{P}^1$
studied 
in~\cite{NN} are described as d-critical 
Grassmannian flips
as above. 
\end{exam}

\begin{rmk}\label{rmk:dMFS}
In Theorem~\ref{thm:starflip},
it is possible that 
the diagram (\ref{dia:Mstar})
is a strict analytic $d$-critical generalized MFS
at $p \in M_t^{\star}(\beta, n) \setminus \Imm q_{M^{\star}}^-$,
but in the notation of Theorem~\ref{thm:Equiver2}
we have $M_{Q^{\star}}^{\xi^{-}}(\vec{m}^{\star})=\emptyset$
and the
morphism
\begin{align}\label{rmk:birational}
q_{Q^{\star}}^{\xi^{+}} \colon 
M_{Q^{\star}}^{\xi^{+}}(\vec{m}^{\star})
\to M_{Q^{\star}}(\vec{m}^{\star})
\end{align}
is birational. 

Indeed suppose that there exist disjoint smooth curves
$C_1, C_2 \subset X$
such that $C_1 \cong \mathbb{P}^1$, 
and $\omega \cdot C_2=(\omega \cdot C_1) \cdot m$
for some integer $m\ge 2$.
Then 
for $t=1/(\omega \cdot C_1)$, 
$\beta=[C_1]+[C_2]$ and $n=m+1$, 
we have the point 
$p \in M_t^{\star}(\beta, n)$
corresponding to the $\mu_t^{\star}$-polystable 
object
\begin{align*}
\oO_X \oplus \oO_{C_1}[-1] \oplus \oO_{C_2}(D)[-1]
\end{align*}
where $D$ is a divisor on $C_2$
with 
$\deg D=m+g(C_2)-1$.
Suppose that $h^1(\oO_{C_2}(D)) \neq 0$.
Then at the point $p$,  
it is easy to see that 
we have the situation mentioned above. 
\end{rmk}

\subsection{PT/L correspondence via d-critical MMP}\label{subsec:PT/L}
Let $l \in \mathbb{Z}$ be the number of elements in $\wW$
given in (\ref{W:wall}). 
For each $1\le i \le l+1$, we set
\begin{align*}
M_i \cneq M_{t_i^{+}}^{\star}(\beta, n)=M_{t_{i-1}^{-}}^{\star}(\beta, n), \ 
A_i \cneq M_{t_i}^{\star}(\beta, n)
\end{align*}
with d-critical structure $s_i$ on 
$M_i$ given in Theorem~\ref{thm:CYdcrit}. 
We also define $L_n^{\pm}(X, \beta)$ to be
\begin{align*}
L_{n}^{\pm}(X, \beta) \cneq M_{\pm \varepsilon}^{\star}(\beta, n), \ 
0<\varepsilon \ll 1. 
\end{align*}
Let us take unique 
$1\le l' \le l+1$
satisfying  
$t_{l'-1}>0 \ge t_{l'}$. 
We have the following zigzag diagram 
which connects $M_1 \cong P_n(X, \beta)$ and 
$M_{l'}=L_n^{+}(X, \beta)$:

\begin{align}\label{zigzag}
\xymatrix{
M_1 \ar[rd]_-{\pi_1^+}  &  &  M_2 \ar[ld]^-{\pi_1^-} \ar[rd]_-{\pi_2^+} &  & \cdots \ar[ld] \ar[rd] &
  &  M_{l'} \ar[ld]^-{\pi_{l'-1}^-}
\\
& A_1  & & A_2 & & A_{l'-1}  &
}
\end{align} 
The wall-crossing diagrams at $t\le 0$ are given by the following 
zigzag diagram, 
which connects $M_{l+1} \cong P_{-n}(X, \beta)$ and $M_{l'}=L_n^+(X, \beta)$: 
\begin{align}\label{zigzag2}
\xymatrix{
M_{l+1} \ar[rd]_-{\pi_{l}^-}  &  &  M_l \ar[ld]^-{\pi_l^+} 
\ar[rd]_-{\pi_{l-1}^-} &  & \cdots \ar[ld] \ar[rd] &
  &  M_{l'} \ar[ld]^-{\pi_{l'}^+}
\\
& A_{l}  & & A_{l-1} & & A_{l'}  &
}
\end{align}

\begin{cor}\label{cor:zigzag}
The diagrams (\ref{zigzag}), (\ref{zigzag2}) are 
d-critical MMP. 
In particular, we have the inequalities of
virtual canonical line bundles: 
\begin{align*}
&(M_1, s_1) \ge_K (M_2, s_2) \ge_K \cdots \ge_K (M_{l'}, s_{l'}), \\
&(M_{l+1}, s_{l+1}) \ge_K (M_l, s_l) \ge_K \cdots \ge_K (M_{l'}, s_{l'}). 
\end{align*}
Moreover for each effective curve class $\beta$, there 
is $n(\beta)>0$ such that 
the diagrams (\ref{zigzag}), (\ref{zigzag2}) are strict
and $M_{l'}=\emptyset$
if  
$\lvert n \rvert >n(\beta)$ holds.
In this case, the morphisms
\begin{align}\label{mor:dMFS}
\pi_{l'-1}^+ \colon M_{l'-1} \to A_{l'-1}, \ 
\pi_{l'}^- \colon M_{l'+1} \to A_{l'}
\end{align}
are $d$-critical generalized MFS which are strict 
at any point in $A_{l'-1}$, $A_{l'}$ respectively. 
\end{cor}
\begin{proof}
The inequalities of
virtual canonical line bundles
 for the diagram (\ref{zigzag}) is immediate from 
Theorem~\ref{thm:starflip}. 
The same argument also applies to the diagram (\ref{zigzag2}). 

If we take $n(\beta)>0$ enough large, then 
the diagrams (\ref{zigzag}), (\ref{zigzag2}) 
are strict as in the argument of Theorem~\ref{thm:starflip}. 
Moreover we have $M_{l'}=\emptyset$ by Lemma~\ref{lem:boundn}. 
Then any point in $A_{l'-1}$, $A_{l'}$ does not correspond to 
$\mu_{t_{l'-1}}^{\star}$-stable object, 
$\mu_{t_{l'}}^{\star}$-stable object, 
respectively (as otherwise $M_{l'} \neq \emptyset$).
Then as in the proof of Theorem~\ref{thm:starflip}, 
the morphisms in (\ref{mor:dMFS}) are 
strict at any point in $A_{l'-1}$, $A_{l'}$ respectively. 
\end{proof}

If $0 \in \wW$,
i.e. $t_{l'}=0$, the wall-crossing 
at $t=0$ is given by the diagram
\begin{align}\label{wall:0}
\xymatrix{
M_{l'}=L_n^{+}(X, \beta) \ar[rd] &  & L_n^{-}(X, \beta) =
M_{l'+1} \ar[ld] \\
& M_{t=0}^{\star}(\beta, n). &
}
\end{align}
\begin{cor}\label{cor:star:dflop}
The diagram (\ref{wall:0}) is an analytic d-critical 
generalized flop.
In particular, we have 
\begin{align*}
L_n^{+}(X, \beta)=_K L_n^{-}(X, \beta).
\end{align*} 
\end{cor}
\begin{proof}
For a point $p \in M_{t=0}^{\star}(\beta, n)$, 
let $a_i$, $b_i$ be as in (\ref{numbers:ai}). 
Then similarly to the proof of Theorem~\ref{thm:starflip}, 
we have $a_i=b_i$. 
This implies that the Ext-quiver $Q^{\ast}$ associated with
the collection (\ref{collect:star}) is symmetric. 
Therefore similarly to Theorem~\ref{thm:dflop:one}, the 
diagram (\ref{wall:0}) is an analytic d-critical 
generalized flop.
\end{proof} 

In the next subsection, we discuss the case 
that $\beta$ is irreducible in detail. 
Here we give some 
explicit examples for non-irreducible curve classes, 
discussed in~\cite[Section~5]{Tolim}. 

\begin{exam}\label{exam:non-irre1}
Let $X \to Y$ be a birational contraction 
with exceptional locus $C=C_1 \cup C_2$, 
where each $C_i$ is isomorphic to $\mathbb{P}^1$, 
$C_1 \cap C_2=\{p\}$
and $N_{C_i/X}=\oO_{C_i}(-1)^{\oplus 2}$. 
We set $d_i \cneq C_i \cdot \omega$ and 
assume $d_1>d_2>0$. 
Let us consider the diagram (\ref{zigzag})
in the case $(\beta, n)=([C], 2)$. In this case, 
we have two walls 
\begin{align*}
\wW=\left\{
\infty>t_1=\frac{1}{d_1} >t_2=\frac{2}{d_1+d_2} >0
\right\}. 
\end{align*}
The reduced part of the diagram (\ref{zigzag})
becomes (see~\cite[Section~5.2]{Tolim})
\begin{align}\label{zigzag:ex1}
\xymatrix{
C \ar[rd]_-{\pi_1^+}  &  &  C_1 \ar[ld]^-{\pi_1^-} \ar[rd]_-{\pi_2^+} &  & 
\emptyset \ar[ld]^-{\pi_2^-}  &  
\\
& C_1  & & \mathrm{pt} & & 
}
\end{align} 
The map $\pi_1^+$ contracts 
to $C_2 \subset C$ to the point 
$p \in C_1$
and $\pi_1^-=\id$. 
The point $p$ corresponds to 
the $\mu_{t_1}^{\star}$-polystable 
object $I_{C_1} \oplus \oO_{C_2}[-1]$, 
where $I_{C_1}$ is the ideal sheaf of $C_1$. 
The Ext-quiver $Q^{\star}$ 
of $\{I_{C_1}, \oO_{C_2}[-1]\}$
is of the following form
\begin{align*}
Q^{\star}=\left(\xymatrix{
0 \ar@<0.5ex>[r] \ar@<1.5ex>[r]& 1 \ar@<0.5ex>[l]
}\right). 
\end{align*}
Therefore locally at $p$, the 
$\pi_1^{\pm}$-relative d-critical charts 
are given by a diagram of the form
\begin{align*}
\xymatrix{
\widehat{\mathbb{C}}^2 \ar[r]^-{f^{+}} \ar[rd]_-{w^+} 
& \mathbb{C}^2 \ar[d]^-{g} &
\ar[l]_-{\id} \mathbb{C}^2 \ar[ld]^-{w^-}\\
& \mathbb{C} &
}
\end{align*}
where $f^{\dag}$ is the blow-up 
at the origin. 
It seems likely that 
$g(u, v)=u^2$, and the scheme structure 
of $P_2(X, \beta)$ at $p \in C=P_2^{\rm{red}}(X, \beta)$
is given by the critical locus of 
$\mathbb{C}^2 \to \mathbb{C}$, 
$(x, y) \mapsto x^2 y^2$. 
In particular, the left diagram of (\ref{zigzag:ex1})
is an analytic d-critical divisorial contraction. 
Similarly the right diagram of (\ref{zigzag:ex1})
is a d-critical MFS. 
\end{exam}

\begin{exam}\label{exam:non-irre2}
Let $X \to Y$ be a birational contraction 
with $C \cong \mathbb{P}^1$, 
$N_{C/X}=\oO_{C}(-1)^{\oplus 2}$. 
We set $d \cneq C \cdot \omega$.
Let us consider the diagram (\ref{zigzag})
in the case $(\beta, n)=(2[C], 4)$. In this case, 
we have two walls 
\begin{align*}
\wW=\left\{
\infty>t_1=\frac{3}{d} >t_2=\frac{2}{d} >0
\right\}. 
\end{align*}
The reduced part of the diagram (\ref{zigzag})
becomes (see~\cite[Section~5.3]{Tolim}
and~\cite[Section~4.1]{PT}
for $P_4^{\rm{red}}(X, \beta)=\mathbb{P}^3$)
\begin{align}\label{zigzag:ex2}
\xymatrix{
\mathbb{P}^3
 \ar[rd]_-{\pi_1^+}  &  &  \mathrm{pt}
 \ar[ld]^-{\pi_1^-} \ar[rd]_-{\pi_2^+} &  & 
\emptyset \ar[ld]^-{\pi_2^-}  &  
\\
& \mathrm{pt}  & & \mathrm{pt} & & 
}
\end{align} 
The map $\pi_1^{+}$ contracts 
$\mathbb{P}^3$ to a point, corresponding 
to the $\mu_{t_1}^{\star}$-polystable 
object $I_C \oplus \oO_C(2)[-1]$. 
The Ext-quiver $Q^{\star}$ 
of $\{I_C, \oO_C(2)[-1]\}$
is of the following form
\begin{align*}
Q^{\star}=\left(\xymatrix{
0 \ar@<-0.5ex>[r] \ar@<0.5ex>[r] \ar@<1.5ex>[r] \ar@<2.5ex>[r] 
& 1 \ar@<1.5ex>[l]
}\right). 
\end{align*}
Therefore the 
$\pi_1^{\pm}$-relative d-critical charts 
are given by a diagram of the form
\begin{align*}
\xymatrix{
\widehat{\mathbb{C}}^4 \ar[r]^-{f^{+}} \ar[rd]_-{w^+} 
& \mathbb{C}^4 \ar[d]^-{g} &
\ar[l]_-{\id} \mathbb{C}^4 \ar[ld]^-{w^-}\\
& \mathbb{C} &
}
\end{align*}
where $f^{\dag}$ is the blow-up 
at the origin. 
It seems likely that 
$g(u, v, t, s)=uv+ts$
and the scheme structure 
on $P_4(X, \beta)$ 
is non-reduced 
along the quadric in 
$\mathbb{P}^3=P_4^{\rm{red}}(X, \beta)$.
In particular, the left diagram of (\ref{zigzag:ex2})
is an analytic d-critical divisorial contraction. 
Similarly the right diagram of (\ref{zigzag:ex2})
is a d-critical MFS. 
\end{exam}

\subsection{The case of irreducible curve class}
Suppose that $\beta$ is an irreducible curve class, i.e. 
it is not written as $\beta_1+\beta_2$ for 
effective curve classes $\beta_i>0$. 
Let $\mM_n(X, \beta)$ be the moduli stack of 
$\mu_{\omega}$-semistable one dimensional 
sheaves $F$ on $X$ with 
$\ch(F)=(\beta, n)$, and 
\begin{align}\label{good:M}
\mM_n(X, \beta) \to M_n(X, \beta)
\end{align}
its good moduli space
for a projective scheme $M_n(X, \beta)$. 
\begin{rmk}\label{rmk:irreducible}
When $\beta$ is irreducible, 
a one dimensional sheaf $F$
with $[F]=\beta$ is 
$\mu_{\omega}$-semistable 
if and only if it is a pure one dimensional 
sheaf. 
In particular 
any point in $\mM_n(X, \beta)$ corresponds to a 
$\mu_{\omega}$-stable sheaf, 
and the morphism (\ref{good:M}) 
is a $\mathbb{C}^{\ast}$-gerbe.  
\end{rmk}
As $\beta$ is irreducible, 
by Remark~\ref{rmk:irreducible}
we have the following diagram:
\begin{align}\label{dia:Pflip2}
\xymatrix{
P_n(X, \beta)  \ar[rd]_-{q_{P}^{+}}
 & & 
P_{-n}(X, \beta) 
\ar[ld]^-{q_{P}^{-}} \\
& M_n(X, \beta). &
}
\end{align}
Here the morphisms 
$q_P^{\pm}$ are given by
\begin{align*}
q_P^{+}(F, s)=F, \ 
q_P^{-}(F', s')=\eE xt^2_{\oO_X}(F', \oO_X). 
\end{align*}
The diagram (\ref{dia:Pflip2}) appeared in~\cite{PT3} 
to show the BPS rationality of the generating series of 
stable pair invariants with irreducible curve classes. 
The above diagram is indeed a special case of the wall-crossing 
in the previous subsection, and we have the following: 
\begin{thm}\label{thm:irreducible}
(i) Suppose that $n>0$. 
Then at a point $p=[F] \in M_n(X, \beta)$, 
the diagram (\ref{dia:Pflip2})
is 
\begin{align}\label{PT:irreducible}
\left\{ \begin{array}{ll}
\mbox{analytic d-critical flip} & \mbox{if } h^1(F)>1 \\
\mbox{analytic d-critical divisorial contraction} & \mbox{if } h^1(F)=1 \\
\mbox{analytic d-critical MFS} & \mbox{if } h^1(F)=0
\end{array}  \right. 
\end{align}

(ii) Suppose that $n=0$. Then at a point $p=[F] \in M_n(X, \beta)$, 
the diagram (\ref{dia:Pflip2}) is 
\begin{align*}
\left\{ \begin{array}{ll}
\mbox{analytic d-critical flop} & \mbox{if } h^1(F)>1 \\
\mbox{isomorphisms}  & \mbox{if } h^1(F)=1 \\
\mbox{empty sets}  & \mbox{if } h^1(F)=0
\end{array}  \right. 
\end{align*}
\end{thm}
\begin{proof}
If $\beta$ is irreducible, 
there is only one wall (see~\cite[Subsection~5.1]{Tolim}):
\begin{align*}
\wW=\left\{ t_1=\frac{n}{\omega \cdot \beta} \right\}.
\end{align*}
As in the previous subsection, 
for $(\beta, n) \in \Gamma_{\le 1}$ we have the wall-crossing 
diagram
\begin{align}\label{dia:Pflip}
\xymatrix{
M_{t_{1}^+}^{\star}(\beta, n) \ar[rd]_-{q_{M^{\star}}^{+}}
 & & M_{t_{1}^-}^{\star}(\beta, n) 
\ar[ld]^-{q_{M^{\star}}^{-}} \\
& M_{t_1}^{\star}(\beta, n). &
}
\end{align}
The algebraic space
$M_{t_1}^{\star}(\beta, n)$
 parametrizes $\mu_{t_1}^{\star}$-polystable 
objects $E$ of the form
\begin{align}\label{polyE01}
E=E_0  \oplus E_1, \ 
E_0=\oO_X, \ E_1=F[-1]
\end{align}
where $F$ is a pure one dimensional 
sheaf satisfying $\ch(F)=(\beta, n)$
(see Remark~\ref{rmk:irreducible}). 
Let us take 
$p \in M_{t_1}^{\star}(\beta, n)$ 
which 
corresponds to a $\mu_t^{\star}$-polystable object (\ref{polyE01}). 
The Ext-quiver $Q^{\star}$ associated with the collection 
$\{E_0, E_1\}$
has two vertices $\{0, 1\}$, 
the number of arrows from $0$ to $1$, $1$ to $0$ and 
the loops at $1$, $0$ are
\begin{align*}
a \cneq h^0(F), \ b \cneq h^1(F), \ c \cneq 
\ext^1(F, F), \ 0=\ext^1(\oO_X, \oO_X)
\end{align*}
respectively. 
Note that $a-b=n$
by the Riemann-Roch theorem. 
Let us set 
\begin{align*}
V^+=H^0(F), \ 
V^-=H^1(F)^{\vee}, \ U=\Ext^1(F, F).
\end{align*}
Let $\vec{m}^{\star}=(1, 1)$ be the dimension vector of 
$Q^{\star}$.  
Similarly to the argument in Example~\ref{subsec:dflop}, 
we have
\begin{align*}
&M_{Q^{\star}}^{\xi^+}(\vec{m}^{\star})=
\mathrm{Tot}_{\mathbb{P}(V^+)}
(\oO_{\mathbb{P}(V^+)}(-1) \otimes
V^-) \times U, \\
&M_{Q^{\star}}^{\xi^-}(\vec{m}^{\star})=
\mathrm{Tot}_{\mathbb{P}(V^-)}
(\oO_{\mathbb{P}(V^-)}(-1) \otimes
V^+) \times U.
\end{align*}
Therefore for $n>0$, the diagram
\begin{align}\notag
\xymatrix{
M_{Q^{\star}}^{\xi^{+}}(\vec{m}^{\star})
 \ar[rd]_-{q_{Q^{\star}}^{\xi^{+}}}
 & & M_{Q^{\star}}^{\xi^{-}}(\vec{m}^{\star}) 
\ar[ld]^-{q_{Q^{\star}}^{\xi^{-}}} \\
& M_{Q^{\star}}(\vec{m}^{\star}) &
} 
\end{align}
is a standard toric flip if $b>1$, 
a divisorial contraction (indeed a blow-up of 
a smooth variety at a smooth center) if $b=1$, 
and 
MFS if $b=0$
(see~Example~\ref{exam:toric}). 
By Theorem~\ref{thm:Equiver2}, 
it follows that 
the diagram (\ref{dia:Pflip})
satisfies the condition (\ref{PT:irreducible}). 

On the other hand, we have 
the isomorphisms 
by Proposition~\ref{prop:isom}
\begin{align*}
P_{\pm n}(X, \beta) \stackrel{\cong}{\to}
M_{t_1^{\pm}}^{\star}(\beta, n). 
\end{align*}
We also have
the morphism of stacks
\begin{align*}
\mM_n(X, \beta) \to \mM_{t_1}^{\star}(\beta, n)
\end{align*}
sending a flat family
of one dimensional sheaves $\fF$ on $X \times S$
over a $\mathbb{C}$-scheme $S$ to 
the family of $\mu_t^{\star}$-semistable objects
$\oO_{X \times S} \oplus \fF[-1]$ over $S$. 
By the universality of good moduli spaces, we have the induced 
morphism
\begin{align}\label{induced:M}
\gamma \colon 
M_n(X, \beta) \to M_{t_1}^{\star}(\beta, n). 
\end{align}
The above morphism $\gamma$ is bijective on closed points, 
by the description of $\mu_t^{\star}$-polystable objects
in (\ref{polyE01}). 
We can also show that $\gamma$ is 
a closed immersion (see Lemma~\ref{prop:isomM} below), 
so $M_{t_1}^{\star}(\beta, n)$ is a nilpotent thickening 
of $M_n(X, \beta)$. 
Therefore 
by comparing the diagram (\ref{dia:Pflip2})
with (\ref{dia:Pflip}), 
the result for $n>0$ follows. 
The result for $n=0$ also holds by
a similar argument.  
\end{proof}

We have used the following lemma: 
\begin{lem}\label{prop:isomM}
The morphism (\ref{induced:M}) is a closed immersion. 
\end{lem}
\begin{proof}
It is enough to show the claim 
analytic locally at any point $p \in M_{t_1}^{\star}(\beta, n)$. 
Let $p$ corresponds to a $\mu_{t_1}^{\star}$-polystable 
object (\ref{polyE01}), $Q^{\star}$ be the
Ext-quiver associated with $\{E_0, E_1\}=\{\oO_X, F[-1]\}$, 
where $F$ is a pure one dimensional sheaf with 
$[F]=\beta$.
Let $Q$ be the sub Ext-quiver 
of $Q^{\star}$
associated with $\{E_1\}=\{F[-1]\}$. 
Let 
\begin{align*}
W^{\star} \in \mathbb{C}\{Q^{\star}\}, \ 
W \in \mathbb{C}\{Q\}
\end{align*}
be the convergent super-potentials defined 
as in (\ref{def:WE})
for the collection $\{E_0, E_1\}$, $\{E_1\}$ respectively. 
Let us take an analytic open neighborhood
$0 \in V \subset M_{Q^{\star}}(\vec{m}^{\star})$
for $\vec{m}^{\star}=(1, 1)$. 
As in (\ref{tr:W}), we have the 
$G=(\mathbb{C}^{\ast})^2$-invariant
analytic function 
\begin{align}\label{trW:star}
\tr W^{\star} \colon 
\pi_{Q^{\star}}^{-1}(V) \to \mathbb{C}
\end{align}
where $\pi_{Q^{\star}} \colon \mathrm{Rep}_{Q^{\star}}(\vec{m}^{\star})
\to M_{Q^{\star}}(\vec{m}^{\star})$ is the quotient morphism. 
By Theorem~\ref{thm:Equiver2}, 
the local analytic structure of $M_{t_1}^{\star}(\beta, n)$ at $p$ is 
isomorphic to the analytic closed subspace in $V$
defined by the ideal
\begin{align}\label{loc:ana}
(d \tr W^{\star})^G \subset 
(\oO_{\pi_{Q^{\star}}^{-1}(V)})^G =\oO_V. 
\end{align}
Let 
\begin{align*}
\vec{x}=(x_1, \ldots, x_a), \ \vec{y}=(y_1, \ldots, y_{b}), \ 
\vec{z}=(z_1, \ldots, z_c)
\end{align*} be the coordinates of 
$\mathbb{E}_{0, 1}$, $\mathbb{E}_{1, 0}$, $\mathbb{E}_{1, 1}$
corresponding to the basis $E_{i, j} \subset \mathbb{E}_{i, j}$
(see the notation~\ref{Eab})
respectively. 
Since $W^{\star}|_{Q}=W$ by Lemma~\ref{rmk:W:restrict}, 
the function (\ref{trW:star})
is of the following form
\begin{align}\label{trW}
&\tr W^{\star}(\vec{x}, \vec{y}, \vec{z})= \\
\notag
&\tr W(\vec{z})+\sum_{i,j}f_{ij}(\vec{z})x_i y_j+
\sum_{i, i', j, j'} f_{i i' j j'}(\vec{z})x_i x_{i'} y_{j} y_{j'}+\cdots.
\end{align}
Here $f_{\bullet}(\vec{z})$ is an analytic function on 
an open neighborhood of $0 \in \mathbb{E}_{1, 1}$. 
Since each $x_i y_j$ is $G$-invariant, the above description of 
$\tr W^{\star}$ implies that 
\begin{align*}
(d \tr W^{\star})^G \subset (d \tr W(\vec{z}), x_i y_j : 
1\le i \le a, 1\le j \le b ).
\end{align*}
Since we have 
\begin{align*}
M_{Q^{\star}}(\vec{m}^{\star})=\Spec \mathbb{C}[x_i y_j : 1 \le i \le a, 1\le j \le b] \times \mathbb{E}_{1, 1}
\end{align*}
it follows that the analytic closed 
subspace in $V$ 
defined by the ideal (\ref{loc:ana})
contains 
the analytic closed subspace defined by 
$d \tr W(\vec{z})=0$ on 
$V \cap (\{0\} \times \mathbb{E}_{1, 1})$. 
Since the latter is analytic locally isomorphic 
to $M_n(X, \beta)$ at $[F] \in M_n(X, \beta)$
(see Theorem~\ref{thm:Equiver}), the lemma holds. 
\end{proof}

In the following example, we see that the classical 
diagrams on symmetric products of curves appear as special 
cases of the diagram (\ref{dia:Pflip2}): 
\begin{exam}\label{exam:curve}
Suppose that $X$ contains 
a smooth projective curve $C$ of genus $g$,
which is super-rigid in $X$, i.e. 
$H^0(N_{C/X})=0$, and set $\beta=[C]$. 
Suppose that $C \subset X$ is the unique curve 
on $X$
with $[C]=\beta$. 
For example, we can take a local model 
(see~Remark~\ref{rmk:sym})
\begin{align}\label{loc:X}
X=\mathrm{Tot}_C(L_1 \oplus L_2)
\end{align}
for general line bundles $L_1$, $L_2$ on $C$
satisfying $\deg L_i=g-1$ and 
$L_1 \otimes L_2 \cong \omega_C$.
Here
$X$ contains $C$ as a zero-section. 
By setting $\beta=[C]$, we have 
the isomorphism
\begin{align*}
S^{n+g-1}(C) \stackrel{\cong}{\to} P_n(X, \beta), \ 
Z \mapsto (\oO_C(Z), s)
\end{align*}
where $s$ is the section of $\oO_C(Z)$ which vanishes at $Z$. 
Therefore in this case, the diagram (\ref{dia:Pflip2}) 
coincides with (\ref{dia:Pflip3}) and the statement
in Example~\ref{exam:symC} follows from Theorem~\ref{thm:irreducible}. 
\end{exam}

\begin{rmk}\label{rmk:sym}
A subtlety of using the local model (\ref{loc:X})
is that it is non-compact. 
Let us compactify $X$ to 
the $\mathbb{P}^2$-bundle $\overline{X} \to C$. 
Still $\overline{X}$ is not CY3, and 
also $H^1(\oO_{\overline{X}}) \neq 0$, 
so we need to modify the argument. 
Let $E_0=\oO_{\overline{X}}$, 
$E_1=F[-1]$ where $F$ is a 
line bundle on $C$ pushed forward to $\overline{X}$
by the zero section of $X \to C$. 
Then for $E=E_0 \oplus E_1$, 
we replace
the $A_{\infty}$-structure on 
$\Ext^{\ast}(E, E)$
with the $L_{\infty}$-structure 
on its traceless part 
$\Ext^{\ast}(E, E)_0$. 
Then the latter is cyclic (though the former is not), 
and the argument similar to Theorem~\ref{thm:Equiver2}
shows that the diagram (\ref{dia:Pflip3})
satisfies the desired property in Example~\ref{exam:symC}. 
\end{rmk}

\appendix
\addcontentsline{toc}{section}{Appendices}

\section{Review of Bridgeland stability conditions}\label{sec:Bridgeland}
Here we recall basic definitions on Bridgeland stability 
conditions on triangulated categories~\cite{Brs1}. 

\subsection{Stability conditions on abelian categories}
Let $\aA$ be an abelian category, and $K(\aA)$ its 
Grothendieck group. 
\begin{defi}\label{defi:stabA}
A \textit{stability condition} on an abelian 
category $\aA$
is a group homomorphism
\begin{align*}
Z \colon K(\aA) \to \mathbb{C}
\end{align*}
satisfying the followings: 

(i) (Positivity property): For any non-zero $0 \neq E \in \aA$, 
we have 
\begin{align*}
Z(E) \in \{ z \in \mathbb{C} : \Imm z>0\} \cup \mathbb{R}_{<0}.
\end{align*}
A non-zero object $E \in \aA$ 
is called \textit{$Z$-(semi)stable} if for any 
non-zero subobject $0 \neq F \subsetneq E$ in $\aA$, 
we have the inequality in $(0, \pi]$ 
\begin{align*}
\arg Z(F)<(\le) \arg Z(E).
\end{align*}

(ii) (Harder-Narasimhan property): 
For any $E \in \aA$, there exists a filtration 
(called \textit{Harder-Narasimhan filtration})
\begin{align*}
0=E_0 \subset E_1 \subset \cdots \subset E_n=E
\end{align*}
such that 
each $F_i \cneq E_i/E_{i-1}$ is $Z$-semistable
with $\arg Z(F_i)> \arg Z(F_{i+1})$ for all 
$1\le i \le n-1$. 
\end{defi}

\subsection{Stability conditions on triangulated categories}
Let $\dD$ be a $\mathbb{C}$-linear triangulated category, 
and $K(\dD)$ its Grothendieck group. 
We fix a finitely generated abelian group $\Gamma$ together with 
a group homomorphism
\begin{align}\label{def:cl}
\cl \colon K(\dD) \to \Gamma.
\end{align}
\begin{rmk}\label{rmk:cl}
A choice of $(\Gamma, \cl)$ 
corresponds to a 
choice of a Chern character map. 
For example if $\dD=D^b(X)$ for a smooth 
projective variety $X$, 
we can take $\Gamma=\Gamma_X$ 
where $\Gamma_X$ is the image of 
the Chern character map 
$\ch \colon K(X) \to H^{2\ast}(X, \mathbb{Q})$
and $\cl=\ch$. 
\end{rmk}

\begin{defi}(\cite{Brs1})\label{def:Bstab}
A \textit{Bridgeland stability condition} on $\dD$ consists 
of data
\begin{align*}
\sigma=(Z, \aA), \ \aA \subset \dD, \ Z \colon \Gamma \to \mathbb{C}.
\end{align*}
Here $\aA$ is the heart of a 
bounded t-structure on $\dD$, 
$Z$ is a group homomorphism (called
\textit{central charge})
such that $Z \circ \cl$ is a stability condition on 
$\aA$ as in Definition~\ref{defi:stabA}. 
An object $E \in \dD$ is called
\textit{$\sigma$-(semi)stable} 
if $E[k] \in \aA$ for some $k \in \mathbb{Z}$ and it 
is $Z$-(semi)stable in $\aA$. 
\end{defi}

\subsection{The space of Bridgeland stability conditions}\label{subsec:space}
Let $\Stab_{\Gamma}(\dD)$ be 
the set of Bridgeland stability conditions 
on $\dD$ with respect to the group homomorphism 
(\ref{def:cl})
satisfying the 
following condition (called \textit{support property}): 

\begin{align*}
\mathrm{sup}\left\{ \frac{\lVert \cl(E) \rVert}{\lvert Z(E) \rvert}
: E \mbox{ is } \sigma\mbox{-semistable} \right\}<\infty. 
\end{align*}
Here $\lVert \ast \rVert$ is a fixed norm on 
$\Gamma_{\mathbb{R}}$.
The following is the main result in~\cite{Brs1}:
\begin{thm}\emph{(\cite[Theorem~1.2]{Brs1})}\label{thm:Stab}
The set $\Stab_{\Gamma}(\dD)$ has a structure of a complex manifold such that 
the map
\begin{align*}
\Stab_{\Gamma}(\dD) \to \Hom(\Gamma, \mathbb{C})
\end{align*}
sending $(Z, \aA)$ to $Z$ is a local isomorphism. 
\end{thm}
Let $X$ be a smooth projective variety 
and take 
$\dD=D^b(X)$. 
By setting $\Gamma$ to be the 
image of the Chern character map 
as in Remark~\ref{rmk:cl}, 
we have the complex manifold
\begin{align*}
\Stab(X) \cneq \Stab_{\Gamma_X}(D^b(X)).
\end{align*}

\begin{rmk}\label{rmk:Caction}
Note that we have the natural $\mathbb{C}^{\ast}$-action 
on $\Hom(\Gamma, \mathbb{C})$
by the multiplication. 
This action lifts to a $\mathbb{C}$-action on 
$\Stab_{\Gamma}(\dD)$ via the universal 
cover $\mathbb{C} \to \mathbb{C}^{\ast}$, which 
does not change semistable objects. 
See~\cite[Section~3.3]{Brs6}. 
\end{rmk}

\section{Other examples}\label{sec:other}
In this section, we discuss some other examples of 
wall-crossing in CY 3-folds in terms of d-critical 
birational geometry. We just describe the results without 
details, since the arguments are similar to Theorem~\ref{thm:intro:main}.
\subsection{DT/PT correspondence}\label{subsec:DT/PT0}
Let $X$ be a smooth projective CY 3-fold. 
For $(\beta, n) \in \Gamma_{\le 1}$, let 
\begin{align}\label{moduli:I}
I_n(X, \beta)
\end{align}
be the moduli space of 
subschemes $C \subset X$
such that $[C]=\beta$ and $\chi(\oO_C)=n$. 
The moduli space (\ref{moduli:I}) is identified with 
the moduli space of rank one torsion free sheaves 
$I$ with Chern character $(1, 0, -\beta, -n)$. 
In particular, it has a canonical d-critical structure. 

The moduli spaces (\ref{moduli:I}) and (\ref{moduli:PT})
are related by wall-crossing phenomena with respect to 
certain Bridgeland-type 
weak stability conditions 
in the derived category (see~\cite{Tcurve1}).
The above wall-crossing is relevant 
in showing the DT/PT correspondence 
conjecture~\cite{PT} (see Subsection~\ref{subsec:DT/PT}).  
As in Section~\ref{sec:wc:mmp}, we have the diagram
\begin{align}\label{DT/PT}
\xymatrix{
I_n(X, \beta) \ar[rd]_-{q_I} &  &  P_n(X, \beta) \ar[ld]^-{q_P} \\
& T_n(X, \beta) &}
\end{align}
where $T_n(X, \beta)$ is an algebraic space which 
parametrizes 
objects of the form
\begin{align*}
I_C \oplus \left(\bigoplus_{i=1}^k V_i \otimes \oO_{x_i}[-1] \right)
\end{align*}
where $I_C$ is an ideal sheaf of a pure one 
dimensional subscheme $C \subset X$, 
and $x_i \neq x_j$ for $i \neq j$, 
satisfying
\begin{align*}
\chi(\oO_C)+\sum_{i=1}^k \dim V_i =n.
\end{align*}
We have the following: 
\begin{thm}\label{thm:DT/PT}
The diagram (\ref{DT/PT}) is an analytic 
d-critical generalized flip 
at any point in $\Imm q_P$, 
an analytic $d$-critical generalized MFS
at any point in $T_n(X, \beta) \setminus \Imm q_P$. 
In particular, we have 
\begin{align*}
I_n(X, \beta) \ge_K P_n(X, \beta).
\end{align*}
\end{thm}
\begin{proof}
The proof is the same as in Theorem~\ref{thm:starflip}. 
\end{proof}

Here is an example of DT/PT correspondence for 
a fixed smooth curve:
\begin{exam}\label{exam:DTC}
In the situation of Example~\ref{exam:curve}, 
we have
\begin{align*}
\mathrm{Quot}(I_C, n+g-1) \stackrel{\cong}{\to}
I_n(X, \beta).
\end{align*}
Here the LHS is the Quot scheme parameterizing 
quotients $I_C \twoheadrightarrow Q$ 
such that $Q$ is a zero dimensional sheaf with 
length $n+g-1$, and the above isomorphism is given 
by taking the kernel of $I_C \twoheadrightarrow Q$. 
By setting $m=n+g-1$, the diagram (\ref{DT/PT}) is 
\begin{align}\notag
\xymatrix{
\mathrm{Quot}(I_C, m) \ar[rd]_-{q_I} &  & S^m(C) \ar[ld]^-{q_P} \\
& S^m(X). &}
\end{align}
Here $q_I$ sends $I_C \twoheadrightarrow Q$ to the support of $Q$, 
and $q_P$ is induced by $C \subset X$. 
Note that $S^m(C)$ is always smooth while 
$\mathrm{Quot}(I_C, m)$ can be singular. 
By Theorem~\ref{thm:DT/PT}, 
the above diagram is an analytic d-critical generalized 
flip at any point in $\Imm q_P$, 
an analytic d-critical generalized MFS
at any point in $S^m(X) \setminus \Imm q_P$. 
\end{exam}

\subsection{Wall-crossing in local $K3$ surfaces}
Let $S$ be a smooth projective $K3$ surface over $\mathbb{C}$, and 
$\Gamma_S$ its Mukai lattice: 
\begin{align*}
\Gamma_S \cneq \mathbb{Z} \oplus \mathrm{NS}(S) \oplus \mathbb{Z}.
\end{align*}
Let $v(-)$ be the Mukai vector 
\begin{align}\label{Mvector}
v \colon K(S) \to \Gamma_S, \ 
E \mapsto \ch(E) \cdot \sqrt{\td}_X. 
\end{align}
Then we have the space of Bridgeland 
stability conditions on 
$D^b(S)$
with respect to 
group homomorphism (\ref{Mvector}), 
denoted by $\Stab(S)$. 
The structure of the space $\Stab(S)$
is studied in~\cite{Brs2}, and 
moduli spaces of Bridgealnd stable 
objects on $S$ are studied in~\cite{MR3194493}. 

Let $X$ be the non-compact CY 3-fold 
defined by 
\begin{align*}
X \cneq \mathrm{Tot}_S(K_S)=S \times \mathbb{C}. 
\end{align*}
We consider moduli spaces of semistable objects 
on the triangulated category
\begin{align*}
D^b_c(X) \cneq \{E \in D^b(X) : 
\Supp(E) \mbox{ is compact }\}. 
\end{align*}
Let $p_S \colon X \to S$ be the projection. 
We have the group homomorphism
\begin{align}\label{cl:K3}
\cl \colon K(D^b_c(X)) \to \Gamma_S, \ 
E \mapsto v(p_{S\ast}E)
\end{align}
Let $\Stab_c(X)$ be the space of 
Bridgeland stability conditions on 
$D^b_c(X)$
with respect to the group homomorphism (\ref{cl:K3}).
Then we have the isomorphism
(see~\cite[Theorem~6.5, Lemma~5.3]{Tst2})
\begin{align*}
p_{S\ast} \colon 
\Stab_{c}(X) \stackrel{\cong}{\to} \Stab(S)
\end{align*}
which is identity on central charges.
Under the above isomorphism, 
an object $E \in D_{c}^b(X)$ is 
$\sigma$-(semi)stable for $\sigma \in \Stab_{c}(X)$
if and only if $p_{S\ast}E \in D^b(S)$
 is $p_{S\ast}\sigma$-(semi)stable. 

For $\sigma=(Z, \aA) \in \Stab_c(X)$
and $v \in \Gamma_S$
with $\Im Z(v)>0$, 
let $\mM_{\sigma}(v)$ be 
the moduli stack of $\sigma$-semistable 
objects
$E \in \aA$ such that 
$\cl(E)=v$, 
and 
\begin{align*}
\mM_{\sigma}(v) \to M_{\sigma}(v)
\end{align*}
be its good moduli space. 
Note that if $v$ is primitive
and $\sigma$ is general, 
then $\mM_{\sigma}(v)$ consists of $\sigma$-stable objects
of the form $i_{c\ast}F$ for 
$c \in \mathbb{C}$, where
$F \in D^b(S)$ is $p_{S\ast}\sigma$-stable 
and $i_c \colon S \times \{c\} \hookrightarrow X$ is the 
inclusion. 
Therefore we have 
\begin{align*}
M_{\sigma}(v)=M_{p_{S\ast}\sigma}(v) \times \mathbb{C}
\end{align*}
where $M_{p_{S\ast}\sigma}(v)$ is the
moduli space of $p_{S\ast}\sigma$-stable objects
in $D^b(S)$ with Mukai vector $v$.
By~\cite{MR3194493}, the moduli space 
$M_{p_{S\ast}\sigma}(v)$
is a projective holomorphic symplectic 
manifold of dimension $2+(v, v)$, where 
$(-, -)$ is the Mukai product on $\Gamma_S$. 

Suppose that $v\in \Gamma_S$ is primitive and $\sigma \in \Stab_c(X)$ lies on a wall 
with respect to $v$. If $\sigma^{\pm} \in \Stab_c(X)$ lie on 
its adjacent chambers, we obtain the diagram
\begin{align}\label{dia:locK3}
\xymatrix{
M_{\sigma^{+}}(v) \ar[rd]_-{q_M^{+}}  & 
& M_{\sigma^{-}}(v)
 \ar[ld]^-{q_M^{-}} \\
& M_{\sigma}(v). &
}
\end{align}
By noting the Ext-quiver associated with 
any collection in $D^b(S)$ is symmetric, 
we have the following: 
\begin{thm}\label{thm:locK3}
The diagram (\ref{dia:locK3}) 
is an analytic $d$-critical generalized flop.
\end{thm}

\section{Wall-crossing formula of DT type invariants}\label{sec:append}
Here we review the 
previous 
works~\cite{BrH, Tcurve1, Tolim, Tolim2, Tsurvey, MR2888981} 
where wall-crossing phenomena in this paper 
were applied to show several properties on DT invariants. 

\subsection{Product expansion formula of PT invariants}
For a CY 3-fold $X$ and an element 
$(\beta, n) \in \Gamma_{\le 1}$, 
the moduli space of stable pairs 
$P_n(X, \beta)$
admits a zero dimensional virtual class~\cite{PT}. 
The PT invariant $P_{n, \beta}\in \mathbb{Z}$
is defined by the integration of the virtual class. 
It also coincides with weighted 
Euler characteristics 
\begin{align*}
P_{n, \beta} = \int_{[P_n(X, \beta)]} \chi_B \ de
\cneq \sum_{m\in \mathbb{Z}} m \cdot e(\chi_B^{-1}(m)). 
\end{align*}
Here $\chi_B$ is the Behrend constructible function~\cite{Beh}
 on $P_n(X, \beta)$.  
The generating series of PT invariants 
satisfies the following product expansion formula
\begin{align}\label{PT:formula}
&1+\sum_{\beta>0, n \in \mathbb{Z}}P_{n, \beta}q^n t^{\beta} \\
\notag
&=\exp \left( \sum_{\beta>0, n>0} (-1)^{n-1} n N_{n, \beta} q^n t^{\beta}
  \right) \cdot \left( \sum_{\beta>0, n \in \mathbb{Z}} 
L_{n, \beta}q^n t^{\beta} \right).
\end{align}
Here 
$N_{n, \beta}$ and $L_{n, \beta}$ are as follows: 
\begin{itemize}
\item The invariant
$N_{n, \beta} \in \mathbb{Q}$ is the 
generalized DT invariant~\cite{JS}
counting one dimensional semistable sheaves on $X$
with 
Chern character $(\beta, n) \in \Gamma_{\le 1}$.
It satisfies the symmetric property
$N_{n, \beta}=N_{-n, \beta}$ and the 
periodicity property $N_{n, \beta}=N_{n, \beta+\omega \cdot \beta}$
for any ample divisor $\omega$ on $X$. 

\item
The invariant  $L_{n, \beta} \in \mathbb{Z}$ 
is a DT type invariant counting certain 
stable objects $E \in D^b(X)$
satisfying $\ch(E)=(1, 0, -\beta, -n)$ 
(see (\ref{def:Linv}) below). 
It satisfies the symmetric property 
$L_{n, \beta}=L_{-n, \beta}$ and the vanishing 
$L_{n, \beta}=0$ for $\lvert n \rvert \gg 0$. 
\end{itemize}

The formula (\ref{PT:formula})
is proved 
using 
wall-crossing phenomena in Section~\ref{sec:wc:mmp}, 
which 
led to the proof of the rationality conjecture of 
the generating series of PT invariants~\cite{MNOP, PT}. 

\subsection{Wall-crossing formula of PT invariants}\label{subsection:PT:WCF}
Below we recall how to derive the formula (\ref{PT:formula}) from 
wall-crossing in Section~\ref{sec:wc:mmp}. 
For $t \in \mathbb{R}$, let us consider the 
moduli stack 
$\mM_t^{\star}(\beta, n)$ and its 
good moduli space $M_t^{\star}(\beta, n)$
as in Subsection~\ref{subsec:weak}. 
Let $\wW \subset \mathbb{R}$ be the 
set of walls (\ref{W:wall}), and take $t \notin \wW$. 
By integrating the Behrend function on $M_t^{\star}(\beta, n)$, we have 
the invariant 
\begin{align}\label{inv:Lt}
L_{n, \beta}^{t} \cneq \int_{M_t^{\star}(\beta, n)} \chi_B \ de \in \mathbb{Z}. 
\end{align}
By Proposition~\ref{prop:isom}, 
we have the identities
\begin{align*}
L_{n, \beta}^t=\left\{ \begin{array}{ll}
P_{n, \beta}, & t\gg 0, \\
P_{-n, \beta}, & t \ll 0.
\end{array}
\right. 
\end{align*}
For $t_i \in \wW$, 
the wall-crossing formula 
of the invariants (\ref{inv:Lt}) is given by
(see~\cite[Theorem~5.7]{Tsurvey})
\begin{align*}
&\lim_{\varepsilon \to +0}
\sum_{\beta>0, n\in \mathbb{Z}}
 L_{n, \beta}^{t_i+\varepsilon}
q^n t^{\beta} \\
&=
\exp\left( \sum_{n/\omega \cdot \beta=t_i} 
(-1)^{n-1} n N_{n, \beta} q^n t^{\beta} \right) \cdot 
\left(\lim_{\varepsilon \to +0} 
\sum_{\beta>0, n\in \mathbb{Z}}
L_{n, \beta}^{t_i-\varepsilon}
q^n t^{\beta}\right). 
\end{align*}
Applying the above formula 
from $t\to \infty$ to $t \to +0$, 
using the identity (\ref{isom:PT})
and setting 
\begin{align}\label{def:Linv}
L_{n, \beta} \cneq L_{n, \beta}^t, \ 
0<t \ll 1
\end{align}
we obtain the formula (\ref{PT:formula}). 

\subsection{DT/PT correspondence}\label{subsec:DT/PT}
Let $I_n(X, \beta)$ the moduli space defined in (\ref{moduli:I}). 
The rank one DT invariants
\begin{align*}
I_{n, \beta} \cneq \int_{I_n(X, \beta)} \chi_B \ de
\end{align*}
are related to PT invariants by the identity
\begin{align*}
\sum_{n \in \mathbb{Z}}I_{n, \beta}q^n=
\prod_{n\ge 1}(1-q^n)^{-n \cdot e(X)}
\sum_{n \in \mathbb{Z}}P_{n, \beta}q^n.
\end{align*}
The above formula, called DT/PT correspondence, 
was conjectured in~\cite{PT} and proved in~\cite{BrH, Tcurve1}, 
via wall-crossing phenomena in the derived category
discussed in Subsection~\ref{subsec:DT/PT0}. 

\bibliographystyle{amsalpha}
\bibliography{math}

Kavli Institute for the Physics and 
Mathematics of the Universe, University of Tokyo (WPI),
5-1-5 Kashiwanoha, Kashiwa, 277-8583, Japan.

\textit{E-mail address}: yukinobu.toda@ipmu.jp

\end{document}